\documentclass[11pt]{amsart}
\usepackage{amsmath}
\usepackage{amssymb}
\usepackage{amsthm}
\usepackage[english]{babel}
\usepackage{hyperref}
\usepackage{mathrsfs}
\usepackage{cite}

\usepackage{enumerate}

\newtheorem{theor}{Theorem}[section]
\newtheorem{lemma}[theor]{Lemma}

\newtheorem{rem}[theor]{Remark}

\newcommand{\hd}{\mathrm{hd}}

\newcommand{\rad}{\mathrm{rad}}
\newcommand{\Hom}{\mathrm{Hom}}
\newcommand{\End}{\mathrm{End}}

\newcommand{\res}{\mathrm{res}}
\newcommand{\Ind}{\mathrm{Ind}}
\newcommand{\Res}{\mathrm{Res}}
\newcommand{\ord}{\mathrm{ord}}
\newcommand{\JS}{\text{JS}}

\newcommand{\s}{{\sf S}}
\newcommand{\A}{{\sf A}}

\newcommand{\Md}{\!\mod}

\newcommand{\Z}{\mathbb{Z}}

\newcommand{\Q}{\mathbb{Q}}

\newcommand{\1}{\mathbf{1}}
\newcommand{\sgn}{\mathbf{\mathrm{sgn}}}

\renewcommand{\epsilon}{\varepsilon}
\renewcommand{\phi}{\varphi}
\newcommand{\xymat}{\xymatrix@R=6pt@C=10pt}

\newcommand{\la}{\lambda}
\newcommand{\be}{\beta}
\newcommand{\al}{\alpha}
\newcommand{\eps}{\epsilon}

\newcommand{\de}{\delta}
\newcommand{\si}{\sigma}

\newcommand{\da}{{\downarrow}}
\newcommand{\ua}{{\uparrow}}

\newcommand{\Mull}{{\tt M}}

\def\RP{{\mathscr {RP}}}
\def\RPar{{\mathscr {RP}}}
\def\Par{{\mathscr {P}}}
\def\Parinv{{\mathscr P}^A}
\def\h{{\mathscr {H}}}

\begin{document}

\title[Irreducible tensor products for covering groups]{Irreducible tensor products of representations of covering groups of symmetric and alternating groups}

\author{\sc Lucia Morotti}
\address{Institut f\"{u}r Algebra, Zahlentheorie und Diskrete Mathematik\\ Leibniz Universit\"{a}t Hannover\\ 30167 Hannover\\ Germany} 
\email{morotti@math.uni-hannover.de}

\thanks{
The author was supported by the DFG grant MO 3377/1-1. The author would like to thank the Isaac Newton Institute for Mathematical Sciences for support and hospitality during the programme ``Groups, representations and applications: new perspectives'' when work on this paper was undertaken. This work was supported by: EPSRC grant number EP/R014604/1.
}


\begin{abstract}
In this paper we completely classify irreducible tensor products of covering groups of symmetric and alternating groups in characteristic $\not=2$.
\end{abstract}

\maketitle

\section{Introduction}

Let $F$ be an algebraically closed filed, $G$ be a group and $V$ and $W$ be irreducible $FG$-representation. A natural question to ask is when the tensor product $V\otimes W$ is irreducible. This is always the case if $V$ or $W$ is 1-dimensional, so the interesting cases are those where neither $V$ nor $W$ is 1-dimensional but $V\otimes W$ is irreducible, in which case we say that $V\otimes W$ is a non-trivial irreducible tensor product. One motivation to this question comes from the Aschbacher-Scott classification of maximal subgroups of finite classical groups, see \cite{a,as}.

Irreducible tensor products of symmetric groups have been fully classified in \cite{bk,gj,gk,m1,z1}. For alternating groups, apart for some cases in in characteristic $2$, non-trivial tensor products have been classified in \cite{bk3,bk2,m2,m3,z1}. For covering groups of symmetric and alternating groups however only partial results are known, that is the characteristic 0 case for $\widetilde\s_n$, see \cite{b2,bk4}, as well as some reduction results obtained in \cite{kt} for $\widetilde\s_n$ and $\widetilde\A_n$ in characteristic $\geq 5$. In this paper we will consider the case where $G=\widetilde\s_n$ or $\widetilde\A_n$ is a covering group of a symmetric or alternating group and completely classify non-trivial irreducible tensor products in characteristic $\not=2$.

By definition there exists $z\in\widetilde\A_n\subseteq\widetilde\s_n$ with $z$ of order 2 and central in $\widetilde\s_n$ such that $\s_n\cong \widetilde\s_n/\langle z\rangle$ and $\A_n\cong \widetilde\A_n/\langle z\rangle$. Since $z$ is central of order 2, irreducible representations of $\widetilde\s_n$ and $\widetilde\A_n$ are of two types, depending on whether $z$ acts as $1$ or $-1$. Let $V$ be an irreducible representation of $\widetilde\s_n$ or $\widetilde\A_n$. If $z$ acts as $1$ on $V$ then $V$ may be viewed also as a representation of $\s_n$ or $\A_n$ by factoring through $\langle z\rangle$ (and $V$ is irreducible also as an $\s_n$- or $\A_n$-representation). On the other hand if $V$ is an (irreducible) representation of $\s_n$ or $\A_n$ then we may lift $V$ to an (irreducible) representation of $\widetilde\s_n$ or $\widetilde\A_n$ on which $z$ acts trivially. If on the other hand $z$ acts as $-1$ on $V$ then we say, when $p\not=2$, that $V$ is a spin representation.

Thus, for $p\not=2$, when considering tensor products $V\otimes W$ of two irreducible representations $V$ and $W$ of $\widetilde\s_n$ or $\widetilde\A_n$ three cases need to be considered: (i) neither $V$ nor $W$ is a spin representation, (ii) $V$ is not a spin representation, while $W$ is a spin representation and (iii) both $V$ and $W$ are spin representations. In case (i) $V\otimes W$ is irreducible as a $\widetilde\s_n$- or $\widetilde\A_n$-representation if and only if it is irreducible as a $\s_n$- or $\A_n$-representation, so this case is already covered by \cite{bk3,bk,bk2,gk,m2,m3,z1}. So only cases (ii) and (iii) will be considered in this paper. As can be seen from Theorems \ref{TS} and \ref{TA} irreducible tensor products of two spin representations only occur for $n$ small, however there exist infinite families of irreducible tensor products of a spin representation and a non-spin representations (see also \cite{b2,bk4,kt} for partial results).

Note that if $p=2$ then $z$ acts trivially on any irreducible representation of $\widetilde\s_n$ or $\widetilde\A_n$. So in this case classifying irreducible tensor products of $\widetilde\s_n$ or $\widetilde\A_n$ is equivalent to classifying irreducible tensor products for $\s_n$ or $\A_n$. So this case will not be considered in this paper. For $\s_n$ this problem has already been completely solved in \cite{bk,gj,gk,m1}. For $\A_n$ partial results, including a complete analysis when neither $V$ nor $W$ is basic spin, can be found in \cite{m3}. 

For $n=6$ or $7$, irreducible tensor products of representations of the triple covers can be easily classified looking at characters table using \cite{Atl}, so they will not be considered here.

It is well known (see for example \cite{JamesBook}) that, in characteristic $p$, irreducible representations of symmetric groups are indexed by $p$-regular partitions, that is partitions with no part repeated $p$ or more times. Let $\Par_p(n)$ be the set of $p$-regular partitions of $n$ and, given $\la\in\Par_p(n)$, let $D^\la$ be the corresponding irreducible representation of $\s_n$. For any $\la\in\Par_p(n)$ let $\la^\Mull\in\Par_p(n)$ be the Mullineux dual of $\la$, that is the partition with $D^{\la^\Mull}\cong D^\la\otimes\sgn$, where $\sgn$ is the sign representation of $\s_n$. A combinatorial description of the Mullineux bijection is known, see \cite{BO,FK,Mull}. For $p\geq 3$ it is also well known that $D^\la\da_{\A_n}$ is irreducible if and only $\la\not=\la^\Mull$, see \cite{f}. In this case we will write $E^\la$ for $D^\la\da_{\A_n}$. Note that $E^\la\cong E^{\la^\Mull}$. On the other hand if $\la=\la^\Mull$ we have that $D^\la\da_{\A_n}\cong E^\la_+\oplus E^\la_-$ with $E^\la_\pm$ non-isomorphic irreducible representations of $\A_n$. Further any irreducible representation of $\A_n$ is either of the form $E^\la$ or of the form $E^\la_\pm$ for some $\la\in\Par_p(n)$. As mentioned above the modules $D^\la$ (resp. $E^\la_{(\pm)}$) can also be viewed as representations of $\widetilde\s_n$ (resp. $\widetilde\A_n$).

In positive characteristic $p\geq 3$, irreducible spin representations of symmetric and alternating groups have been described in \cite{BK,BK2}. There it has been proved that if $\RP_p(n)$ is the set of $p$-restricted $p$-strict partitions of $n$, that is partitions $\la$ with $1-\de_{p\mid\la_i}\leq\la_i-\la_{i+1}\leq p-\de_{p\mid n}$, then (pairs of) spin irreducible representations of $\widetilde\s_n$ or $\widetilde\A_n$ are indexed by elements of $\RP_p(n)$ (here as in the following $\de_{p\mid m}=1$ if $p\mid m$, while $\de_{p\mid m}=0$ if $p\nmid m$). More in particular for any $\la\in\RP_p(n)$ there either exists an irreducible spin representation $D(\la,0)$ of $\widetilde\s_n$ or there exist two non-isomorphic representations $D(\la,\pm)$ of $\widetilde\s_n$. In either case, for $\eps=0$ or $\pm$ depending on $\la$, we have that $D(\la,\eps)\cong D(\la,-\eps)\otimes\sgn$, so that in the first case $D(\la,0)\da_{\tilde\A_n}\cong E(\la,+)\oplus E(\la,-)$ with $E(\la,\pm)$ non-isomorphic irreducible spin representations of $\widetilde\A_n$, while in the second case $D(\la,\pm)\da_{\tilde\A_n}\cong E(\la,0)$ with $E(\la,0)$ irreducible. Further again any spin irreducible representation of $\widetilde\s_n$ or $\widetilde\A_n$ is of one of these forms.

For $n\geq 1$ write $n=dp+e$ with $0\leq e<p$. Define $\be_n:=(p^d,e)$ if $e>0$ or $\be_n:=(p^{d-1},p-1,1)$ if $e=0$. Note that $\be_n$ is a $p$-restricted $p$-strict partition of $n$. The irreducible spin representations indexed by $\be_n$ are called basic spin modules and will play a special role in this paper. Such representations are the composition factors of the reduction modulo $p$ of basic spin modules in characteristic 0, see \cite{BK3,Wales}.

Given $\la\in\Par_p(n)$ write $\la=(a_1^{b_1},\ldots,a_h^{b_h})$ with $a_1>\ldots>a_h\geq 1$ and $b_i\geq 1$. We say that $\la$ is JS if $a_i-a_{i+1}+b_i+b_{i+1}\equiv 0\Md p$ for $1\leq i<h$. It has been proved (see \cite{JS,k2}) that $\la\in\Par_p(n)$ is JS if and only if $D^\la\da_{\s_{n-1}}$ is irreducible. For any $a\geq 1$ let $a=bp+c$ with $1\leq c\leq p$ and define $\res(a):=\min\{c-1,p-c\}$. For $\la\in\RP_p(n)$ we say that $\la=(\la_1,\ldots,\la_h)$ is JS(0) if $\la_h=1$ and $\res(\la_i)=\res(\la_{i+1}+1)$ for $1\leq i<h$. In view of \cite{BK,p} it can be checked that $\la\in\RP_p(n)$ is JS(0) if and only if $D(\la,\eps)\da_{\tilde\s_{n-1}}$ and $E(\la,\eps')\da_{\tilde\A_{n-1}}$ are both irreducible. An equivalent characterisation is also that $D(\la,0)\da_{\tilde\s_{n-1}}$ is irreducible if $\la$ indexes only one spin representation of $\widetilde\s_n$ or that $E(\la,0)\da_{\tilde\A_{n-1}}$ is irreducible if $\la$ indexes two spin representations of $\widetilde\s_n$.

Before stating our main results we list here few irreducible tensor products of representations of $\widetilde\s_n$ or $\widetilde\A_n$, which will turn out to be exactly the exceptional irreducible tensor products of representations of $\widetilde\s_n$ or $\widetilde\A_n$, see Theorems \ref{TS} and \ref{TA} and their proof in Section \ref{s7}. As will be seen in the main theorems, any other irreducible tensor product is part of an infinite family of irreducible tensor products. In rows 4 and 5, $\chi_V$ is the character of $V$, $\chi_W$ the character of $W$ and $\widetilde{(1,2,3,4,5)}$ the lift of order 5 of the 5-cycle $(1,2,3,4,5)$.

{\small
\[\begin{array}{|c|c|c|c|c|c|}
\hline
G&V&W&V\otimes W&p&\text{further assumptions}\\
\hline\hline
\rlap{$\phantom{E^{(3^3)}_\pm}$}\widetilde\s_6&D((3,2,1),\pm)&D(\be_6,\pm)&D^{(3,2,1)}&p\geq 7&\\
\hline
\rlap{$\phantom{E^{(3^3)}_\pm}$}\widetilde \A_5&E(\be_5,+)&E(\be_5,-)&E^{(4,1)}&p\not=5&\\
\hline
\rlap{$\phantom{E^{(3^3)}_\pm}$}\widetilde \A_6&E(\be_6,+)&E(\be_6,-)&E^{(5,1)}&p=3&\\
\hline
\rlap{$\phantom{\widetilde{E^{(3^3)}_\pm}}$}\widetilde \A_5&E^{(3,1^2)}_\pm&E(\be_5,\pm)&E((4,1),0)&p\not=5&\chi_V\chi_W\widetilde{(1,2,3,4,5)}=1\\
\hline
\rlap{$\phantom{\widetilde{E^{(3^3)}_\pm}}$}\widetilde \A_6&E^{(4,1^2)}_\pm&E(\be_6,\pm)&E((4,2),\pm)&p=3&\chi_V\chi_W\widetilde{(1,2,3,4,5)}=1\\
\hline
\rlap{$\phantom{\widetilde{E^{(3^3)}_\pm}}$}\widetilde \A_6&E^{(4,1^2)}_+&E^{(4,1^2)}_-&E^{(4,2)}&p=3&\\
\hline
\rlap{$\phantom{\widetilde{E^{(3^3)}_\pm}}$}\widetilde \A_9&E^{(3^3)}_\pm&E(\be_9,\pm)&E((5,3,1),\pm)&p\geq 7&\\
\hline
\end{array}\]
\vspace{1mm}
\centerline{\sc Table I}
}

In the next theorems, as well as in the remaining of the paper, if $\al$ and $\be$ are partitions, let $\al+\be:=(\al_1+\be_1,\al_2+\be_2,\ldots)$ and $\al\cup\be$ be the partition obtained by rearranging the parts of $(\al,\be)=(\al_1,\al_2,\ldots,\be_1,\be_2,\ldots)$. Further, for any partition $\al$, let $h(\al)$ be the number of non-zero parts of $\al$.

The next two theorems completely characterise irreducible tensor of representations of covering groups of symmetric and alternating groups respectively. Parts of the theorems (the classification of irreducible tensor products of two non-spin representations, the characteristic 0 case for $\tilde\s_n$ and some reduction results for the other cases) can be recovered from the papers mentioned at the beginning of the introduction, but we still state the theorems in complete form.

\begin{theor}\label{TS}
Let $p\geq 3$ and $V,W$ be irreducible $F\widetilde\s_n$-representations which are not 1-dimensional. Then $V\otimes W$ is irreducible if and only if one of the following holds up to exchange of $V$ and $W$:
\begin{enumerate}[(i)]
\item $n\not\equiv 0\Md p$, $V\in\{D^{(n-1,1)},D^{(n-1,1)^\Mull}\}$, $W\cong D(\la,\pm)$ with $\la\in\RPar_p(n)$ a $\JS(0)$-partition, in which case $V\otimes W\cong D(\nu,0)$ where $\nu=(\la\setminus A)\cup B$ with $A$ is the bottom removable node of $\la$ and $B$ is the top addable node of $\la$,

\item $n\not\equiv 0,\pm 2\Md p$ is even, $V\cong D^\la$ where $\la\in\Par_p(n)$ is a JS-partition with $\min\{h(\la),h(\la^\Mull)\}=2$, and $W$ is basic spin, in which case, assuming $h(\la)=2$, if 
$\la_1\not=\la_2$ then $V\otimes W\cong D(\be_{\la_1}+\be_{\la_2},0)$, while if 
$\la_1=\la_2$ then $V\otimes W\cong D(\be_{n/2+1}\cup\be_{n/2-1},0)$,

\item $V$ and $W$ are as in row 1 of Table I.
\end{enumerate}
\end{theor}

\begin{theor}\label{TA}
Let $p\geq 3$ and $V,W$ be irreducible $F\widetilde \A_n$-representations which are not 1-dimensional. Then $V\otimes W$ is irreducible if and only if one of the following holds up to exchange of $V$ and $W$:
\begin{enumerate}[(i)]
\item $n\not\equiv 0\Md p$, $V\cong E^{(n-1,1)}$, $W\cong E^\la_\pm$ with $\la\in\Par_p(n)$ a JS-partition satisfying $\la=\la^\Mull$, in which case $V\otimes W\cong E^\nu$ where $\nu=(\la\setminus A)\cup B$ with $A$ is the top removable node of $\la$ and $B$ is either of the two bottom addable nodes of $\la$,

\item $n\not\equiv 0\Md p$, $V\cong E^{(n-1,1)}$, $W\cong E(\la,\pm)$ with $\la\in\RPar_p(n)$ a $\JS(0)$-partition, in which case $V\otimes W\cong E(\nu,0)$ where $\nu=(\la\setminus A)\cup B$ with $A$ is the bottom removable node of $\la$ and $B$ is the top addable node of $\la$,

\item $n\not\equiv 0,\pm 2\Md p$ is odd, $V\cong E^\la$ where $\la\in\Par_p(n)$ is a JS-partition with $\min\{h(\la),h(\la^\Mull)\}=2$, and $W$ is basic spin, in which case, assuming $h(\la)=2$, if 
$\la_1\not=\la_2+p-2$ then $V\otimes W\cong E(\be_{\la_1}+\be_{\la_2},0)$, while if 
$\la_1=\la_2+p-2$ then $V\otimes W\cong E(\be_{\la_1}\cup\be_{\la_2},0)$,

\item $V$ and $W$ are as in rows 2-7 of Table I.
\end{enumerate}
\end{theor}

Although in cases (ii) of Theorem \ref{TS} and (iii) of Theorem \ref{TA} we only describe for $V\otimes W$ if $h(\la)=2$, in the other case a description can be easily obtained, since $D^{\la^\Mull}\cong D^\la\otimes \sgn$ and $E^{\la^\Mull}\cong E^\la$.

In the next section we will introduce notation that will be used in the paper and state some well known/easy results. In Section \ref{s3} we will study endomorphism rings $\End_F(V)$ of general classes of modules $V$ of $\widetilde\s_n$ or $\widetilde\A_n$. In order to extend these results to some special classes of modules or at least obtain similar results in Section \ref{s5}, we will in Section \ref{s4} study the structure of certain permutation modules. In both Sections \ref{s3} and \ref{s5} we will often use branching results to obtain informations on $\End_F(V)$. In Section \ref{s6} we will study tensor product with certain special classes of modules, using results on branching or known results in characteristic 0 and knowledge of decomposition matrices. Finally in Section \ref{s7} we will prove Theorems \ref{TS} and \ref{TA}.

\section{Notation and basic results}

Throughout the paper $p$ will denote the characteristic of the field $F$ and we will assume that $p\not=2$.

\subsection{Covering groups}

Let $\widetilde\s_n$ be any of the two double covers of $\s_n$ and $z$ be the non-trivial central element of $\widetilde\s_n$ (which has order 2). There exists a short exact sequence
\[1\rightarrow\langle z\rangle\rightarrow\widetilde\s_n\xrightarrow{\pi}\s_n\rightarrow 1.\]
For any group $G\leq\s_n$ define $\widetilde G:=\pi^{-1}G\leq\widetilde\s_n$. In particular $\widetilde \A_n$ is the double cover of $\A_n$. Further for elements $g\in \s_n$ let $\widetilde g\in\widetilde\s_n$ be a (fixed) element in $\pi^{-1}\{g\}$, so that $\pi^{-1}\{g\}=\{\widetilde g,z\widetilde g\}$. If $g$ has odd order, one of the elements in $\pi^{-1}\{g\}$ has order $\ord(g)$, while the other has order $2\ord(g)$. In this case choose $\widetilde g$ to have the same order as $g$.

As noted in the introduction, the irreducible representations of $F\widetilde\s_n$ (resp. $F\widetilde \A_n$) are given by the irreducible representations of $F\s_n$ (resp. $F\A_n$), on which $z$ acts trivially, and the spin irreducible representations, on which $z$ acts as $-1$.

Note that it does not matter which double cover of the symmetric group $\s_n$ we consider, since the group algebras of the two double covers of $\s_n$ are isomorphic, see \cite[Note after Theorem 1.2]{s} for the characteristic 0 case, the general case holding similarly.

\subsection{
Representations of symmetric and alternating groups}

As noted in the introduction irreducible representations of $\s_n$ or $\A_n$ are indexed by elements of $\Par_p(n)$, that is $p$-regular partitions of $n$. We write $\Parinv_p(n)$ for the set of partitions $\la\in\Par_p(n)$ with $\la=\la^\Mull$, that is partitions $\la$ for which $D^\la\da_{\A_n}$ splits.

Given a partition $\la\in\Par_p(n)$ define normal, good, conormal and cogood nodes of $\la$ as in \cite[\S11.1]{KBook}. It can be easily seen from the definition that $\la$ is JS if and only if it has only one normal node.

If $(a,b)$ is a node, let $(b-a)\Md p$ be the residue of $(a,b)$. For any partition $\la$ let the content of $\la$ be the tuple $(a_0,\ldots,a_{p-1})$, where $a_i$ is the number of nodes of $\la$ of residue $i$ for each $0\leq i<p$. It is well known that if $\la,\mu\in\Par_p(n)$, then $D^\la$ and $D^\mu$ are in the same block if and only if $\la$ and $\mu$ have the same content, see \cite[2.7.41, 6.1.21]{jk} and \cite[\S11]{JamesBook}, so that we may speak of content of a block or of a block with a certain content (which is unique if such a block exists). Let $V$ be an $F\s_n$-module in a block $B$ with content $(a_0,\ldots,a_{p-1})$. For any residue $i$, we define $e_iV$ to be the projection of $V\da_{\s_{n-1}}$ to the block with content $(a_0,\ldots,a_{i-1},a_i-1,a_{i+1},\ldots,a_{p-1})$ and $f_iV$ to be the projection of $V\ua^{\s_{n+1}}$ to the block with content $(a_0,\ldots,a_{i-1},a_i+1,a_{i+1},\ldots,a_{p-1})$. We then extend the definition of $e_iV$ and $f_iV$ to arbitrary $F\s_n$-modules additively to obtain functors 
$$
e_i:\mod{F\s_n}\to \mod{F\s_{n-1}},\quad f_i:\mod{F\s_n}\to \mod{F\s_{n+1}}.
$$

More generally, for any $r\geq 1$ let
$$e_i^{(r)}:\mod{F \s_n}\rightarrow \mod{F \s_{n-r}},\quad f_i^{(r)}:\mod{F \s_n}\rightarrow\mod{F \s_{n+r}},$$ 
be the divided power functors, see \cite[\S11.2]{KBook}. The following is well-known, see for example \cite[Lemma 8.2.2(ii),  Theorems 8.3.2(i), 11.2.7,  11.2.8]{KBook}: 

\begin{lemma}\label{Lemma45}
For any residue $i$ and any $r\geq 1$, the functors $e_i^{(r)}$ and $f_i^{(r)}$ are biadjoint and commute with duality. Further,  
for any $F \s_n$-module $V$ we have 
\[V\da_{\s_{n-1}}\cong e_0V\oplus\ldots\oplus e_{p-1}V\hspace{24pt}\text{and}\hspace{24pt}V\ua^{\s_{n+1}}\cong f_0V\oplus\ldots\oplus f_{p-1}V.\]
\end{lemma}

For any partition $\la\in\Par_p(n)$ and any residue $i$, let $\eps_i(\la)$ be its number of $i$-normal nodes and $\phi_i(\la)$ be its number of $i$-conormal nodes. If $\eps_i(\la)>0$ let $\widetilde e_i\la\in\Par_p(n-1)$ be the partition obtained from $\la$ by removing the bottom $i$-normal node, while if $\phi_i(\la)>0$ let $\widetilde f_i\la\in\Par_p(n+1)$ be the partition obtained from $\la$ by adding the top $i$-conormal node (see \cite[\S 11.1]{KBook}). For $r\geq 2$ define
\[e_i^rV:=\underbrace{e_i\cdots e_i}_{r\text{ times}}V\]
and define similarly $f_i^rV$, $\tilde e_i^r\la$ and $\tilde f_i^r\la$ (the last two are only defined if $\eps_i(\la)\geq r$ or $\phi_i(\la)\geq r$ respectively).

The following two results hold by \cite[Theorems E(iv), E'(iv)]{BrK1}, \cite[Theorems 11.2.10,  11.2.11]{KBook} and \cite[Theorem 1.4]{KDec}.

\begin{lemma}\label{Lemma39}
Let $\lambda\in\Par_p(n)$. Then for any residue $i$ and any $r\geq 1$:
\begin{enumerate}
\item[{\rm (i)}] $e_i^rD^\lambda\cong(e_i^{(r)}D^\lambda)^{\oplus r!}$;
\item[{\rm (ii)}]  $e_i^{(r)}D^\lambda\not=0$ if and only if $r\leq \eps_i(\lambda)$, in which case $e_i^{(r)}D^\lambda$ is a self-dual indecomposable module with socle and head both isomorphic to $D^{\widetilde e_i^r\la}$.  
\item[{\rm (iii)}]  $[e_i^{(r)}D^\lambda:D^{\widetilde e_i^r\la}]=\binom{\eps_i(\lambda)}{r}=\dim\End_{\s_{n-r}}(e_i^{(r)}D^\lambda)$;
\item[{\rm (iv)}] if $D^\mu$ is a composition factor of $e_i^{(r)}D^\lambda$ then $\eps_i(\mu)\leq \eps_i(\lambda)-r$, with equality holding if and only if $\mu=\widetilde e_i^r\la$;
\item[{\rm (v)}] 
$\dim\End_{\s_{n-1}}(D^\lambda\da_{\s_{n-1}})=\sum_{j\in I}\eps_j(\lambda)$.
\item[{\rm (vi)}] Let $A$ be a removable node of $\la$ such that $\la_A$ is $p$-regular. Then $D^{\la_A}$ is a composition factor of $e_i D^\la$ if and only if $A$ is $i$-normal, in which case $[e_i D^\la:D^{\la_A}]$ is one more than the number of $i$-normal nodes for $\la$ above $A$. 
\end{enumerate}
\end{lemma}

\begin{lemma}\label{Lemma40}
Let $\lambda\in\Par_p(n)$. Then for any residue $i$ and any $r\geq 1$:
\begin{enumerate}
\item[{\rm (i)}] $f_i^rD^\lambda\cong(f_i^{(r)}D^\lambda)^{\oplus r!}$;
\item[{\rm (ii)}] $f_i^{(r)}D^\lambda\not=0$ if and only if $r\leq \phi_i(\lambda)$, in which case $f_i^{(r)}D^\lambda$ is a self-dual indecomposable module with socle and head both isomorphic to $D^{\widetilde f_i^r\la}$.  
\item[{\rm (iii)}]  $[f_i^{(r)}D^\lambda:D^{\widetilde f_i^r\la}]=\binom{\phi_i(\lambda)}{r}=\dim\End_{\s_{n+r}}(f_i^{(r)}D^\lambda)$;
\item[{\rm (iv)}] if $D^\mu$ is a composition factor of $f_i^{(r)}D^\lambda$ then $\phi_i(\mu)\leq \phi_i(\lambda)-r$, with equality holding if and only if $\mu=\widetilde f_i^r\la$.
\item[{\rm (v)}]
$\dim\End_{\s_{n+1}}(D^\lambda\ua^{\s_{n+1}})=\sum_{j\in I}\phi_j(\lambda)$.
\item[{\rm (vi)}] Let $B$ be an addable node for $\la$ such that $\la^B$ is $p$-regular. Then $D^{\la^B}$ is a composition factor of $f_i D^\la$ if and only if $B$ is $i$-conormal, in which case $[f_i D^\la:D^{\la^B}]$ is one more than the number of $i$-conormal nodes for $\la$ below~$B$. 
\end{enumerate}
\end{lemma}

The next lemma compares the functors $e_ie_j$ and $e_je_i$ for different residues $i$ and $j$.

\begin{lemma}\label{l22}{\cite[Lemma 4.8]{m1}}
Let $\lambda\vdash n$ be $p$-regular. For $i\not=j$ we have that
\[\dim\Hom_{\s_{n-2}}(e_je_iD^\lambda,e_ie_jD^\lambda)\geq\epsilon_i(\lambda)\epsilon_j(\lambda).\]
\end{lemma}

When considering (co)good or (co)normal nodes and the Mullineux map we have the following result:

\begin{lemma}\label{l17}{\cite[Theorem 4.7]{k4}}
For any partition $\lambda$ and for any residue $i$,
\[\epsilon_i(\lambda)=\epsilon_{-i}(\lambda^\Mull)\hspace{12pt}\mbox{and}\hspace{12pt}\phi_i(\lambda)=\phi_{-i}(\lambda^\Mull).\]
If $\epsilon_i(\lambda)>0$ then $\widetilde{e}_i(\lambda)^\Mull=\widetilde{e}_{-i}(\lambda^\Mull)$, while if $\phi_i(\lambda)>0$ then $\widetilde{f}_i(\lambda^\Mull)=\widetilde{f}_{-i}(\lambda^\Mull)$.
\end{lemma}

\subsection{Spin 
representations}

As noted in the introduction spin irreducible representations of $\widetilde\s_n$ and $\widetilde\A_n$ in characteristic $p\not=0$ are indexed by elements of $\RP_p(n)$, that is $p$-restricted $p$-strict partitions of $n$. In characteristic 0 they are instead indexed by $\RP_0(n)$, the set of partitions in distinct parts, see for example \cite{s}. In view of Lemma \ref{L54}, the labeling in characteristic 0 exactly corresponds to the labeling in characteristic $>n$, as is the case for the non-spin case (see for example \cite{JamesBook}). For $\la\in\RP_p(n)$ remember that $h(\la)$ is the number of parts of $\la$ and define $h_{p'}(\la)$ to be the number of parts of $\la$ which are not divisible by $p$. If $n-h_{p'}(\la)$ is even let $a(\la):=0$, while if $n-h_{p'}(\la)$ is odd let $a(\la):=1$. In \cite{BK,BK2} it has been proved that if $a(\la)=0$ then $\la$ indexes one spin irreducible representation of $\widetilde\s_n$ and two of $\widetilde\A_n$, while if $a(\la)=1$ then $\la$ indexes two spin irreducible representations of $\widetilde\s_n$ and one of $\widetilde\A_n$. So
\begin{align*}
&\{D(\la,0)|\la\in\RP_p(n)\text{ with }a(\la)=0\}\\
&\cup\{D(\la,+),D(\la,-)|\la\in\RP_p(n)\text{ with }a(\la)=1\}
\end{align*}
is a complete set of spin irreducible $F\widetilde\s_n$-representations up to isomorphism and
\begin{align*}
&\{E(\la,+),E(\la,-)|\la\in\RP_p(n)\text{ with }a(\la)=0\}\\
&\cup\{E(\la,0)|\la\in\RP_p(n)\text{ with }a(\la)=1\}
\end{align*}
is a complete set of spin irreducible $F\widetilde \A_n$-representations up to isomorphism.

When $a(\la)=1$ it is often easier to work with $D(\la,+)\oplus D(\la,-)$ instead of working with $D(\la,+)$ and $D(\la,-)$ separately. For this reason we define the irreducible supermodule
\[D(\la):=\left\{\begin{array}{ll}
D(\la,0),&a(\la)=0,\\
D(\la,+)\oplus D(\la,-),&a(\la)=1.
\end{array}\right.\]
Similarly we define $E(\la)$. 
We say that $D(\la)$ is of type M if $a(\la)=0$ or of type Q if $a(\la)=1$. Note that $\dim\End_{\widetilde\s_n}(D(\la))=1+a(\la)$, since if $a(\la)=1$ then $D(\la)$ is the direct sum of two non-isomorphic simple modules.
 
When considering spin modules of $\s_n$ in characteristic 0 we will also write $S(\la)$ for $D(\la)$ and similarly $S(\la,0)$ or $S(\la,\pm)$ for $D(\la,0)$ or $D(\la,\pm)$.

Given supermodules $V$ of $\widetilde\s_{\mu}$ and $W$ of $\widetilde\s_{\nu}$  (with $\mu,\nu$ compositions) we can consider their ``outer'' tensor product $V\boxtimes W$ as a supermodule of $\widetilde\s_{\mu,\nu}$. Outer tensor products of supermodules are not always simple as supermodules (see for example \cite[Section 2-b]{BK}). If $V$ and $W$ are irreducible supermodules, then there exists an irreducible supermodule $M\circledast N$ such that:
\begin{enumerate}[-]
\item
if both $V$ and $W$ are of type M then $V\boxtimes W\cong V\circledast W$ is of type M,

\item
if one $V$ and $W$ is of type M and the other of type Q then $V\boxtimes W\cong V\circledast W$ is of type Q,

\item
if both $V$ and $W$ are of type Q then $V\boxtimes W\cong (V\circledast W)^{\oplus 2}$ with $V\circledast W$ of type M.
\end{enumerate}
For partitions $\la^j\in\RP_p(n_j)$, this can then be extended to define simple supermodules $D(\la^1)\circledast\cdots\circledast D(\la^h)$. We will write $D(\la^1,\ldots,\la^h)$ for $D(\la^1)\circledast\cdots\circledast D(\la^h)$ and $D(\la^1,\ldots,\la^h,0)$ or $D(\la^1,\ldots,\la^h,\pm)$ for its simple components (as module).

For supermodules $V$ and $W$, we write $V\simeq W$ if there exists an even isomorphism $V\to W$ (see \cite[\S2-b]{BK} and \cite[\S2]{BK2}). In particular if $V\simeq W$ then $V\cong W$ as modules.

There are branching rules for spin irreducible supermodules which are similar to branching rules for irreducible representations of symmetric groups. Before  stating them, we need to define some combinatorial notions. We start by defining residues of nodes. The residue of the node $(a,b)$ is given by $\res(b)$, where $\res(b)$ is defined as in the introduction. So the residue of any node is an integer $i$ with $0\leq i\leq\ell$, where $\ell=\ell_p=(p-1)/2$, and on any row residues are given by
\[0,1,\ldots,\ell-1,\ell,\ell-1,\ldots,1,0,0,1,\ldots,\ell-1,\ell,\ell-1,\ldots,1,0,\ldots.\]
Again define the content of a partition $\la$ to be $(a_0,\ldots,a_\ell)$ if for every $0\leq i\leq\ell$ we have that $\la$ has $a_i$ nodes of residue $i$. Normal nodes (and conormal, good and cogood nodes) can be defined also for $p$-restricted $p$-strict partitions, see for example \cite[Section 9-a]{BK}. Let $\la\in\RP_p(n)$. For $0\leq i\leq (p-1)/2$ let $\eps_i(\la)$ be the number of $i$-normal nodes of $\la$ and $\phi_i(\la)$ be the number of $i$-conormal nodes of $\la$. If $\eps_i(\la)>0$ we write $\widetilde e_i\la$ for the partition obtained from $\la$ by removing the $i$-good node of $\la$. Similarly, if $\phi_i(\la)>0$ we write $\widetilde f_i\la$ for the partition obtained from $\la$ by adding the $i$-cogood node of $\la$. We say that $\la\in\RP_p(n)$ is JS if it has only one normal node. As will be seen for example in Lemmas \ref{Lemma39s} and \ref{L15}, the residue of the normal node will play an important role (in particular it is important if the unique normal node has residue 0 or not). If $\la$ is JS and its normal node has residue $i$ we say that $\la$ is $\JS(i)$ (or write $\la\in\JS(i)$). For $i=0$ a combinatorial description of JS(0) partitions has been given in the introduction. By definition we easily have that

\begin{lemma}\label{Lef}
Let $\la\in\RP_p(n)$ and $0\leq i\leq(p-1)/2$.
\begin{enumerate}[-]
\item if $\eps_i(\la)>0$ then $\phi_i(\widetilde e_i\la)>0$ and $\widetilde f_i\widetilde e_i\la=\la$. Further if $i=0$ then $a(\widetilde e_i\la)=a(\la)$, while if $i>0$ then $a(\widetilde e_i\la)=1-a(\la)$;

\item if $\phi_i(\la)>0$ then $\eps_i(\widetilde f_i\la)>0$ and $\widetilde e_i\widetilde f_i\la=\la$. Further if $i=0$ then $a(\widetilde f_i\la)=a(\la)$, while if $i>0$ then $a(\widetilde f_i\la)=1-a(\la)$.
\end{enumerate}
\end{lemma}

 It can be checked that $D(\la,\de)$ and $D(\mu,\eps)$ are in the same block if and only if $\la$ and $\mu$ have the same content (unless possibly if $\la=\mu$ is a $p$-bar core, in which case the blocks have weight 0), see \cite[(22.7), 22.3.20]{KBook} and the definition of $p$-bar core just before \cite[(22.7)]{KBook}. If $M$ is a spin module of $\widetilde\s_n$ contained in the block(s) with content $(a_0,\ldots,a_\ell)$ and $i$ is a residue, let $\Res_i M$ be the block(s) component(s) of $M\da_{\widetilde\s_{n-1}}$ corresponding to the blocks with content $(a_0,\ldots,a_{i-1},a_i-1,a_{i+1},\ldots,a_\ell)$. Define similarly $\Ind_i M$ as the block(s) component(s) of $M\ua^{\widetilde\s_{n+1}}$ corresponding to the blocks with content $(a_0,\ldots,a_{i-1},a_i+1,a_{i+1},\ldots,a_\ell)$. This can then be extended to arbitrary spin modules. Often the modules $\Res_iD(\la)$ and $\Ind_iD(\la)$ are not indecomposable as supermodules. However there exist modules $e_iD(\la)$ and $f_iD(\la)$ such that the following, see \cite[Theorems 9.13, 9.14]{BK} and \cite[Theorem A]{ksh}:

\begin{lemma}\label{Lemma39s}
Let $\lambda\in\RP_p(n)$ and $0\leq i\leq\ell$. Then:
\begin{enumerate}
\item $\Res_i D(\lambda)\cong (e_iD(\lambda))^{\oplus 1+a(\la)\de_{i>0}}$;
\item $D(\la)\da_{\widetilde\s_{n-1}}\cong e_0D(\la)\oplus\bigoplus_{j=1}^\ell (e_jD(\la))^{\oplus 1+a(\la)}$;
\item  $e_iD(\lambda)\not=0$ if and only if $\eps_i(\lambda)>0$, in which case $e_iD(\lambda)$ is a self-dual indecomposable supermodule with socle and head both isomorphic to $D(\widetilde e_i\la)$;
\item  $[e_iD(\lambda):D(\widetilde e_i\la)]=\eps_i(\la)$;
\item if $D(\mu)$ is a composition factor of $e_iD(\lambda)$ then $\eps_i(\mu)\leq \eps_i(\lambda)-1$, with equality holding if and only if $\mu=\widetilde e_i\la$;
\item 
$\End_{\widetilde\s_{n-1}}(e_i D(\la))\simeq\End_{\widetilde\s_{n-1}}(D(\widetilde e_i\la))^{\oplus \eps_i(\la)}$;
\item $\Hom_{\widetilde\s_{n-1}}(e_i D(\la),e_iD(\nu))=0$ if $\nu\in\RP_p(n)$ with $\nu\not=\la$;
\item if $A$ is an $i$-normal node of $\la$ such that $\la\setminus A\in\RP_p(n-1)$, then $D(\la\setminus A)$ is a composition factor of $e_i D(\la)$.
\end{enumerate}
\end{lemma}

\begin{lemma}\label{Lemma40s}
Let $\lambda\in\RP_p(n)$, $0\leq i\leq(p-1)/2$. Then:
\begin{enumerate}
\item $\Ind_i D(\lambda)\cong (f_iD(\lambda))^{\oplus 1+a(\la)\de_{i>0}}$;
\item $D(\la)\ua^{\widetilde\s_{n+1}}\cong f_0D(\la)\oplus\bigoplus_{j=1}^\ell (f_jD(\la))^{\oplus 1+a(\la)}$;
\item  $f_iD(\lambda)\not=0$ if and only if $\phi_i(\lambda)>0$, in which case $f_iD(\lambda)$ is a self-dual indecomposable supermodule with socle and head both isomorphic to $D(\widetilde f_i\la)$;
\item  $[f_iD(\lambda):D(\widetilde f_i\la)]=\phi_i(\la)$;
\item if $D(\mu)$ is a composition factor of $f_iD(\lambda)$ then $\phi_i(\mu)\leq \phi_i(\lambda)-1$, with equality holding if and only if $\mu=\widetilde f_i\la$;
\item 
$\End_{\widetilde\s_{n+1}}(f_i D(\la))\simeq(\End_{\widetilde\s_{n+1}}(D(\widetilde f_i\la)))^{\oplus \phi_i(\la)}$;
\item $\Hom_{\widetilde\s_{n+1}}(f_i D(\la),f_iD(\nu))=0$ if $\nu\in\RP_p(n)$ with $\nu\not=\la$;
\item if $B$ is an $i$-conormal node of $\la$ such that $\la\cup B\in\RP_p(n+1)$, then $D(\la\cup B)$ is a composition factor of $f_i D(\la)$.
\end{enumerate}
\end{lemma}

When considering restrictions to $\widetilde\s_{n-r}$ we have that there exists divided power modules $e_i^{(r)}D(\la)$ with $e_i^{(1)}D(\la)\cong e_iD(\la)$ such that the following holds, see \cite[Lemma 22.3.15]{KBook} for the first part and use Lemma \ref{Lemma39s} to obtain the other two (there also exists divided power $F\widetilde\s_{n+r}$-modules $f_i^{(r)}D(\la)$ with corresponding properties, though these will not be needed in this paper). Again define $\Res_i^rV$ and $\tilde e_i^r\la$ similarly to what had been done for the non-spin case.

\begin{lemma}\label{Lemma39sr}
Let $\lambda\in\RP_p(n)$, $0\leq i\leq(p-1)/2$. Then:
\begin{enumerate}
\item $\Res_i^r D(\lambda)\cong(e_i^{(r)}D(\lambda))^{\oplus 2^{\de_{i>0}\lfloor (r+a(\la))/2\rfloor}r!}$;
\item  $e_i^{(r)}D(\lambda)\not=0$ if and only if $\eps_i(\lambda)\geq r$;
\item  $[e_i^{(r)}D(\lambda):D(\widetilde e_i^r\la)]=\binom{\eps_i(\la)}{r}$.
\end{enumerate}
\end{lemma}

Further by \cite[Lemma 19.1.1, Theorems 22.2.2, 22.2.3]{KBook}:

\begin{lemma}\label{Lefd}
The functors $e_i$ and $f_i$ are biadjoint and commute with duality.
\end{lemma}

Comparing the number of normal and conormal nodes, we obtain the following lemma, which holds by Lemmas \ref{Lef}, \ref{Lemma39s} and \ref{Lemma40s}:

\begin{lemma}\label{L15}
Let $\la\in\RPar_p(n)$. Then
\begin{align*}
\dim\End_{\widetilde\s_{n-1}}(D(\la)\da_{\widetilde\s_{n-1}})&=(\epsilon_0(\la)+2\sum_{i\geq 1}\epsilon_i(\la))\dim\End_{\widetilde\s_n}(D(\la)),\\
\dim\End_{\widetilde\s_{n+1}}(D(\la)\ua^{\widetilde\s_{n+1}})&=(\phi_0(\la)+2\sum_{i\geq 1}\phi_i(\la))\dim\End_{\widetilde\s_n}(D(\la)).
\end{align*}
\end{lemma}

Further, by the same lemmas, the following holds about the module $D(\la)\da_{\widetilde\s_{n-1}}\ua^{\widetilde\s_n}$:

\begin{lemma}\label{L9}
Let $\la\in\RPar_p(n)$. Then
\[[\Ind_i\Res_i D(\la):D(\la)]=\epsilon_i(\la)(\phi_i(\la)+1)(1+\de_{i>0}).\]
In particular
\[[D(\la)\otimes M_1:D(\la)]=\epsilon_0(\la)(\phi_0(\la)+1)+2\sum_{i\geq 1}\epsilon_i(\la)(\phi_i(\la)+1).\]
\end{lemma}

By Mackey induction-reduction theorem we have that $M\ua^{\widetilde\s_{n+1}}\da_{\widetilde\s_n}\cong M\oplus M\da^{\widetilde\s_{n-1}}\ua_{\widetilde\s_n}$ for any module $M$ of $F\widetilde\s_n$. The next two lemmas then follows (for the first one use also Lemma \ref{L15}):

\begin{lemma}\label{L10}
Let $\la\in\RPar_p(n)$. Then
\[\epsilon_0(\la)+2\sum_{i\geq 1}\epsilon_i(\la)+1=\phi_0(\la)+2\sum_{i\geq 1}\phi_i(\la).\]
In particular $\eps_0(\la)+\phi_0(\la)$ is odd.
\end{lemma}

\begin{lemma}\label{L051218_3}
If $i\not=j$ and $A$ is any spin module of $\widetilde\s_{n}$ then $\Ind_j\Res_i M\cong \Res_i\Ind_j M$.
\end{lemma}

The next results consider normal nodes of different residues. 
\begin{lemma}\label{L081218_2}
Let $\la\in\RP_p(n)$. If $i\not=j$ and $\eps_i(\la),\eps_j(\la)>0$ then $\End_{\widetilde\s_{n-2}}(e_i D(\widetilde e_j\la),e_j D(\widetilde e_i\la))\not=0$.
\end{lemma}

\begin{proof}
Using Lemmas \ref{Lemma39s}, \ref{Lemma40s} and \ref{L051218_3} we have that there exists $c>0$ such that
\begin{align*}
\dim\End_{\widetilde\s_{n-2}}(e_i D(\widetilde e_j\la),e_j D(\widetilde e_i\la))&=c\dim\End_{\widetilde\s_{n-1}}(\Ind_j\Res_i D(\widetilde e_j\la),D(\widetilde e_i\la))\\
&=c\dim\End_{\widetilde\s_{n-1}}(\Res_i\Ind_jD(\widetilde e_j\la),D(\widetilde e_i\la))\\
&=c\End_{\widetilde\s_n}(\Ind_j D(\widetilde e_j\la),\Ind_iD(\widetilde e_i\la))
\end{align*}
from which the lemma follows, since both $\Ind_j D(\widetilde e_j\la)$ and $\Ind_i D(\widetilde e_i\la)$ contain $D(\la)$ in their head and socle by Lemmas \ref{Lef} and \ref{Lemma40s}.
\end{proof}

\begin{lemma}\label{L051218_4}
If $\la\in\RP_p(n)$ and $\epsilon_i(\la)>0$ then $\epsilon_j(\widetilde e_i\la)\geq\epsilon_j(\la)$ for $j\not=i$.
\end{lemma}

\begin{proof}We may assume that $\eps_j(\la)>0$. Then from Lemmas \ref{Lef}, \ref{Lemma40s} and \ref{L051218_3},
\[0\not= \Res_i^{\epsilon_i(\la)} D(\la)\subseteq \Res_i^{\epsilon_i(\la)} \Ind_j D(\widetilde e_j\la)\cong \Ind_j\Res_i^{\epsilon_i(\la)} D(\widetilde e_j\la).\]
In particular $\Res_i^{\epsilon_i(\la)} D(\widetilde e_j\la)\not=0$ from which the lemma follows by Lemma \ref{Lemma39sr}.
\end{proof}

\subsection{Reduction modulo $p$}

We now consider some results about reduction modulo $p$ of spin representations in characteristic 0. If $\mu\in\RP_0(n)$, let $\mu^R\in\RP_p(n)$ be as defined in \cite{BK3}. Given two partitions $\alpha,\beta$, write $\al\lhd\be$ if $\al$ is smaller than $\be$ in the dominance order. The main known result is the following, see \cite[Theorem 10.8]{BK2}, \cite[Theorem 10.4]{BK4}and \cite[Theorem 4.4]{BK3}:

\begin{lemma}\label{L54}
Let $\mu\in\RP_0(n)$ and $\nu\in\RP_p(n)$. If $[S(\mu):D(\nu)]>0$ then $\nu\unlhd \mu^R$. Further
\[[S(\mu):D(\mu^R)]=2^{(h(\mu)-h_{p'}(\mu)+a(\mu)-a(\mu^R))/2}.\]
\end{lemma}

Let $n=ap+b$ with $0\leq b<p$. The spin irreducible representations of $\widetilde\s_n$ and $\widetilde \A_n$ indexed by the partition
\[\be_n:=\left\{\begin{array}{ll}
(p^a,b),&b\not=0,\\
(p^{a-1},p-1,1),&b=0
\end{array}\right.\]
are called basic spin modules. By the next lemma, which holds by \cite[Table III]{Wales} and Lemma \ref{L54}, basic spin modules in characteristic $p$ are exactly the composition factors of the reduction modulo $p$ of basic spin modules in characteristic 0 (indexed by $(n)\in\RP_0(n)$).

\begin{lemma}\label{LBS}\label{l1}
Let $p\geq 3$. Then
\begin{enumerate}[-]
\item if $p\nmid n$ and $2\nmid n$ then $S((n),0)\cong D(\be_n,0)$,

\item if $p\nmid n$ and $2\mid n$ then $S((n),\pm)\cong D(\be_n,\pm)$,

\item if $p\mid n$ and $2\nmid n$ then $S((n),0)\cong D(\be_n,+)\oplus D(\be_n,-)$,

\item if $p\mid n$ and $2\mid n$ then $S((n),\pm)\cong D(\be_n,0)$.
\end{enumerate}
\end{lemma}

The next lemma shows that there are cases where it is easy to compute $\mu^R$ using the partions $\be_{\mu_i}$. Here as in the following $\unlhd$ denotes the dominance order.

\begin{lemma}\label{L55}
Let $p\geq 3$ and $\mu=(\mu_1,\ldots,\mu_{h(\mu)})\in\RP_0(n)$. Then $\mu^R\unlhd\sum\be_{\mu_i}$ with equality holding if and only if $\mu_i\geq\mu_{i+1}+p$ for $1\leq i<h(\mu)$. Further if $\mu_i\geq\mu_{i+1}+p+\de_{p\mid\mu_{i+1}}$ and $\nu\in\RP_0(n)$ with $\nu\not\!\!\unlhd\mu$ then $[S(\nu):D(\mu^R)]=0$.
\end{lemma}

\begin{proof}
Let
\[\bar\mu:=\cup\{(j,p(i-1)+k)|(j,k)\in\be_{\mu_i}\},\]
so that the first $p$ columns of $\bar\mu$ correspond to $\be_{\mu_1}$, the second $p$ columns to $\be_{\mu_2}$ and so on. Note that $\bar\mu$ is not necessarily (the Young diagram of) a partition, but $\bar\mu$ and of $\sum\be_{\mu_i}$ always have the same number of nodes on any row. Further $\mu$ and $\bar\mu$ have the same number of nodes on any ladder. It then easily follows from the definition of $\mu^R$ that $\mu^R\unlhd \sum\be_{\mu_i}$.

Assume next that $\mu_i<\mu_{i+1}+p$ for some $1\leq i<h(\mu)$. Let $(j,k)$ be the good node of $\be_{\mu_{i+1}}$. Then $(j+1,p(i-1)+k)\not\in\bar\mu$ and
\[(\bar\mu\setminus(j,pi+k))\cup(j+1,p(i-1)+k)\]
has the same number of nodes as $\mu$ on each ladder. It then follows that $\mu^R\not=\sum\be_{\mu_i}$ in this case.

Assume now that $\mu_i\geq\mu_{i+1}+p$ for $1\leq i<h(\mu)$. Let $A$ be the set of all $1\leq r<h(\mu)$ with $p\mid\mu_r=\mu_{r+1}+p$. Then
\[\sum\be_{\mu_i}=(\bar\mu\setminus\{(h(\be_{\mu_{r+1}}),pr+1)|r\in A\})\cup\{(h(\be_{\mu_{r+1}}),pr)|r\in A\}\]
(that is $\sum\be_{\mu_i}$ is obtained from $\bar\mu$ by moving the last node in the $(r+1)$-th set of $p$ columns one node to the left for all $r\in A$). So $\sum\be_{\mu_i}$ and $\bar\mu$ have the same number of nodes on any ladder. Further by assumption that $\mu_i\geq\mu_{i+1}+p$ it can be checked that $\sum\be_{\mu_i}\in\RP_p(n)$ and so $\mu^R=\sum\be_{\mu_i}$.

Last assume that $\mu_i\geq\mu_{i+1}+p+\de_{p\mid\mu_{i+1}}$. Note that in this case $\bar\mu=\sum\be_{\mu_i}=\mu^R$ by the last paragraph. Let $\nu\in\RP_0(n)$ with $\nu\not\!\!\unlhd\mu$. By Lemma \ref{L54} and the above it is enough to prove that $\sum\be_{\mu_i}\not\!\!\unlhd\sum\be_{\nu_i}$. Pick $r$ with $\nu_1+\ldots+\nu_r>\mu_1+\ldots+\mu_r$ and define $\bar\nu$ similarly to $\bar\mu$. Then the first $rp$ columns of $\bar\nu$ contain more nodes than the first $rp$ columns of $\bar\mu$. In particular the first $rp$ columns of $\sum\be_{\nu_i}$ contain more nodes than the first $rp$ columns of $\sum\be_{\mu_i}$ and so $(\sum\be_{\mu_i})'\not\!\unrhd(\sum\be_{\nu_i})'$, that is $\sum\be_{\mu_i}\not\!\!\unlhd\sum\be_{\nu_i}$.
\end{proof}

\subsection{Module structure}

Often we will need to consider the structure of certain modules. We write
\[W\sim V_1|\ldots|V_h\]
if $W$ has a filtration with subquotients $V_j$ counted from the bottom and
\[W\sim (V_{1,1}|\ldots|V_{1,h_1})\,\,\oplus\,\,\ldots\,\,\oplus\,\,(V_{k,1}|\ldots|V_{k,h_k})\]
if $W\cong W_1\oplus\ldots\oplus W_k$ with $W_j\sim V_{j,1}|\ldots|V_{j,h_j}$.

If $V_1,\ldots,V_h$ are simple we will also write
\[W\cong V_1|\ldots|V_h\]
if $W$ is uniserial with composition factors $V_j$ counted from the bottom and
\[W\cong (V_{1,1}|\ldots|V_{1,h_1})\,\,\oplus\,\,\ldots\,\,\oplus\,\,(V_{k,1}|\ldots|V_{k,h_k})\]
if $W\cong W_1\oplus\ldots\oplus W_k$ with $W_j\cong V_{j,1}|\ldots|V_{j,h_j}$.

Further for groups $G,H$ and modules $A$ of $FG$ and $B$ of $FH$ we will write $A\boxtimes B$ for the corresponding modules of $F(G\times H)$.

\subsection{Permutation modules}

In this subsection we will consider the structure of certain permutation modules and prove some results connecting such permutations modules and the endomorphism ring $\End_F(V)$, for $V$ a $\widetilde\s_n$ or $\widetilde \A_n$ module.

For $\al\in\Par(n)$ a partition of $n$ let $S^\la$ be the reduction modulo $p$ of the Specht module indexed by $\al$ (which can be viewed as an $\widetilde\s_n$-module). Further let $\s_\al$ be the Young subgroup $\s_{\al_1}\times\s_{\al_2}\times\cdots\leq\s_n$ and define $M^\al:=\1\ua_{\widetilde\s_\al}^{\widetilde\s_n}$ to be the corresponding permutation module. It is well known (see for example \cite{JamesBook}) that $S^\al\subseteq M^\al$. Let $\A_\al=\s_\al\cap \A_n$. If $\al\not=(1^n)$ then $\s_\al$ contains an odd element, so $\s_n$ is a single $(\A_n,\s_\al)$ double coset. Hence Mackey's theorem gives that $M^\al\da_{\widetilde \A_n}\cong \1\ua_{\widetilde \A_{\al}}^{\widetilde \A_n}$.

The next lemma holds by Frobenius reciprocity and the definition of $M^\al$.

\begin{lemma}\label{l2}
For any $F\widetilde\s_n$-module $V$ and any $\alpha\in\Par(n)$ we have that
\[\dim\Hom_{\widetilde\s_n}(M^\alpha,\End_F(V))=\dim\End_{\widetilde\s_\alpha}(V\da_{\widetilde\s_\alpha}).\]
Similarly for any $F\widetilde \A_n$-module $W$ and any $(1^n)\not=\alpha\in\Par(n)$ we have that
\[\dim\Hom_{\widetilde \A_n}(M^\alpha,\End_F(W))=\dim\End_{\widetilde \A_\alpha}(W\da_{\widetilde \A_\alpha}).\]
\end{lemma}

We will also use {\em Young modules $Y^\al$} which can be defined using the following well-known facts  contained for example in \cite{JamesArcata} and \cite[\S4.6]{Martin}:

\begin{lemma} \label{LYoung} 
There exist indecomposable $F \s_n$-modules $Y^\al$ for $\al\in\Par(n)$ such that $M^\al\cong Y^\al\,\oplus\, \bigoplus_{\be\rhd\al}(Y^\be)^{\oplus m_{\be,\al}}$ for some $m_{\be,\al}\geq 0$. Moreover, $Y^\al$ can be characterized as the unique indecomposable direct summand of $M^\al$ such that  $S^\al\subseteq Y^\al$. Finally, we have $(Y^\al)^*\cong Y^\al$ for all $\al\in\Par(n)$.
\end{lemma}

In order to prove that in most cases $V\otimes W$ is not irreducible, we will usually prove that $\Hom_G(\End_F(V),\End_F(W))$ is not 1-dimensional by studying the modules $\End_F(V)$ and $\End_F(W)$ separately. This will in many cases be done with the next lemma, which is an analogue of \cite[Lemma 4.2]{m2} for covering groups of symmetric and alternating groups.

\begin{lemma}\label{l15}
Let $G\in\{\widetilde\s_n,\widetilde \A_n\}$ and $B$ and $C$ be $FG$-modules. For $\alpha\in\Par(n)$ let $b_\alpha$ and $c_\alpha$ be such that there exist $\phi^\alpha_1,\ldots,\phi^\alpha_{b_\alpha}\in\Hom_G(M^\alpha,B^*)$ with $\phi^\alpha_1|_{S^\alpha},\ldots,\phi^\alpha_{b_\alpha}|_{S^\alpha}$ linearly independent and that similarly there exist $\psi^\alpha_1,\ldots,\psi^\alpha_{c_\alpha}\in\Hom_G(M^\alpha,C)$ with $\psi^\alpha_1|_{S^\alpha},\ldots,\psi^\alpha_{c_\alpha}|_{S^\alpha}$ linearly independent. Then
\[\dim\Hom_G(B,C)\geq\sum_{\alpha\in D}b_\alpha c_\alpha,\]
where $D=\Par_p(n)$ if $G=\widetilde\s_n$ or $D=\{\al\in\Par_p(n)|\alpha>\alpha^\Mull\}$ if $G=\widetilde \A_n$.
\end{lemma}

Since we will often consider permutation modules $M^{(n-m,\mu)}$ for certain fixed partitions $\mu\in\Par(m)$ (with $m$ small), we will write $M_{\mu_1,\mu_2,\ldots}$ (or $M_\mu$) for the module $M^{(n-m,\mu)}$. Similarly we will write $D_\mu$, $S_\mu$ and $Y_\mu$ for $D^{(n-m,\mu)}$, $S^{(n-m,\mu)}$ and $Y^{(n-m,\mu)}$ (when they are defined).

\subsection{Hooks}

We now consider the structure of the reduction modulo $p$ of Specht modules indexed by hook partitions. Such modules have a quite easy structure, since $p\not=2$.

For $0\leq k\leq n-1-\de_{p\mid n}$ define
\[\overline{D}_{n,k}=\overline{D}_k:=\left\{\begin{array}{ll}
D^{(n-k,(k)^\Mull)},&k<n(p-1)/p,\\
D^{((k+1)^\Mull,n-k-1)},&k\geq n(p-1)/p\text{ and }p\nmid n,\\
D^{((k+2)^\Mull,n-k-2)},&k\geq n(p-1)/p\text{ and }p\mid n.
\end{array}\right.\]
Note that for $k<p$ we then have that $\overline{D}_k=D_{1^k}$ (unless $k=p-1=n-1$). Define $\h_p(n):=\{(a,(b)^\Mull),((c)^\Mull,d)\}\cap\Par_p(n)$, so that $\h_p(n)$ is the set of partition labeling the modules $\overline{D}_k$.

The next lemma holds by \cite[p. 52]{JamesR} and \cite[Theorem 2]{Peel}

\begin{lemma}\label{LH}
Let $p\geq 3$. Then for $0\leq k\leq n-1$:
\begin{enumerate}[-]
\item if $p\nmid n$ then $S_{1^k}\cong \overline{D}_k$,

\item if $p\mid n$ then $S_{1^k}\cong \overline{D}_{k-1}|\overline{D}_k$, where $\overline{D}_{-1}=\overline{D}_{n-1}=0$.
\end{enumerate}
\end{lemma}

The following properties then easily follows:

\begin{lemma}\label{L20}
Let $c=1$ if $p\nmid n$ or $c=2$ if $p\mid n$. Then $\overline{D}_k\cong \overline{D}_{n-c-k}\otimes\sgn$ for each $0\leq k\leq n-c$. In particular $\overline{D}_k\cong \overline{D}_k\otimes\sgn$ if and only if $k=(n-c)/2$.
\end{lemma}

\begin{lemma}\label{L35}
Let $p\geq 3$. Then $\la\in\h_p(n)$ if and only if $\la^\Mull\in\h_p(n)$.
\end{lemma}

If $k\not=(n-1-\de_{p\mid n})/2$ we will then write $\overline{E}_k$ for $\overline{D}_k\da_{\A_n}$. On the other hand if $k=(n-1-\de_{p\mid n})/2$ we will then write $\overline{E}_{k,\pm}$ for the composition factors of $\overline{D}_k\da_{\A_n}$. When working for $\widetilde\s_n$ and $\widetilde \A_n$ at the same time, we will often write $\overline{D}_k$ to also to indicate its restriction to $\widetilde \A_n$.

\section{Special homomorphisms}\label{s3}

In this section, for $G=\widetilde\s_n$ or $\widetilde\A_n$, we will prove that for certain large classes of modules $V$ there exist homomorphisms $\psi\in\Hom_G(M,\End_F(V))$ with $M=M_\mu$ or $S_\mu$ which do not vanish on $S_\mu$.

\subsection{Definition of homomorphisms}\label{ssh}

We now defining certain special elements $x_\mu$. Using these elements we will then define the homomorphisms that will play a role in this section. After having proved some branching rules in \S\ref{sbr}, we will then prove in \S\ref{shr} that these homomorphisms do not vanish on $S^\mu$ for large classes of modules $V$. For $k\geq 3$ odd let $C_k^+$ and $C_k^-$ be the conjugacy classes in $\widetilde \A_{k+1}$ of $\widetilde{(1,2,3,\ldots,k)}$ and $\widetilde{(2,1,3,\ldots,k)}$ respectively (so that $C_k^\pm$ are the two conjugacy classes in $\widetilde \A_{k+1}$ consisting of the odd order lifts of $k$-cycles). Note that since $k\geq 3$ is odd the conjugacy classes $C_k^+$ and $C_k^-$ are distinct, as $(k,1)$ has odd distinct parts and so $(2,1,3,\ldots,k)=(1,2)(1,2,3,\ldots,k)(1,2)$ is not in the $A_{k+1}$ conjugacy class of $(1,2,3,\ldots,k)$.

Define
\begin{align*}
x_3=&\sum_{g\in\s_{\{1,4\}}\times\s_{\{2,5\}}\times\s_{\{3,6\}}}\sgn(g)\widetilde g(\widetilde{(1,2,3)}+\widetilde{(1,3,2)})(\widetilde g)^{-1},\\
x_{3,1^2}:=&\sum_{g\in \s_{4,2,2}}\sum_{h\in\s_{\{2,6,8\}}}\sgn(g)\widetilde g\widetilde h\widetilde{(2,6,8,3,4)}(\widetilde h)^{-1}(\widetilde g)^{-1},\\
x_{1^k}:=&\sum_{g\in C_k^+}\widetilde g-\sum_{g\in C_k^-}\widetilde g,
\end{align*}
where $x_{1^k}$ is defined only for $k\geq 3$ odd.

Let $h\in\s_n$ have odd order. By definition $\widetilde h$ also has odd order. So, for any $g\in\s_n$, $\widetilde{g}\widetilde{h}\widetilde{g}^{-1}=\widetilde{ghg^{-1}}$ is the lift of odd order of $ghg^{-1}$ and this element does not depend on the choice of lift of $g$. 
This fact will be used in the proofs of Lemmas \ref{L6}, \ref{M311} and \ref{L1a}.

In the next statements we will use standard bases for the modules $M^\mu$ and $S^\mu$. For definitions of such bases see \cite[\S4, 8.4]{JamesBook}. When considering a basis element corresponding to a certain tableau, we will just write elements appearing below the first row of the tableau.

\begin{lemma}\label{L6}
Let $n\geq 6$, $G\in\{\widetilde\s_n,\widetilde \A_n\}$ and $V$ be an $FG$-module. If $x_3 V\not=0$ then there exists $\psi\in\Hom_G(M_3,\End_F(V))$ which does not vanish on $S_3$.
\end{lemma}

\begin{proof}
If $V$ is also a representation of $\s_n$ this result is just \cite[Lemma 6.1]{kmt}. For $p>3$ it can be recovered from the proof of \cite[Theorem 7.2]{kt} and the case $p=3$ could be similarly obtained, but we will prove the result here in the general setting.

Let $\{v_{a,b,c}|1\leq a<b<c\leq n\}$ be the standard basis of $M_3$ (by identifying a tabloid with its second row). 
Define $\psi:M_3\to\End_F(V)$ through
\[\psi(v_{a,b,c})(w)=(\widetilde{(a,b,c)}+\widetilde{(a,c,b)})w.\]
For any $g\in\s_n$, $0\leq x\leq 1$ and any pairwise distinct $1\leq a,b,c\leq n$ we have by definition of the standard basis of $M_3$ that $z^x\widetilde{g}v_{a,b,c}=v_{g(a),g(b),g(c)}$ (up to reordering the indexes of $v_{g(a),g(b),g(c)}$) and so, provided $g\in\A_n$ if $G=\widetilde\A_n$,
\begin{align*}
\psi(z^x\widetilde{g}v_{a,b,c})(w)&=\psi(v_{g(a),g(b),g(c)})(w)\\
&=(\widetilde{(g(a),g(b),g(c))}+\widetilde{(g(a),g(c),g(b))})w\\
&=z^x\widetilde{g}(\widetilde{(a,b,c)}+\widetilde{(a,c,b)})(z^x\widetilde{g})^{-1}w\\
&=(z^x\widetilde{g}\psi(v_{a,b,c}))(w).
\end{align*}
Thus $\psi\in\Hom_G(M_3,\End_F(V))$. Let now $t$ be the standard basis vector of $S_3$ corresponding to
\[\begin{array}{cccccc}
1&2&3&7&\cdots&n.\\
4&5&6
\end{array}\]
Then $t=\sum_{g\in\s_{\{1,4\}}\times\s_{\{2,5\}}\times\s_{\{3,6\}}}\sgn(g)v_{g(4),g(5),g(6)}$ and so
\begin{align*}
\psi(t)(w)&=\sum_{g\in\s_{\{1,4\}}\times\s_{\{2,5\}}\times\s_{\{3,6\}}}\sgn(g)(\widetilde{(g(4),g(5),g(6))}+\widetilde{(g(4),g(6),g(5))})w\\
&=-x_3w.
\end{align*}
It follows that $\psi$ does not vanish on $S_3$ if $x_3V\not=0$.
\end{proof}

\begin{lemma}\label{M311}
Let $n\geq 8$, $G\in\{\widetilde\s_n,\widetilde \A_n\}$ and $V$ be an $FG$-module. If $x_{3,1^2}V\not=0$ then there exists $\psi\in\Hom_G(M_{3,1^2},\End_F(V))$ which does not vanish on $S_{3,1^2}$.
\end{lemma}

\begin{proof}
Let $\{v_{\{a,b,c\},d,e}|1\leq a,b,c,d,e\leq n\text{ pairwise distinct}\}$ be the standard basis of $M_{3,1^2}$ (here $v_{\{a,b,c\},d,e}$ corresponds to the tabloid with the elements of $\{1,\ldots,n\}\setminus\{a,b,c,d,e\}$ in the first row, $a,b,c$ in the second row, $d$ in the third row and $e$ in the fourth row). Define $\psi:M_{3,1^2}\to\End_F(V)$ through
\begin{align*}
\psi(v_{\{a,b,c\},d,e})(w)=\sum_{h\in\s_{\{a,b,c\}}}\widetilde h\widetilde{(a,b,c,d,e)}{\widetilde h}^{-1}w
\end{align*}
for $w\in V$. Then similarly to the previous lemma $\psi\in\Hom_G(M_{3,1^2},\End_F(V))$ and if $t$ is the element of the standard basis of $S_{3,1^2}$ corresponding to
\[\begin{array}{cccccc}
1&5&7&9&\cdots&n\\
2&6&8\\
3\\
4
\end{array}\]
then $\psi(t)$ is just multiplication with $x_{3,1^2}$ (by definitions of $t$ and $\psi$).
\end{proof}

\begin{lemma}\label{L1a}
Let $G\in\{\widetilde\s_n,\widetilde \A_n\}$ and $V$ be an $FG$-module. If $k\geq 3$ is odd, $n>k$ and $x_{1^k}V\not=0$ then there exists $0\not=\psi\in\Hom_G(S_{1^k},\End_F(V))$. If $p\nmid k$ then $\psi$ extends to $\phi\in\Hom_G(M_{1^k},\End_F(V))$.
\end{lemma}

\begin{proof}
Let $\{v_{b_1,\ldots,b_k}:1\leq b_j\leq n\text{ pairwise distinct}\}$ and $\{w_{b_1,\ldots,b_k}:2\leq b_1<\ldots<b_k\leq n\}$ be the standard bases of $M_{1^k}$ and $S_{1^k}$ respectively. Since $k\geq 3$ is odd, the cycles $(b_1,b_2,b_3,\ldots,b_k)$ and $(b_2,b_1,b_3,\ldots,b_k)$ are not conjugate in $\A_{\{1,b_1,\ldots,b_k\}}$. 
It follows that also $\widetilde{(b_1,b_2,b_3,\ldots,b_k)}$ and $\widetilde{(b_2,b_1,b_3,\ldots,b_k)}$ are not conjugated in $\widetilde{\A}_{\{1,b_1,\ldots,b_k\}}$ (and so their $\widetilde{\A}_{\{1,b_1,\ldots,b_k\}}$-conjugacy classes are the two distinct conjugacy classes of odd ordered lifts of $k$-cycles). 
For $w\in V$ define
\begin{align*}
\overline\phi(v_{b_1,\ldots,b_k})(w)&:=\widetilde{(b_1,\ldots,b_k)}w,\\
\psi(w_{b_1,\ldots,b_k})(w)&:=\sum_{g\in C_{b_1,\ldots,b_k}^+}gw-\sum_{g\in C_{b_1,\ldots,b_k}^-}gw,
\end{align*}
with $C_{b_1,\ldots,b_k}^+$ and $C_{b_1,\ldots,b_k}^-$ the conjugacy classes of $\widetilde{(b_1,b_2,b_3,\ldots,b_k)}$ and $\widetilde{(b_2,b_1,b_3,\ldots,b_k)}$  in $\widetilde \A_{\{1,b_1,\ldots,b_k\}}$.

Then $\overline\phi\in\Hom_G(M_{1^k},\End_F(V))$ and $\psi\in\Hom_G(S_{1^k},\End_F(V))$ (similarly to Lemma \ref{L6}). Since $\psi(w_{2,\ldots,k+1})$ is given by multiplication with $\pm x_{1^k}$, the first part of the lemma follows. The second part follows from
\[\overline\phi|_{S_{1^k}}=|C_{\s_{k+1}}(b_1,b_2,b_3,\ldots,b_k)|\psi=k\psi.\]
\end{proof}

\subsection{Branching recognition}\label{sbr}

In order to check that in most cases if $V$ is an irreducible representation of $\widetilde\s_n$ or $\widetilde \A_n$ we have that $x_\mu V\not=0$ (for $x_\mu$ one of the elements defined in the previous section), we will prove that $x_\mu W\not=0$, for $W$ a composition factor of $V\da_{\widetilde\s_m}$ or $V\da_{\widetilde \A_m}$ with $m$ small (depending on $\mu$). In order to do this, we wil prove in this section that the restrictions $V\da_{\widetilde\s_m}$ and $V\da_{\widetilde \A_m}$ often contain modules indexed by partitions with similar property as the partition indexing $V$.

\begin{lemma}{\cite[Lemma 2.4]{kt}}\label{L2.4}
Let $p\geq 3$, $n\geq 6$ and $\la\in\RP_p(n)\setminus\{\be_n\}$. Then there exists $\mu\in\RP_p(n-1)\setminus\{\be_{n-1}\}$ such that $D(\mu)$ is a composition factor of $D(\la)\da_{\widetilde\s_{n-1}}$.
\end{lemma}

\begin{lemma}
Let $p=3$, $n\geq 9$ and $\la=(\la_1,\la_2)$ with $\la_1\geq\la_2+2\geq 5$. Then there exists $\mu=(\mu_1,\mu_2)$ with $\mu_1\geq\mu_2+2\geq 5$ such that $D^\mu$ is a composition factor of $D^\la\da_{\s_{n-1}}$.
\end{lemma}

\begin{proof}
If $\la_1\geq\la_2+3$ then $D^{(\la_1-1,\la_2)}$ is a composition factor of $D^\la\da_{\s_{n-1}}$ by Lemma \ref{Lemma39}. If $\la_1=\la_2+2$ then $\la_2\geq 3$ and $D^{(\la_1,\la_2-1)}$ is a composition factor of $D^\la\da_{\s_{n-1}}$ by the same lemma.
\end{proof}

\begin{lemma}\label{L34}
Let $p\geq 3$, $n\geq 7$ and $\la\in\Par_p(n)\setminus\h_p(n)$. Then there exists a composition factor of $D^\la\da_{\s_{n-1}}$ of the form $D^\mu$ with $\mu\in\Par_p(n-1)\setminus\h_p(n-1)$.

Assume now that $n\geq 10$. If further $h(\la),h(\la^\Mull)\geq 3$, then there exists $\mu\in\Par_p(n-1)\setminus\h_p(n-1)$ with $h(\mu),h(\mu^\Mull)\geq 3$ and such that $D^\mu$ a composition factor of $D^\la\da_{\s_{n-1}}$.
\end{lemma}

\begin{proof}
Throughout the proof we will use Lemma \ref{Lemma39} without further reference to it. By Lemma \ref{L35} we have that $\h_p(n)$ is fixed under the Mullineux map. So the lemma holds for $\la$ if and only if it holds for $\la^\Mull$.

We may assume (up to taking $\la^\Mull$) that $\la$ has a good node $A$ such that $\mu=\la\setminus A\in\h_p(n-1)$ or that $n\geq 10$, $\la=(\la_1,\la_2,1)$, $h(\la^\Mull)\geq 3$, $(3,1)$ is good, while $(1,\la_1)$ and $(2,\la_2)$ are not good.

{\bf Case 1:} $n\geq 10$, $\la=(\la_1,\la_2,1)$, $h(\la^\Mull)\geq 3$, $(3,1)$ is good, while $(1,\la_1)$ and $(2,\la_2)$ are not good.

{\bf Case 1.1:} $p=3$. In this case we may assume that $\la_1\geq\la_2+2\geq 5$, since else $\la\in\h_p(n)$. If $\la_1\geq\la_2+3$ then let $B:=(1,\la_1)$.  If $\la_1=\la_2+2$ then $\la_2\geq 4$. In this case let $B:=\la\setminus (2,\la_2)$. is normal in $\la$, $\la\setminus B\not\in\h_p(n-1)$ and it can be easily checked that $h(\la\setminus B)=3$ and $h((\la\setminus B)^\Mull)=5$.

{\bf Case 1.2:} $p\geq 5$. In this case we may assume that $\la_2\geq 2$. Further $\la_1\geq 5$. If $\la_1>\la_2$ let $B:=(1,\la_1)$. If $\la_1=\la_2$ let $B:=(2,\la_2)$. Then $B$ is normal, $\la\setminus B\not\in\h_p(n)$ and $h(\la\setminus B)=3$. Since $(\la\setminus B)_1\geq 4$, the $p$-rim of $\la\setminus B$ contains at least $\min\{p+\de_{p=5},(\la\setminus B)_1+2\}\geq 6$ nodes (where $\de_{p=5}=1$ if $p=5$ and 0 else). So $h((\la\setminus B)^\Mull)\geq 3$.

{\bf Case 2:} $\mu=(n-1)$. Then $\la\in\{(n),(n-1,1)\}$, so $\la\in\h_p(n)$, contradicting the assumptions.

{\bf Case 3:} $\mu=(n-1)^\Mull$. In this case $(n)$ can be obtained from $\la^\Mull$ by removing a good node by Lemma \ref{l17}. So this case follows from case 2.

{\bf Case 4:} $\mu=(n-k-1,1^k)$ with $1\leq k\leq p-2$. Then $\la\in\{(n-k,1^k),(n-k-1,1^{k+1}),(n-k-1,2,1^{k-1})\}$, so we may assume that $\la=(n-k-1,2,1^{k-1})$. Let $B:=(1,n-k-1)$.

{\bf Case 4.1:} $n-k\geq p+3$ or $n-k=p+1$. In this case $B$ is normal in $\la$ and $\la\setminus B\not\in\h_p(n-1)$. Further the first columns of the Mullineux symbols of $\la$ and $\la\setminus B$ are equal. Thus $h(\la)=h(\la\setminus B)$ and $h(\la^\Mull)=h((\la\setminus B)^\Mull)$ and then the lemma holds.

{\bf Case 4.2:} $n-k=p+2$. In this case $n\leq 2p$ and so $p\geq 5$. Again $B$ is normal and $\la\setminus B\not\in\h_p(n-1)$. Further $h(\la\setminus B)=h(\la)$ and, since the first column of the Mullineux symbol of $\la\setminus B$ is $\binom{p+k-1}{k+1}$, we have that $h((\la\setminus B)^\Mull)\geq p-2\geq 3$. So the lemma holds.

{\bf Case 4.4:} $4\leq n-k\leq p$. We may assume that $(n-k,k)\not=(4,p-3)$, since else $\la\setminus B=(n-1)^\Mull$, which was already covered in case 3 (since $B$ is normal). In this case $n\leq 2p-2$ and so $p\geq 5$. Then $B$ is normal and $\la\setminus B\not\in\h_p(n-1)$. Further $h(\la\setminus B)=h(\la)$ and $h((\la\setminus B)^\Mull)\geq n-k-1\geq 3$. So the lemma holds.

{\bf Case 4.5:} $n-k=3$. In this case $\la=(2^2,1^{k-1})$. If $k=p-2$ then $\la\in\h_p(n)$, so, since $n\geq 7$,  we may assume that $4\leq k\leq p-3$ (so in particular $p\geq 7$). In this case $\la^\Mull=(k+1,2)$, so we only have to prove the first part of the lemma, which follows from $C:=(1,k+1)$ being normal in $\la^\Mull$ and from $\la^\Mull\setminus C\not\in\h_p(n-1)$.

{\bf Case 5:} $h(\mu)=p$. By Lemmas \ref{l17} and \ref{L20} we may assume that $\mu=(c,(d)^\Mull)$ with $c+d=n-1$ and $c>d\geq p-1$ (otherwise $\mu^\Mull$ it is of this form or of one of the forms considered in cases 2-4). In this case $\la\in\{(c+1,(d)^\Mull),(c,(d+1)^\Mull),(c,(d)^\Mull,1),(c,(d)^\Mull)\cup (2,(d)^\Mull_1+1)\}$. We may assume that $d\geq p$ and $\la=(c,(d)^\Mull,1)$ or that $(p-1)\nmid d$ and $\la=(c,(d)^\Mull)\cup (2,(d)^\Mull_1+1)$. From $c>d\geq p$ and since $c+d=n-1$ we have $p\leq d\leq(n-2)/2$ and so
\begin{align*}
\la_1-\la_2&\geq c-(\lceil\frac{d}{p-1}\rceil+1)\\
&=(n-d-1)-(\lceil\frac{d}{p-1}\rceil+1)\\
&\geq n-2-\lceil\frac{pd}{p-1}\rceil\\
&\geq \lfloor(n-2)\frac{p-2}{2(p-1)}\rfloor\\
&\geq \lfloor \frac{p(p-2)}{p-1}\rfloor\\
&=\lfloor\frac{(p-1)^2-1}{p-1}\rfloor\\
&=p-2.
\end{align*}

{\bf Case 5.1:} $p\geq 5$. Let $B:=(1,c)$. Then $B$ is normal, $\la\setminus B\not\in\h_p(n-1)$ and $h(\la\setminus B)=h(\la)$. Further the $p$-rim of $\la\setminus B$ contains at least
\[\min\{(\la\setminus B)_1-(\la\setminus B)_2+1,p\}+h(\la\setminus B)-1\geq h(\la\setminus B)+p-2\]
nodes. So the first column of the Mullineux symbol of $\la\setminus B$ is $\binom{\geq h(\la\setminus B)+p-2}{h(\la\setminus B)}$ and then $h((\la\setminus B)^\Mull)\geq p-2\geq 3$.

{\bf Case 5.2:} $p=3$. Again let $B:=(1,c)$. If $h(\la)=3$ then the $p$-rim contains at least $\min\{(\la\setminus B)_1-(\la\setminus B)_2+1,p\}+2$ nodes. If $h(\la)=4$ then the $p$-rim contains at least $\min\{(\la\setminus B)_1-(\la\setminus B)_2+1,p\}+3$ nodes. If $\la_1-\la_2=(\la\setminus B)_1-(\la\setminus B)_2+1\geq 3$ we can then argue in either case as in case 5.1. So assume that $\la_1-\la_2\leq 2$. Then $c-(d+1)/2-1\leq\la_1-\la_2\leq 2$. Since $c\geq n/2$ and $d\leq (n-2)/2$ it follows that $n/2\leq c\leq n/4+3$. So $n\leq 12$ and it can then be easily checked that $\la\in\{(4,2,1^2),(5,3,1),(6,4,2)\}$. In the first case $D^{(3,2,1^2)}$ gives a composition factor of $D^\la\da_{\s_{n-1}}$ as wanted, in the second case $D^{(5,3)}$, in the third case $D^{(6,4,1)}$.
\end{proof}

\subsection{Endomorphisms rings}\label{shr}

We are now ready to study the endomorphisms rings $\End_F(V)$ for $V$ simple $F\widetilde\s_n$- or $F\widetilde \A_n$-modules indexed by certain (large) families of partitions. We will use the elements $x_\mu$ defined at the beginning of \S\ref{ssh}.

In the proofs of this subsection, results considering the number of elements in conjugacy classes of $\s_m$ or $\widetilde\s_m$ have been obtained with GAP \cite{gap}. Character tables for spin irreducible characters in characteristic 0 for $n\leq 13$ can be found in \cite{morris}. These character tables allow us to compute spin character values in this subsection, after having computed if the lifts of elements of odd order on which characters of basic spin modules in characteristic 0 take positive value have odd or even order using \cite[p. 57-58]{morris} (the representation $\Gamma^*_n$ defined there is isomorphic is the representation we denoted by $S((n))$, see \cite[Section 3]{s}).


\begin{lemma}\label{LM3}
Let $p\geq 3$, $n\geq 6$ and $\la\in\Par_p(n)$. If $h(\la),h(\la^\Mull)\geq 3$ and $V$ is a simple $F\s_n$- or $F \A_n$-module indexed by $\la$ then there exists $\phi:M_3\to\End_F(V)$ which does not vanish on $S_3$.
\end{lemma}

\begin{proof}
In view of Lemma \ref{L6} if $x_3V\not=0$ the result holds. If $V\cong D^\la$ (and then also if $V\cong E^\la$) we have that $x_3V\not=0$ by \cite[Proposition 3.8]{bk5} if $p\geq 5$ or by \cite[Lemma 6.6]{kmt} if $p=3$. So we may assume that $V\cong E^\la_\pm$. Since $D^\la\da_{\A_n}\cong E^\la_+\oplus E^\la_-$ there exists $\eps\in\{\pm\}$ such that $x_3E^\la_\eps\not=0$ and then the result holds for $E^\la_\eps$. Since $E^\la_+\cong (E^\la_-)^\si$ for $\si\in\s_n\setminus \A_n$, the result holds also for $E^\la_{-\eps}$.
\end{proof}

\begin{lemma}\label{LM3S}
Let $p\geq 3$, $n\geq 6$ and $\la\in\RP_p(n)\setminus\{\be_n\}$. If $V$ is a simple $F\widetilde\s_n$- or $F\widetilde \A_n$-module indexed by $\la$ then there exists $\phi:M_3\to\End_F(V)$ which does not vanish on $S_3$.
\end{lemma}

\begin{proof}
For $p\geq 5$ this holds by \cite[Theorem 7.2]{kt} (and its proof). So we may assume that $p=3$.

From Lemma \ref{L6} it is enough to prove that $x_3V\not=0$. From Lemma \ref{L2.4}, since $\la\not=\be_n$, there exists a composition factor of $V\da_{\widetilde \A_6}$ of the form $E((4,2),\pm)$. So it is enough to prove that $x_3E((4,2),\pm)\not=0$. Let $W((6),0)$ be the reduction modulo 3 of the basic spin module of $\widetilde \A_6$ in characteristic 0 and $W((4,2),\pm)$ be the reduction modulo 3 of the simple spin modules of $\widetilde \A_6$ indexed by $(4,2)$ in characteristic 0. Let $\chi^{(6),0}$ and $\chi^{(4,2),\pm}$ be the characters of the reduction modulo 3 of $W((6),0)$ and $W((4,2),\pm)$ respectively. Using decomposition matrices and Lemma \ref{L54} it can be checked that the characters of $E((4,2),\pm)$ (over the field $F$) are $\chi^\pm=\chi^{(4,2),\pm}-\chi^{(6),0}$. In order to prove that $x_3E((4,2),\pm)\not=0$ it is enough to prove that $\chi^\pm(x_3y)\not=0$ for some $y\in\tilde\A_6$. Let $y:=\widetilde{(1,5,2,3)(4,6)}$. It can be computed that, up to exchange of $y$ with $zy$, $x_3y$ is given by
\begin{align*}
&z^{\ldots}\widetilde{(1,5,3,2)(4,6)}+z^{\ldots}\widetilde{(1,5)(4,6)}+z^{\ldots}\widetilde{(1,3)(2,4,6,5)}+z^{\ldots}\widetilde{(1,4,6,3)(2,5)}\\
&+\widetilde{(1,5,6,2,3)}+z\widetilde{(1,5,4,2,3)}+\widetilde{(1,6,4)(2,3,5)}+\widetilde{(2,3,6,4,5)}\\
&-z\widetilde{(1,5,3)(2,4,6)}-z\widetilde{(1,5,4,6,3)}-z^{\ldots}\widetilde{(1,3,5,2)(4,6)}-z^{\ldots}\widetilde{(2,5)(4,6)}\\
&-z^{\ldots}\widetilde{(1,5,6,4)(2,3)}-z^{\ldots}\widetilde{(1,5)(2,3,6,4)}-z\widetilde{(1,6,5,2,3)}-\widetilde{(1,4,5,2,3)}.
\end{align*}
Let $C^+$ and $C^-$ be the two conjugacy classes of lifts in $\tilde\A_6$ of elements of $\A_6$ with cycle partition $(4,2)$. Up to exchange of $C^+$ and $C^-$, it can be computed that the lifts of $(1,5,3,2)(4,6)$, $(1,3)(2,4,6,5)$, $(1,3,5,2)(4,6)$ and $(1,5)(2,3,6,4)$ appearing in $x_3y$ are all in $C^+$, while those of $(1,4,6,3)(2,5)$ and $(1,5,6,4)(2,3)$ are in $C^-$. Since all lifts of elements of the form $(a,b)(c,d)$ are conjugated in $\widetilde \A_6$, it then follows from $\chi^\pm=\chi^{(4,2),\pm}-\chi^{(6),0}$ that
\[\chi^{\pm}(x_3y)=2\chi^\pm(\widetilde{(1,2,3,4,5)})+2\chi^\pm(\widetilde{(1,2,3)(4,5,6)})\equiv 2\Md 3,\]
so the lemma holds.
\end{proof}

\begin{lemma}\label{L52}
Let $p=3$, $n\geq 8$, $G\in\{\s_n,\A_n\}$ and $\la=(\la_1,\la_2)\in\Par_3(n)$ with $\la_1\geq\la_2+2\geq 5$. Let $V$ be an irreducible $FG$-module indexed by $\la$. Then there exists $\psi\in\Hom_G(M_{3,1^2},\End_F(V))$ which does not vanish on $S_{3,1^2}$.
\end{lemma}

\begin{proof}
By \cite[Lemma 1.8]{ks2}, $\la\not=\la^\Mull$, so $V\cong D^\la$ or $E^\la$. So it is enough to prove the lemma for $G=\s_n$. From Lemma \ref{M311} it is enough to prove that $x_{3,1^2}D^\la\not=0$. Throughout this proof we will consider $x_{3,1^2}$ as an element of $F\s_8$ instead of $F\widetilde\s_8$ by sending $\widetilde g$ to $g$. Note that by Lemma \ref{Lemma39} and \cite[Tables]{JamesBook}, $D^{(5,3)}\cong S^{(5,3)}$ is a composition factor of $D^\la\da_{\s_8}$. Let $\chi$ be the character of $S^{(5,3)}$. Let $y:=(2,6,8,3,4)$. In order to prove that $x_{3,1^2}D^\la\not=0$ it is enough to prove that $\chi(yx_{3,1^2})\not=0$. Note that $yx_{3,1^2}=X_+-X_-$ where
\begin{align*}
X_+&=y\sum_{g\in \A_{4,2^2}}\sum_{h\in\s_{\{2,6,8\}}}gh{(2,6,8,3,4)}h^{-1}g^{-1},\\
X_-&=y\sum_{g\in \s_{4,2^2}\setminus \A_{4,2^2}}\sum_{h\in\s_{\{2,6,8\}}}gh{(2,6,8,3,4)}h^{-1}g^{-1}.
\end{align*}
It can be computed with GAP \cite{gap} that the number of elements appearing $X_\pm$ corresponding to each conjugacy class of $\s_8$ is as follows ($X_\pm\in F \A_8$ so that not all conjugacy classes have to be considered):
\[\begin{array}{l|c|c|c|c|c|c}
\mbox{cycle type}&(1^8)&(2^2,1^4)&(2^4)&(3,1^5)&(3,2^2,1)&(3^2,1^2)\\
\hline
X_+&0&18&0&15&32&11\\
X_-&2&13&0&10&12&46
\end{array}\]
\[\begin{array}{l|c|c|c|c|c|c}
\mbox{cycle type}&(4,2,1^2)&(4^2)&(5,1^3)&(5,3)&(6,2)&(7,1)\\
\hline
X_+&27&4&53&22&8&98\\
X_-&67&24&36&12&18&48,
\end{array}\]
from which it easily follows that $\chi(yx_{3,1^2})\equiv 2\Md 3$.
\end{proof}

\begin{lemma}\label{L51}
Let $p=3$, $n\geq 8$ and $\la\in\RPar_3(n)\setminus\{\be_n\}$. If $V$ is an irreducible spin representation of $G\in\{\widetilde\s_n,\widetilde \A_n\}$ indexed by $\la$ then there exists $\psi\in\Hom_G(M_{3,1^2},\End_F(V))$ which does not vanish on $S_{3,1^2}$.
\end{lemma}

\begin{proof}
Assume first that $G=\widetilde\s_n$. By Lemma \ref{L2.4} there exists a composition factor of $D(\la,\de)\da_{\widetilde\s_8}$ of the form $D(\mu,\eps)$ with $\mu\in\{(5,2,1),(4,3,1)\}$. Let $\chi$ be the character of $D(\mu,\eps)$ and $\chi^{(8),\pm}$, $\chi^{(6,2),0}$ and $\chi^{(7,1),0}$ be the characters of the reduction modulo 3 of the simple spin modules in characteristic 0 indexed by the corresponding partitions. Then $\chi\in\{1/2\chi^{(6,2),0}-\chi^{(8),\pm},\chi^{(7,1),0}\}$ using decomposition matrices and Lemma \ref{L54}.

In order to prove the lemma for $\widetilde\s_n$ it is enough by Lemma \ref{M311} to prove that $x_{3,1^2}D(\mu,\eps)\not=0$. Let $y:=\widetilde{(2,6,8,3,4)}$ and 
\begin{align*}
X_+=&y\sum_{g\in \A_{4,2^2}}\sum_{h\in\s_{\{2,6,8\}}}\widetilde g\widetilde h\widetilde{(2,6,8,3,4)}(\widetilde h)^{-1}(\widetilde g)^{-1},\\
X_-=&y\sum_{g\in \s_{4,2^2}\setminus \A_{4,2^2}}\sum_{h\in\s_{\{2,6,8\}}}\widetilde g\widetilde h\widetilde{(2,6,8,3,4)}(\widetilde h)^{-1}(\widetilde g)^{-1}.
\end{align*}
Note that $yx_{3,1^2}=X_+-X_-$. It can be computed that the number of elements appearing $X_\pm$ corresponding to each conjugacy class of $\widetilde\s_8$ is as follows:
\[\begin{array}{l|c|c|c|c|c|c}
\mbox{cycle type}&(1^8)&(1^8)&(3,1^5)&(3,1^5)&(3^2,1^2)&(3^2,1^2)\\
\mbox{order of el.}&1&2&3&6&3&6\\
\hline
X_+&0&0&4&11&7&4\\
X_-&2&0&4&6&32&14
\end{array}\]
\[\begin{array}{l|c|c|c|c|c|c|c}
\mbox{cycle type}&(5,1^3)&(5,1^3)&(5,3)&(5,3)&(7,1)&(7,1)&\mbox{others}\\
\mbox{order of el.}&5&10&15&30&7&14&\\
\hline
X_+&11&42&22&0&62&36&89\\
X_-&9&27&12&0&42&6&134.
\end{array}\]
Since $X_\pm\in F\widetilde \A_8$, it easily follows that $\chi(yx_{3,1^2})\equiv 1\Md 3$. So the lemma holds for $\widetilde\s_n$. Assume now that $G=\widetilde \A_n$. If $V\cong E(\la,0)$ then $V\cong D(\la,\pm)\da_{\widetilde \A_n}$. So in this case the lemma holds by the previous part. If $V\cong E(\la,\pm)$ the lemma can be proved similarly to Lemma \ref{LM3}.
\end{proof}

\begin{lemma}\label{L3}
Let $p\geq 3$, $n\geq 6$, $\la\in\Par_p(n)\setminus\h_p(n)$ and $G\in\{\s_n,\A_n\}$. Let $V$ be an $G$-module indexed by $\la$. Then there exists a non-zero $\psi\in\Hom_G(S_{1^3},\End_F(V))$. If $p\not=3$ then $\psi$ extends to $\phi\in\Hom_G(M_{1^3},\End_F(V))$.
\end{lemma}

\begin{proof}
By Lemma \ref{L1a} it is enough to prove that $x_{1^3}V\not=0$. We will consider $x_{1^3}$ as an element of $F \A_n$. By Lemma \ref{L34} it is enough to prove that $x_{1^3}E\not=0$ for all irreducible modules $E$ of $\A_6$ indexed by $\mu\in\Par_p(6)\setminus\h_p(6)$. So we may assume that $E\in\{E^{(4,2)},E^{(3^2)},E^{(3,2,1)}_\pm\}$ if $p>5$, $E\in\{E^{(4,2)},E^{(3^2)}\}$ if $p=5$ or $E=E^{(4,2)}$ if $p=3$. Each of theses modules is just the reduction modulo $p$ of the characteristic 0 module indexed by the same partition, so that characters are known.

Note that $x_{1^3}E\not=0$ if and only if $x_{1^3}(1,2,3)E\not=0$. It can be computed that $\pm x_{1^3}(1,2,3)$ is equal to
\[(1,3)(2,4)+(1,2)(3,4)+(1,4)(2,3)+1-(1,4,3)-(1,2,4)-(2,3,4)-(1,3,2).\]
If $\chi$ is the character of $E$ it then follows that $\chi(x_{1^3}(1,2,3))=\pm 12\not\equiv 0\Md p$ if $p\geq 5$. So assume that $p=3$. It can be computed that $\pm x_{1^3}(2,6,3,5,4)$ is equal to
\begin{align*}
&(2,5,4,6,3)+(1,2,6,3)(4,5)+(1,6,3,5,4)+(1,5,4,2)(3,6)\\
&-(4,6)(4,5)-(1,5,4)(2,6,3)-(1,2)(3,5,4,6)-(1,6,3)(2,5,4)
\end{align*}
and so $\chi(x_{1^3}(2,6,3,5,4))=\pm 2\not\equiv 0\Md 3$. The lemma then follows.
\end{proof}

\begin{lemma}\label{L4}
Let $p\geq 3$, $n\geq 4$, $G\in\{\widetilde \s_n,\widetilde \A_n\}$ and $\la\in\RP_p(n)\setminus\{\be_n\}$. If $V$ is a spin irreducible representation of $G$ indexed by $\la$ then there exists a non-zero $\psi\in\Hom_G(S_{1^3},\End_F(V))$. If $p\not=3$ then $\psi$ extends to $\phi\in\Hom_G(M_{1^3},\End_F(V))$.
\end{lemma}

\begin{proof}
From \cite[Lemma 2.4]{kt} we have that if $m\geq 6$ and $\mu\in\RP_p(m)\setminus\{\be_m\}$, then $D(\mu)\da_{\widetilde\s_{m-1}}$ has a composition factor which is not basic spin.

Assume first that $p\geq 5$. In this case it can then be easily checked that $V\da_{\widetilde \A_4}$ has a composition factor $E((3,1),\pm)$. Let $g:=\widetilde{(1,2,3)}$. Up to exchange of $C_{3}^\pm$ we have that
\begin{align*}
gx_{1^3}=&1+z^{\ldots}\widetilde{(1,2)(3,4)}+z^{\ldots}\widetilde{(1,3)(2,4)}+z^{\ldots}\widetilde{(1,4)(2,3)}\\
&-z\widetilde{(1,4,3)}-z\widetilde{(1,2,4)}-z\widetilde{(2,3,4)}-\widetilde{(1,3,2)}.
\end{align*}
Since $p\geq 5$, $E((3,1))\cong S((3,1))$ by Lemma \ref{L54}. If $\chi$ is the character of $E((3,1))$ we then have that $\chi(gx_{1^3})=\pm 6$ and so the action of $x_{1^3}$ on at least one between $E((3,1),+)$ and $E((3,1),-1)$ is non-zero.

If $G=\widetilde\s_n$ both $E((3,1),+)$ and $E((3,1),-)$ are composition factors of $V\da_{\widetilde\A_4}$, so by Lemma \ref{L1a} there exists a non-zero $\psi\in\Hom_G(S_{1^3},\End_F(V))$.

If $G=\widetilde\A_n$ and $V=E(\la,0)$ then $V=D(\la,\pm)\da_{\widetilde\A_n}$, so again both $E((3,1),+)$ and $E((3,1),-)$ are composition factors of $V\da_{\widetilde\A_4}$ and we can conclude as in the previous case.

If $G=\widetilde \A_n$ and $V=E(\la,\pm)$ then $E((3,1),\pm)$ is a composition factor of $E(\la,+)\da_{\widetilde \A_4}$ if and only if $E((3,1),\mp)$ is a composition factor of $E(\la,-)\da_{\widetilde\A_4}$, so by Lemma \ref{L1a} there exists a non-zero $\psi_+\in\Hom_{\widetilde \A_n}(S_{1^3},\End_F(E(\la,+)))$ or a non-zero $\psi_-\in\Hom_{\widetilde \A_n}(S_{1^3},\End_F(E(\la,-)))$. Since $S_{1^3}\cong S_{1^3}^\si$ and $E(\la,\pm)\cong E(\la,\mp)^\si$ for $\si\in\s_n\setminus\A_n$, it follows that there exists a non-zero $\psi_+\in\Hom_{\widetilde \A_n}(S_{1^3},\End_F(V)))$.

Assume now that $p=3$. Then $n\geq 5$ and $V\da_{\widetilde \A_5}$ has a composition factor $E((4,1),0)$. Let $g:=\widetilde{(1,2)(4,5)}$. Then, up to exchange of $C_{3}^\pm$,
\begin{align*}
gx_{1^3}=&\widetilde{(1,2,3,5,4)}+\widetilde{(1,5,4,3,2)}+z\widetilde{(2,5,4)}+z^{\ldots}\widetilde{(1,3)(4,5)}\\
&-z\widetilde{(1,2,5,4,3)}-z\widetilde{(1,3,5,4,2)}-\widetilde{(1,5,4)}-z^{\ldots}\widetilde{(2,3)(4,5)}.
\end{align*}
By Lemma \ref{L54} and using decomposition matrices, $E((4,1))\cong S((4,1))$. If $\chi$ is the character of $E((4,1),0)$ then $\chi(gx_{1^3})=\pm 4$, from which the lemma follows also in this case by Lemma \ref{L1a}.
\end{proof}

\begin{lemma}\label{L36}
Let $n\geq 6$ and $G\in\{\s_n,\A_n\}$. Assume that $p\geq 5$ and $\la\in\Par_p(n)\setminus\h_p(n)$ with $h(\la),h(\la^\Mull)\geq 3$ or that $p=3$ and $\la\in\Par_3(n)$ with $h(\la),h(\la^\Mull)\geq 4$.  Let $V$ be an $G$-module indexed by $\la$. Then there exists a non-zero $\psi\in\Hom_G(S_{1^5},\End_F(V))$. If $p\not=5$ then $\psi$ extends to $\phi\in\Hom_G(M_{1^5},\End_F(V))$.
\end{lemma}

\begin{proof}
If $p\geq 5$ and $n\geq 9$ then by Lemma \ref{L34} there exists $\mu\in\Par_p(9)\setminus\h_p(9)$ with $h(\mu),h(\mu^\Mull)\geq 3$ such that $D^\mu$ is a composition factor of $D^\la\da_{\s_9}$. It can then be easily checked using Lemma \ref{Lemma39} and decomposition matrices that if $p\geq 7$ then $D^{(3,2,1)}$ is a composition factor of $D^\la\da_{\s_6}$, while if $p=5$ then $n\geq 7$ and $E^{(4,2,1)}$ is a composition factor of $D^\la\da_{\A_7}$. If $p=3$ then $n\geq 8$ and by \cite[Lemma 4.13]{m3}, $D^{(4,2,1^2)}$ is a composition factor of $D^\la\da_{\s_8}$. Each of theses modules is just the reduction modulo $p$ of the characteristic 0 module indexed by the same partition (for example looking at decomposition matrices), so characters are known.

Consider $x_{1^5}$ and $C_{5}^\pm$ upon projection to $\A_n$.

If $p\geq 7$ let $g:=(1,2,3,4,5)$. Then, up to exchange of $C_{5}^\pm$, we have that the number of elements of $C_{5}^{\pm}g$ in each conjugacy class of $\s_6$ is as follows:
\[\begin{array}{l|c|c|c|c|c|c|c}
\text{cycle type}&(1^6)&(2^2,1^2)&(3,1^3)&(3^2)&(4,2)&(5,1)&\text{others}\\
\hline
C_{5}^+g&1&10&5&5&20&31&0\\
C_{5}^-g&0&10&10&10&20&22&0.
\end{array}\]
If $\chi$ is the character of $D^{(3,2,1)}$ it then follows that $\chi(x_{1^5}g)=\pm 45$.

If $p=5$ let $g:=(2,7,4)(3,6,5)$. Then, up to exchange of $C_{5}^\pm$, we have that the number of elements of $C_{5}^{\pm}g$ in each conjugacy class of $\s_7$ is as follows:
\[\begin{array}{l|c|c|c|c|c|c|c|c}
\hspace{-0.5pt}\text{cycle type}\hspace{-0.5pt}&\hspace{-0.5pt}(2^2,1^3)\hspace{-0.5pt}&\hspace{-0.5pt}(3,1^4)\hspace{-0.5pt}&\hspace{-0.5pt}(3,2^2)\hspace{-0.5pt}&\hspace{-0.5pt}(3^2,1)\hspace{-0.5pt}&\hspace{-0.5pt}(4,2,1)\hspace{-0.5pt}&\hspace{-0.5pt}(5,1^2)\hspace{-0.5pt}&\hspace{-0.5pt}(7)\hspace{-0.5pt}&\hspace{-0.5pt}\text{others}\hspace{-0.5pt}\\
\hline
C_{5}^+g&0&3&6&3&27&9&24&0\\
C_{5}^-g&3&0&12&6&9&18&24&0.
\end{array}\]
If $\chi$ is the character of $E^{(4,2,1)}$ it then follows that $\chi(x_{1^5}g)=\pm 9$.

If $p=3$ and $h(\la),h(\la^\Mull)\geq 4$ let $g=(1,2,3,4,5)(6,7,8)$. Then, up to exchange of $C_{5}^\pm$, we have that the number of elements of $C_{5}^{\pm}g$ in each conjugacy class of $\s_8$ is as follows:
\[\begin{array}{l|c|c|c|c|c}
\text{cycle type}&(3,1^5)&(3,2^2,1)&(3^2,1^2)&(4,2,1^2)&(4^2)\\
\hline
C_{5}^+g&0&5&5&5&10\\
C_{5}^-g&1&0&5&10&5
\end{array}\]
\[\begin{array}{l|c|c|c|c|c}
\text{cycle type}&(5,1^3)&(5,3)&(6,2)&(7,1)&\text{others}\\
\hline
C_{5}^+g&5&12&10&20&0\\
C_{5}^-g&0&11&15&25&0.
\end{array}\]
If $\chi$ is the character of $D^{(4,2,1^2)}$ it can be easily checked that $\chi(x_{1^5}g)=\pm 5$.

For $\s_n$ the lemma then follows. For $\A_n$ it holds similarly to Lemma \ref{LM3}.
\end{proof}

\begin{lemma}\label{L14}
Let $p\geq 3$, $n\geq 6$, $G\in\{\widetilde\s_n,\widetilde \A_n\}$ and $\la\in\RP_p(n)\setminus\{\be_n\}$ with $\la_1\geq 5$. If $V$ is an irreducible spin representation of $G$ indexed by $\la$, then there exists a non-zero $\psi\in\Hom_G(S_{1^5},\End_F(V))$. If $p\not=5$ then $\psi$ extends to $\phi\in\Hom_G(M_{1^5},\End_F(D))$.
\end{lemma}

\begin{proof}
If $p\geq 7$ and $n\geq 11$ then by Lemma \ref{L2.4} there exists $\mu\in\RP_p(11)\setminus\{\be_{11}\}$ such that $D(\mu)$ is a composition factor of $D(\la)\da_{\widetilde\s_{11}}$. It can then be easily checked using Lemma \ref{Lemma39s} that $D(\la)\da_{\widetilde\s_6}$ has a composition factor $D((5,1),0)$. Let $g=\widetilde{(1,2,3,4,5)}$. Up to exchange of $C_{5}^\pm$, we have that the number of elements of $C_{5}^\pm g$ in each conjugacy class of $\widetilde \s_6$ is as follows:
\[\begin{array}{l|c|c|c|c|c|c|c|c|c}
\hspace{-0.2pt}\text{cycle type}\hspace{-0.2pt}&\hspace{-0.2pt}(1^6)\hspace{-0.2pt}&\hspace{-0.2pt}(1^6)\hspace{-0.2pt}&\hspace{-0.2pt}(3,1^3)\hspace{-0.2pt}&\hspace{-0.2pt}(3,1^3)\hspace{-0.2pt}&\hspace{-0.2pt}(3^2)\hspace{-0.2pt}&\hspace{-0.2pt}(3^2)\hspace{-0.2pt}&\hspace{-0.2pt}(5,1)\hspace{-0.2pt}&\hspace{-0.2pt}(5,1)\hspace{-0.2pt}&\hspace{-0.2pt}\text{others}\hspace{-0.2pt}\\
\hspace{-0.2pt}\text{order of el.}\hspace{-0.2pt}&1&2&3&6&3&6&5&10&\\
\hline
C_{5}^+g&0&0&5&5&5&5&2&20&30\\
C_{5}^-g&1&0&0&5&0&5&11&20&30.
\end{array}\]
Let $\chi$ be the character of $D((5,1),0)\cong S((5,1),0)$ (by Lemma \ref{L54} since $p>6$). Then $\chi(x_{1^5}g)=\pm 45\not=0$.

Assume next that $p=3$ or $p=5$ and $\la_1\geq 6$. Then $n\geq 7$. If $p=3$ then $E((5,2),0)$ is a composition factor of $V\da_{\widetilde \A_7}$ by Lemma \ref{Lemma39s} by always removing the bottom normal node for which the obtained partition is in $\RP_p(m)$. If $p=5$ and $\la_1\geq 6$ then similarly $E((6,1),0)$ is a composition factor of $V\da_{\widetilde \A_7}$. Again by by Lemma \ref{L54} and decomposition matrices, $E((5,2),0)$ for $p=3$ and $E((6,1),0)$ for $p=5$ are just the reduction modulo $p$ of the corresponding characteristic 0 modules. Let $g=\widetilde{(2,3)(4,5,6,7)}$. Up to exchange of $C_{5}^\pm$ and choice of $g$, we have that the number of elements of $C_{5}^\pm g$ in each conjugacy class of $\widetilde \s_7$ is as follows:
\[\begin{array}{l|c|c|c|c|c|c}
\text{cycle type}&(1^7)&(1^7)&(3,1^4)&(3,1^4)&(3^2,1)&(3^2,1)\\
\text{order of el.}&1&2&3&6&3&6\\
\hline
C_{5}^+g&0&0&1&0&5&4\\
C_{5}^-g&0&0&0&1&4&5
\end{array}\]
\[\begin{array}{l|c|c|c|c|c}
\text{cycle type}&(5,1^2)&(5,1^2)&(7)&(7)&\text{others}\\
\text{order of el.}&5&10&7&14&\\
\hline
C_{5}^+g&6&5&14&10&27\\
C_{5}^-g&5&6&10&14&27.
\end{array}\]
If $p=3$ and $\chi$ is the character of $E((5,2),0)$ then $\chi(x_{1^5}g)=\pm 10$. If $p=5$ and $\chi$ is the character of $E((6,1),0)$ then $\chi(x_{1^5}g)=\pm 18$.

Last assume that $p=5$ and $\la_1=5$. Then $n\geq 8$. If $n\geq 11$ and $D(\la)\da_{\widetilde\s_{11}}$ has a composition factor $D(\mu)$ with $\mu_1\geq 6$ we can apply the previous paragraph. So we may assume this is not the case. Then by Lemma \ref{L2.4} if $n\geq 11$ then $D(\la)\da_{\widetilde\s_{11}}$ has a composition factor $D((5,3,2,1))$ or $D((5,4,2))$. It can then be checked (also when $n\leq 10$) that $D((5,2,1),0)$ is a composition factor of $D(\la)\da_{\widetilde\s_8}$. Let $g:=\widetilde{(2,3)(4,5,7)(6,8)}$. Up to exchange of $C_{5}^\pm$ and choice of $g$, we have that the number of elements of $C_{5}^\pm g$ in each conjugacy class of $\widetilde \s_8$ is as follows:
\[\begin{array}{l|c|c|c|c|c|c}
\text{cycle type}&(1^8)&(1^8)&(3,1^5)&(3,1^5)&(3^2,1^2)&(3^2,1^2)\\
\text{order of el.}&1&2&3&6&3&6\\
\hline
C_{5}^+g&0&0&0&0&2&0\\
C_{5}^-g&0&0&0&0&0&2
\end{array}\]
\[\begin{array}{l|c|c|c|c|c|c|c}
\text{cycle type}&(5,1^3)&(5,1^3)&(5,3)&(5,3)&(7,1)&(7,1)&\text{others}\\
\text{order of el.}&5&10&15&30&7&14&\\
\hline
C_{5}^+g&0&2&2&8&10&12&36\\
C_{5}^-g&2&0&8&2&12&10&36.
\end{array}\]
If $\chi$ is the character of $D((5,2,1),0)$ then $\chi(x_{1^5}g)=\pm 4$ (using decomposition matrices and Lemma \ref{L54} it can be checked that $D((5,2,1),0)$ is the reduction modulo 5 of either module indexed by $(5,2,1)$ in characteristic 0).

The lemma then follows for $\widetilde\s_n$. For $\widetilde \A_n$ it follows similarly to the proof of Lemma \ref{LM3}.
\end{proof}

In the next section we will study the structure of certain permutation modules. In \S\ref{s2r} to \S\ref{sbs} we will then study more in details most classes of modules for which some of the results in this section do not apply and obtain similar results on the endomorphisms rings of those modules. These results will then be used in \S\ref{snat} to \S\ref{sbssbs} to study tensor products of certain special classes of modules.

\section{Permutation modules}\label{s4}

In order to extend the results obtained in the previous section to (some) of the classes of families which were not considered, we will need to study permutation modules more in detail and then study restrictions of some classes of irreducible modules to certain subgroups. We start here by considering the structure of certain permutation modules.

The following three lemmas on the structure of $M^\la$ for certain 2-rows partitions $\la$ follow easily from \cite[17.17,24.15]{JamesBook} and \cite[6.1.21,2.7.41]{jk}.

\begin{lemma}\label{Mk}
Let $1\leq k<p$. Then $M_k\sim S_k|M_{k-1}$.
\end{lemma}

 \begin{lemma}\label{M1}
 Let $p\geq 3$ and $n\geq 2$. If $p\nmid n$ then $M_1\cong D_0\oplus D_1$, while if $p\mid n$ then $M^{(n-1,1)}\cong D_0|D_1|D_0$.
 \end{lemma}

\begin{lemma}\label{L160817_0}\label{L160817_1}\label{L160817_2}
Let $p=3$ and $n\geq 4$. Then
\[M_2\cong\left\{\begin{array}{ll}
M_1\oplus D_2,&n\equiv 0\Md 3,\\
D_1\oplus (D_0|D_2|D_0),&n\equiv 1\Md 3,\\
D_0\oplus (D_1|D_2|D_1),&n\equiv 2\Md 3.
\end{array}\right.\]
\end{lemma}

We will also need information about the structure of certain permutation modules corresponding to subgroups $\s_{n-k}$.

\begin{lemma}\label{L17}
Let $p\geq 3$ and $n\not\equiv 0\Md p$. If $n\geq 2$ then
\[M_1\cong D_1\oplus M_0.\]
If $n\geq 4$ then
\[M_{1^2}\oplus M_0\cong D_{1^2}\oplus M_2\oplus M_1.\]
If $p\geq 5$ and $n\geq 6$ then
\[M_{1^3}\oplus M_3\oplus M_2\oplus M_1\cong D_{1^3}\oplus M_{2,1}^{\oplus 2}\oplus M_{1^2}\oplus M_0.\]
\end{lemma}

\begin{proof}
From Lemma \ref{LH} we have that in each of the above cases $D_{1_k}\cong S_{1_k}\subseteq M_{1^k}$. Since $M_{1^k}$ and $D_{1^k}$ are self-dual, $D_{1^k}$ is also a quotient of $M_{1^k}$. Let $V\subseteq M_{1^k}$ with $V\cong D_{1^k}$ and $W\subseteq M_{1^k}$ with $M_{1^k}/W\cong D_{1^k}$. By \cite[12.1]{JamesBook} we  have that $[M_{1^k}:D_{1^k}]=1$, so $V\not\subseteq W$. By dimension we then have that $M_{1^k}=V\oplus W$, so $D_{1^k}\cong S_{1^k}$ is a direct summand of $M_{1^k}$. The lemma then follows by comparing composition factors (for example using Specht filtrations from \cite[17.14]{JamesBook}) and Lemma \ref{LYoung}, since if $\la\unrhd (n-k,1^k)$ and $k<p$ then $\la\in\Par_p(n)$.
\end{proof}

\begin{lemma}\label{L16}
Let $p\geq 3$ and $n\equiv 0\Md p$. If $n\geq 2$ then
\[M_1\cong Y_1\]
and if $n\geq 4$
\[M_{1^2}\cong M_2\oplus Y_2.\]
If $p\geq 5$ and $n\geq 6$ then
\[M_{1^3}\oplus M_3\cong M_{2,1}^{\oplus 2}\oplus Y_3\]
and if $n\geq 8$
\[M_{1^4}\oplus M_{2^2}\oplus M_{3,1}^{\oplus 2}\cong M_{2,1^2}^{\oplus 2}\oplus M_4\oplus Y_4.\]
If $p=3$ and $n\geq 6$ then
\[M_{1^3}\oplus M_1\cong M_{2,1}\oplus M_{1^2}\oplus Y_3'.\]

In each of the above cases $Y_k$ or $Y_k'$ 
is indecomposable with simple head and socle isomorphic to $D_{1^{k-1}}$ and
\begin{align*}
Y_1&\cong \overbrace{D_0|D_1}^{S_1}|\overbrace{D_0}^{S_0},\\
Y_2&\sim \overbrace{D_1|D_{1^2}}^{S_{1^2}}|\overbrace{D_0|D_1}^{S_1},\\
Y_3&\sim \overbrace{D_{1^2}|D_{1^3}}^{S_{1^3}}|\overbrace{D_1|D_{1^2}}^{S_{1^2}},\\
Y_3'&\sim S_{1^3}|S_{2,1}|S_{1^2},\\
Y_4&\sim \overbrace{D_{1^3}|D_{1^4}}^{S_{1^4}}|\overbrace{D_{1^2}|D_{1^3}}^{S_{1^3}}.
\end{align*}
\end{lemma}

\begin{proof}
Note that $M_{1^k}=M^{(n-k,1^{k-1})}\ua^{\s_n}$. In particular in each of the above cases since $(n-k,1^{k-1})\in\Par_p(n-1)$ from Lemmas \ref{Lemma45} and \ref{LH} and self-duality of $M^{(n-k,1^{k-1})}$ we have that $D^{(n-k,1^{k-1})}\cong S^{(n-k,1^{k-1})}$ and that $e_{-k}D^{(n-k,1^{k-1})}\ua^{\s_n}$ is a direct summand of $M_{1^k}$. Let $Y_k$ or $Y_k'$ 
be this direct summand. Then $Y_k$ or $Y_k'$ 
has simple head and socle isomorphic to $D_{1^{k-1}}$ by Lemma \ref{Lemma39} and it has the right Specht filtration by \cite[Corollary 17.14]{JamesBook} and block decomposition. Structure of hook Specht modules can be obtained by Lemma \ref{LH}.

The lemma then follows by comparing composition factors (for example using Specht filtrations) and Lemma \ref{LYoung}, since $\la\in\Par_p(n)$ if $\la\rhd (n-k,1^k)$ and $k\leq p$.
\end{proof}

\section{More on endomorphisms rings}\label{s5}

In this section we study branching for certain classes of modules in order to extend in many cases results from \S\ref{shr} to families of modules which were not considered there. We divide this section according to different classes of modules.

\subsection{Partitions with two or three rows}\label{s2r}

\begin{lemma}\label{L7a}
Let $p=3$, $n\geq 7$, $G\in\{\s_n,\A_n\}$ and $\la=(n-2,2)$. Let $V$ be an irreducible $FG$-module indexed by $\la$. If $n\not\equiv 2\Md 3$ then there exists $\psi\in\Hom_H(M_{3},\End_F(V))$ which does not vanish on $S_{3}$.
\end{lemma}

\begin{proof}
By \cite[Lemma 1.8]{ks2}, $\la\not=\la^\Mull$, so $V\cong D^\la=D_2$ or $E^\la$. So it is enough to prove the lemma for $\s_n$. From \cite[Lemma 6.5]{m3} it is enough to prove that
\[\dim\End_{\s_{n-3,3}}(D_2\da_{\s_{n-3,3}})>\dim\End_{\s_{n-2,2}}(D_2\da_{\s_{n-2,2}}).\]
Note that the assumption on $n$ is equivalent to $(n-2,2)$ not being a JS-partition.

If the two removable nodes have different residue this holds by \cite[Lemma 6.7]{m3}. So we may assume that the removable nodes have the same residue, in which case $n\equiv 0\Md 3$. From Mackey induction-reduction theorem we have that
\begin{align*}
M_1\da_{\s_{n-2,2}}&\cong 1\ua_{\s_{n-2,1^2}}^{\s_{n-2,2}}\oplus 1\ua_{\s_{n-3,1,2}}^{\s_{n-2,2}}\\
&\cong (M^{(n-2)}\boxtimes M^{(1^2)})\oplus (M^{(n-3,1)}\boxtimes M^{(2)}),\\
M_1\da_{\s_{n-3,3}}&\cong 1\ua_{\s_{n-3,2,1}}^{\s_{n-3,3}}\oplus 1\ua_{\s_{n-4,1,3}}^{\s_{n-3,3}}\\
&\cong (M^{(n-3)}\boxtimes M^{(2,1)})\oplus (M^{(n-4,1)}\boxtimes M^{(3)}),\\
M_2\da_{\s_{n-2,2}}&\cong 1\oplus 1\ua_{\s_{n-3,1^3}}^{\s_{n-2,2}}\oplus 1\ua_{\s_{n-4,2^2}}^{\s_{n-2,2}}\\
&\cong (M^{(n-2)}\boxtimes M^{(2)})\oplus (M^{(n-3,1)}\boxtimes M^{(1^2)})\oplus (M^{(n-4,2)}\boxtimes M^{(2)}),\\
M_2\da_{\s_{n-3,3}}&\cong 1\ua_{\s_{n-3,2,1}}^{\s_{n-3,3}}\oplus 1\ua_{\s_{n-4,1,2,1}}^{\s_{n-3,3}}\oplus 1\ua_{\s_{n-5,2,3}}^{\s_{n-3,3}}\\
&\cong (M^{(n-3)}\boxtimes M^{(2,1)})\oplus (M^{(n-4,1)}\boxtimes M^{(2,1)})\oplus (M^{(n-5,2)}\boxtimes M^{(3)}).
\end{align*}
From Lemma \ref{L160817_0} we have that $M_2\cong M_1\oplus D_2$. Comparing $M_2\da_H$ and $M_1\da_H$ for $H\in\{\s_{n-2,2},\s_{n-3,3}\}$ using Lemmas \ref{M1} and \ref{L160817_2}, it follows that
\begin{align*}
D_2\da_{\s_{n-2,2}}&\hspace{-1.5pt}\!\cong\!\hspace{-1pt} (\hspace{-0.5pt}D^{(n-3,1)}\!\boxtimes\! (\hspace{-0.5pt}D^{(2)}\!\oplus\! D^{(1^2)}\hspace{-0.5pt})\hspace{-1pt})\!\oplus\! (\hspace{-1.5pt}(\hspace{-0.5pt}D^{(n-2)}|D^{(n-4,2)}|D^{(n-2)}\hspace{-0.5pt})\!\boxtimes\! D^{(2)}\hspace{-0.5pt}),\\
D_2\da_{\s_{n-3,3}}&\hspace{-1.5pt}\!\sim\!\hspace{-1pt} (\hspace{-0.5pt}D^{(n-5,2)}\!\boxtimes\! D^{(3)}\hspace{-0.5pt})\!\oplus\! (\hspace{-1.5pt}(\hspace{-0.5pt}D^{(n-3)}|D^{(n-4,1)}|D^{(n-3)}\hspace{-0.5pt})\!\boxtimes\! (\hspace{-0.5pt}D^{(3)}|D^{(2,1)}|D^{(3)}\hspace{-0.5pt})\hspace{-1.5pt}).
\end{align*}
It then follows that 
\[\dim\End_{\s_{n-3,3}}(D_2\da_{\s_{n-3,3}})=5>4=\dim\End_{\s_{n-2,2}}(D_2\da_{\s_{n-2,2}}).\]
\end{proof}

\begin{lemma}\label{L39}
Let $p\geq 3$,  $n\geq 6$ with $n\not\equiv 0\Md p$ and $\la\in\Par_p(n)\setminus\h_p(n)$. If $\la$ is not JS then $\overline{D}_0\oplus \overline{D}_1\oplus \overline{D}_3\subseteq\End_F(D^\la)$.
\end{lemma}

\begin{proof}
Clearly $\overline{D}_0\cong D_0\subseteq\End_F(D^\la)$. From Lemmas \ref{Lemma39} and \ref{l2} we have that
\[\dim\Hom_{\s_n}(M_1,\End_F(D^\la))=\dim\End_{\s_{n-1}}(D^\la)\geq 2.\]
From Lemma \ref{L17} we then have that $\overline{D}_1\cong D_1\subseteq\End_F(D^\la)$.  From Lemma \ref{LH} we have that $\overline{D}_3\cong S_{1^3}$. So $\overline{D}_3\subseteq \End_F(D^\la)$ by Lemma \ref{L3}.
\end{proof}

\begin{lemma}\label{L41}
Let $p\geq 5$, $n\geq 6$ with $n\equiv 0\Md p$ and $\la\in\Par_p(n)$. If $h(\la)=2$ and $\la_1-\la_2\not\equiv 0,-1,-2\Md p$ then $\overline{D}_0\oplus \overline{D}_2\subseteq\End_F(D^\la)$ or $\overline{D}_0\oplus \overline{D}_1\oplus \overline{D}_3\subseteq\End_F(D^\la)$.
\end{lemma}

\begin{proof}
Notice that by Lemma \ref{Lemma39} and considering branching in characteristic 0,
\begin{align*}
D^\la\da_{\s_{n-2,2}}\cong\,& (D^{(\la_1-2,\la_2)}\boxtimes D^{(2)})\oplus (D^{(\la_1-1,\la_2-1)}\boxtimes D^{(2)})\\
&\oplus (D^{(\la_1-1,\la_2-1)}\boxtimes D^{(1^2)})\oplus (D^{(\la_1,\la_2-2)}\boxtimes D^{(2)})^{\oplus a},
\end{align*}
with $a=1$ if $\la_2\geq 2$ and $\la_1-\la_2\not\equiv -3\Md p$ or $a=0$ else. From Lemmas \ref{l2} and \ref{L16} it follows that $\overline{D}_1$ or $\overline{D}_2$ is contained in $\End_F(D^\la)$. From Lemmas \ref{LH} and \ref{L3} we also have that $\overline{D}_2$ or $\overline{D}_3$ is contained in $\End_F(D^\la)$. The lemma follows.
\end{proof}

\begin{lemma}\label{L43}\label{L42}
Let $p\geq 3$, $n\geq 6$ with $n\equiv 0\Md p$ and $\la\in\Par_p(n)$. If $h(\la)=2$, $\la_1>\la_2\geq 2$ and $\la_1-\la_2\equiv 0$ or $-1\Md p$, then $\overline{D}_0\oplus \overline{D}_2\subseteq\End_F(D^\la)$ or $\overline{D}_0\oplus \overline{D}_1\oplus \overline{D}_3\subseteq\End_F(D^\la)$.
\end{lemma}

\begin{proof}
We will use Lemma \ref{Lemma39} without further reference to it.

We may assume that $\overline{D}_2\not\subseteq \End_F(D^\la)$. From Lemmas \ref{LH} and \ref{L3} we then have that $\overline{D}_3\subseteq \End_F(D^\la)$. So it is enough to prove that $\overline{D}_1\cong D_1\subseteq\End_F(D^\la)$. From Lemmas \ref{l2} and \ref{L16} it is enough to prove that
\[\dim\End_{\s_{n-2}}(D^\la\da_{\s_{n-2}})-\dim\End_{\s_{n-2,2}}(D^\la\da_{\s_{n-2,2}})\geq 2.\]

{\bf Case 1:} $\la_1-\la_2\equiv 0\Md p$. Note that $\la_1\equiv\la_2\equiv 0\Md p$. So
\[D^\la\da_{\s_{n-2}}\cong D^{(\la_1-1,\la_2-1)}\oplus D^{(\la_1,\la_2-2)}\oplus e_{-2}D^{(\la_1-1,\la_2)},\]
with $e_{-2}D^{(\la_1-1,\la_2)}$ indecomposable with simple head and socle and
\[e_{-2}D^{(\la_1-1,\la_2)}\sim D^{(\la_1-1,\la_2-1)}|A|D^{(\la_1-1,\la_2-1)}\]
with $[A:D^{(\la_1-1,\la_2-1)}]=0$.  Further $D^{(\la_1-1,\la_2-1)}\otimes D^{(1^2)}$ is a composition factor of $D^\la\da_{\s_{n-2,2}}$  with multiplicity 1 by \cite[Lemma 1.11]{bk5}. So by self-duality of $D^\la\da_{\s_{n-2,2}}$ (or block decomposition) it follows that
\[D^\la\da_{\s_{n-2}}\cong (D^{(\la_1-1,\la_2-1)}\boxtimes D^{(1^2)})\oplus (D^{(\la_1,\la_2-2)}\boxtimes D^\nu)\oplus B,\]
with $\nu\in\{(2),(1^2)\}$ and $B$ indecomposible with simple head and socle isomorphic to to $D^{(\la_1-1,\la_2-1)}\boxtimes D^{(2)}$ and no other such composition factor. It then follows that
\[\dim\End_{\s_{n-2}}(D^\la\da_{\s_{n-2}})-\dim\End_{\s_{n-2,2}}(D^\la\da_{\s_{n-2,2}})=6-4=2.\]

{\bf Case 2:} $\la_1-\la_2\equiv -1\Md p$. In this case $\la_1\equiv (p-1)/2\Md p$ and $\la_2\equiv (p+1)/2\Md p$. So both removable nodes have the same residue. Then by \cite[Lemma 4.2]{m2} we have that
\[D^\la\da_{\s_{n-2,2}}\cong (D^{(\la_1-1,\la_2-1)}\boxtimes D^{(2)})\oplus(D^{(\la_1-1,\la_2-1)}\boxtimes D^{(1^2)})\oplus B\]
for a certain module $B$ and then
\[D^\la\da_{\s_{n-2}}\cong (D^{(\la_1-1,\la_2-1)})^{\oplus 2}\oplus B'\]
with $B'\cong B\da_{\s_{n-2}}$. It then follows that
\[\dim\End_{\s_{n-2}}(D^\la\da_{\s_{n-2}})-\dim\End_{\s_{n-2,2}}(D^\la\da_{\s_{n-2,2}})\geq 2.\]
\end{proof}

\begin{lemma}\label{L44}
Let $p\geq 5$, $n\geq 8$, $\la=(\la_1,\la_2)\in\Par_p(n)$ with $\la_2\geq 2$. If $\la$ is JS then
\[\dim\Hom_{\s_n}(\widetilde{D}_k,\End_F(D^\la))=\left\{\begin{array}{ll}
1,&k\in\{0,3\},\\
0,&\text{else}.
\end{array}\right.\]
\end{lemma}

\begin{proof}
We will use Lemma \ref{Lemma39} throughout the proof without further reference to it.

If $h(\mu)\geq 5$ then $D^\mu\not\subseteq \End_F(D^\la)$, since $\la$ has only 2 rows (note that any composition factor of $D^\la$ is also a composition factor of $S^\la\otimes M^\la\cong S^\la\da_{\s_{\la_1,\la_2}}\ua^{\s_n}$).

Since $\la\not=\la^\Mull$ by \cite[Lemma 1.8]{ks2}, we have that $D^{(n)^\Mull}\not\subseteq \End_F(D^\la)$. So we only need to check the lemma for $k\leq 3$. For $k=0$ the lemma clearly holds.
 
If $\la_1>\la_2$ then $\la_1\geq\la_2+3$ since $\la$ is JS. If $\la_1=\la_2$ then $\la_2\geq 4$ since $n\geq 8$.

It can be checked that $D^{(2)}$, $D^{(1^2)}$, $D^{(2,1)}$, $D^{(3,1)}$ and $D^{(2^2)}$ are composition factors of $D^\la\da_{\s_k}$ for the corresponding $k$. Comparing dimensions and multiplicities as well as $D^\la\da_{\s_{n-k,k}}\da_{\s_{n-k-1,1,k}}$ and $D^\la\da_{\s_{n-k-1,k+1}}\da_{\s_{n-k-1,1,k}}$ we that have that if $\la_1>\la_2$ then
\begin{align*}
D^\la\da_{\s_{n-1}}\cong\,& D^{(\la_1-1,\la_2)},\\
D^\la\da_{\s_{n-2,2}}\cong\,& (D^{(\la_1-2,\la_2)}\boxtimes D^{(2)})\oplus (D^{(\la_1-1,\la_2-1)}\boxtimes D^{(1^2)}),\\
D^\la\da_{\s_{n-3,3}}\cong\,& (D^{(\la_1-3,\la_2)}\boxtimes D^{(3)})\oplus (D^{(\la_1-2,\la_2-1)}\boxtimes D^{(2,1)}),\\
D^\la\da_{\s_{n-4,4}}\cong\,& (D^{(\la_1-4,\la_2)}\boxtimes D^{(4)})^{\oplus \de_{\la_1\geq\la_2+4}}\oplus (D^{(\la_1-3,\la_2-1)}\boxtimes D^{(3,1)})\\
&\oplus (D^{(\la_1-2,\la_2-2)}\boxtimes D^{(2^2)}),
\end{align*}
while $\la_1=\la_2$ and
\begin{align*}
D^\la\da_{\s_{n-1}}\cong\,& D^{(\la_1,\la_2-1)},\\
D^\la\da_{\s_{n-2,2}}\cong\,& (D^{(\la_1,\la_2-2)}\boxtimes D^{(2)})\oplus (D^{(\la_1-1,\la_2-1)}\boxtimes D^{(1^2)}),\\
D^\la\da_{\s_{n-3,3}}\cong\,& (D^{(\la_1,\la_2-3)}\boxtimes D^{(3)})\oplus (D^{(\la_1-1,\la_2-2)}\boxtimes D^{(2,1)}),\\
D^\la\da_{\s_{n-4,4}}\cong\,& (D^{(\la_1,\la_2-4)}\boxtimes D^{(4)})^{\oplus \de_{p=5}}\oplus (D^{(\la_1-1,\la_2-3)}\boxtimes D^{(3,1)})\\
&\oplus (D^{(\la_1-2,\la_2-2)}\boxtimes D^{(2^2)})
\end{align*}
(in the first case as well as some parts in the second case this also follows from \cite[Lemma 1.11]{bk5} and by comparing multiplicities and dimensions).

From Lemma \ref{L3} we have that $\overline{D}_2$ or $\overline{D}_3$ is contained in $\End_F(D^\la)$. The lemma then follows from Lemma \ref{L17} or \ref{L16} together with Lemma \ref{l2}.
\end{proof}

\begin{lemma}\label{L56}
Let $p=3$, $n\geq 6$ with $n\equiv 0\Md 3$ and $\la=(\la_1,\la_2,\la_3)\in\Par_3(n)$ with $\la_1>\la_2>\la_3\geq 1$. If $\la$ is not JS then $\overline{D}_0\oplus \overline{D}_1\oplus \overline{D}_2\subseteq\End_F(D^\la)$ or $\overline{D}_0\oplus \overline{D}_1\oplus \overline{D}_3\subseteq\End_F(D^\la)$.
\end{lemma}

\begin{proof}
In view of Lemmas \ref{LH} and \ref{L3} we have that $\overline{D}_2$ or $\overline{D}_3$ is contained in $\End_F(D^\la)$. Thus is it enough to prove that $\overline{D}_1\cong D_1\subseteq\End_F(D^\la)$. By assumption $\la_1-\la_2\equiv\la_2-\la_3\Md 3$ and we may assume that $\la_1-\la_2\not\equiv 1\Md 3$.

If $\la_1-\la_2\equiv 2\Md 3$ then $\la$ has 3 normal nodes. So $\dim\End_{\s_{n-3}}(D^\la)=3$ by Lemma \ref{Lemma39}. It then follows from Lemmas \ref{l2} and \ref{L16} that $D_1\subseteq\End_F(D^\la)$.

So assume now that $\la_1-\la_2\equiv 0\Md 3$. In this case if $i$ is the residue of $(1,\la_1)$ then $\eps_i(\la)=1$, $\eps_{i-1}(\la)=1$, $\phi_i(\la)\geq 1$ and $\phi_{i-1}(\la)\geq 1$. So, by Lemmas \ref{Lemma39} and \ref{Lemma40},
\[D^\la\otimes M_1\cong f_ie_iD^\la\oplus f_{i-1}e_{i-1}D^\la\oplus M\sim (D^\la|\ldots|D^\la)\oplus (D^\la|\ldots|D^\la)\oplus M\]
for a certain module $M$. It then follows from Lemma \ref{L16} that also in this case $D_1\subseteq\End_F(D^\la)$.
\end{proof}

\subsection{Spin representations}\label{sr}

The results obtained in this section will only be used to obtain reduction to tensor products with the natural module and with the basic spin module for $p=3$. However we prove them in general, since the proof in the general case is not more complicated or much longer. We refer the reader to \cite[Section 2-b]{BK} for the definition of even and odd homomorphisms.

\begin{lemma}\label{L071218_3}
Let $A$ be a superalgebra and $M$ be an $A$-supermodule with $\hd(M)\cong D$ simple of type Q. If $\End_A(M)\simeq\End_A(D)^{[M:D]}$ then $M$ admits an odd involution.
\end{lemma}

\begin{proof}
Note that $\Hom_A(M,\rad M)\simeq\End_A(D)^{[M:D]-1}$, since $\hd(M)\cong D$ and $\End_A(M)\simeq\End_A(D)^{[M:D]}$. Since $D$ is of type Q, there are $[M:D]$ even and $[M:D]$ odd linearly independent homomorphisms in $\End_A(M)$ but only $[M:D]-1$ even and $[M:D]-1$ odd linearly independent homomorphisms in $\Hom_A(M,\rad M)$. The lemma then follows.
\end{proof}

\begin{lemma}\label{L071218_4}
Let $A$ and $B$ be superalgebras, $M$ be an $A$-supermodule and $N$ be a $B$-supermodule. If both $M$ and $N$ admit an odd involution then there exists an $A\otimes B$-supermodule $L$ such that $M\boxtimes N\cong L^{\oplus 2}$.
\end{lemma}

\begin{proof}
As in \cite[Section 2-b]{BK} (in order for the argument to work it is not required that $M$ and $N$ are simple).
\end{proof}

\begin{lemma}\label{L101218_2}
Let $\nu\in\RP_p(n-1)$. If $\epsilon_i(\nu)>0$ then the following happen.
\begin{enumerate}
\item If $D(\widetilde e_i\nu)$ is of type M then $e_iD(\nu)\ua^{\widetilde\s_{n-2,2}}\cong e_iD(\nu)\boxtimes D((2))=: e_iD(\nu)\circledast D((2))$.

\item If $D(\widetilde e_i\nu)$ is of type Q then $e_iD(\nu)\ua^{\widetilde\s_{n-2,2}}\cong e_iD(\nu)\boxtimes D((2))\cong (e_iD(\nu)\circledast D((2)))^{\oplus 2}$ for a certain module $e_iD(\nu)\circledast D((2))$.
\end{enumerate}
Further $e_iD(\nu)\circledast D((2))$ has simple head and socle isomorphic to $D(\widetilde e_i\nu,(2))$ and
\[\dim\End_{\widetilde\s_{n-2,2}}(e_iD(\nu)\circledast D((2)))=\eps_i(\nu)\dim\End(D(\widetilde e_i\nu,(2))).\]
\end{lemma}

\begin{proof}
(i) clearly holds. For (ii) note that by Lemma \ref{Lemma39s}, $\hd(e_iD(\nu))\cong D(\tilde e_i\nu)$ and $\End_{\s_{n-2}}(e_iD(\nu))\simeq \End_{\s_{n-2}}(D(\tilde e_i\nu))^{\oplus [e_iD(\nu):D(\tilde e_i\nu)]}$. So $e_iD(\nu)$ admits an odd involution by Lemma \ref{L071218_3} and then (ii) follows from Lemma \ref{L071218_4}. 

Further for any $\mu\in\RP_p(n-2)$
\begin{align*}
&\dim\Hom_{\widetilde\s_{n-2,2}}(e_iD(\nu)\ua^{\widetilde\s_{n-2,2}},D(\mu,(2)))\\
&=\dim\Hom_{\widetilde\s_{n-2}}(e_iD(\nu),D(\mu,(2))\da_{\widetilde\s_{n-2}})\\
&=2^{1-a(\mu)}\dim\Hom_{\widetilde\s_{n-2}}(e_iD(\nu),D(\mu))\\
&=2\delta_{\mu,\widetilde e_i\nu}.
\end{align*}
Since $D(\widetilde e_i\nu)$ and $D(\widetilde e_i\nu,(2))$ are of different type, it follows that head and socle of $e_iD(\nu)\circledast D((2))$ are isomorphic to $D(\widetilde e_i\nu,(2))$.

Last, from Lemma \ref{Lemma39s}, we have that
\begin{align*}
&\dim\End_{\widetilde\s_{n-2,2}}(e_iD(\nu)\circledast D((2)))\\
&=2^{-2a(\widetilde e_i\nu)}\dim\End_{\widetilde\s_{n-2,2}}(e_iD(\nu)\ua^{\widetilde\s_{n-2,2}})\\
&=2^{-2a(\widetilde e_i\nu)}\dim\Hom_{\widetilde\s_{n-2}}(e_iD(\nu),e_iD(\nu)\ua^{\widetilde\s_{n-2,2}}\da_{\widetilde\s_{n-2}})\\
&=2^{1-2a(\widetilde e_i\nu)}\dim\End_{\widetilde\s_{n-2}}(e_iD(\nu))\\
&=2^{1-a(\widetilde e_i\nu)}\eps_i(\nu)\\
&=\eps_i(\nu)\dim\End(D(\widetilde e_i\nu,(2))).
\end{align*}
\end{proof}

\begin{lemma}\label{L051218_5}
Let $n\geq 5$ and $\la\in\RP_p(n)$. If $\epsilon_0(\la),\epsilon_i(\la)=1$ and $\eps_j(\la)=0$ for $j\not=0,i$ then at least one of $\widetilde e_0\la$ and $\widetilde e_j\la$  is not JS.
\end{lemma}

\begin{proof}
Notice first that $h(\la)\geq 2$. Assume that $\widetilde e_j\la$ is JS. Then it is JS(0) by Lemma \ref{L051218_4}. Since $\phi_j(\widetilde e_j\la)\geq 1$, it follows from Lemma \ref{L10} that the top addable node of $\la$ is the only conormal node of $\widetilde e_j\la$ and this node has residue $j$. So the normal nodes of $\la$ are on row 1 (of residue $j$) and on row $h(\la)$ (of residue 0). It is easy to see that $(1,\la_1)$ is normal also in $\widetilde e_0\la$ (any removable node in $\la$ is also removable in $\widetilde e_0\la$ apart for the node $(h(\la),1)$ and any addable node in $\widetilde e_0\la$ is also addable in $\la$ again apart the node $(h(\la),1)$). Since $\widetilde e_0\la$ is JS it follows that $\la=(n-1,1)$. From $\widetilde e_j\la=(n-2,1)$ being JS(0) it follows from \cite[Lemma 3.7]{p} that $\la=(3,1)$ or $(p,1)$. The first case contradicts $n\geq 5$ while in the second case both removable nodes have residue 0.
\end{proof}

\begin{lemma}\label{L12}
Let  $\la\in\RPar_p(n)$ be $\la\in JS(0)$. Then $\phi_0(\la)\geq 1$ if and only if $n\equiv 0\Md p$.
\end{lemma}

\begin{proof}
Assume first that $\la\in\JS(0)$ and $n\equiv 0\Md p$. From Lemma \ref{M1} we have that $[D(\la)\otimes M^{(n-1,1)}:D(\la)]\geq 2$. So from Lemma \ref{L9} it follows that $\phi_0(\la)\geq 1$.

Assume now that $\la\in JS(0)$ and $\phi_0(\la)\geq 1$. Then all normal and conormal nodes of $\la$ have residue 0 by Lemma \ref{L10}. In particular the top addable node has residue 0. So $\la_1\equiv 0\mbox{ or }-1\Md p$.

If $\la_1\equiv -1\Md p$ then from \cite[Lemma 3.7]{p} we have that $\bar\la:=(\la_1+1,\la_1,\la_2,\ldots)\in JS(0)$. Since $|\la|\equiv|\bar\la|\Md p$, we may assume that $\la_1\equiv 0\Md p$.

For a residue $i$ define
\begin{align*}
A_i&:=\{2\leq j\leq h(\la)|\res(j,\la_j)=i=\res(j-1,\la_{j-1})-1\},\\
B_i&:=\{2\leq j\leq h(\la)|\res(j,\la_j)=i=\res(j-1,\la_{j-1})+1\},\\
C_i&:=\{2\leq j\leq h(\la)|\res(j,\la_j)=i=\res(j-1,\la_{j-1})\}.
\end{align*}
From \cite[Lemma 3.7]{p} we have the following:
\begin{enumerate}[-]
\item
$\cup_i(A_i\cup B_i\cup C_i)=\{2,\ldots,h(\la)\}$,

\item
if $C_i\not=\emptyset$ then $i=0$,

\item
if $j\in C_0$ then $\la_j\equiv 0\Md p$

\item
if $j\in A_i$ then $\la_j\equiv i+1\Md p$

\item
if $j\in B_{i+1}$ then $\la_j\equiv -i-1\Md p$.
\end{enumerate}
Since $\res(1,\la_1)=0=\res(h(\la),\la_{h(\la)})$ it then follows that $|A_i|=|B_{i+1}|$ for each $0\leq i<(p-1)/2$. In particular
\begin{align*}
|\la|&=\la_1+\sum_{i=0}^{(p-3)/2}\sum_{j\in A_i}\la_j+\sum_{i=1}^{(p-1)/2}\sum_{j\in B_i}\la_j+\sum_{j\in C_0}\la_j\\
&\equiv \sum_{i=0}^{(p-3)/2}(|A_i|(i+1)-|B_{i+1}|(i+1))\equiv 0\Md p.
\end{align*}
\end{proof}

\begin{rem}
Let $\la\in\RPar_p(n)$ be JS(0) with $\phi_0(\la)\geq 1$. Then by Lemma \ref{L10} we have that $\phi_0(\la)=2$ and $\phi_i(\la)=0$ for $i>0$. From Lemmas \ref{Lef} and \ref{L051218_4} we have that $\widetilde f_0\la$ only has normal nodes of residue 0. So it can be seen that the following are equivalent:
\begin{enumerate}[-]
\item $\la$ is JS(0) with $\phi_0(\la)\geq 1$,

\item $\la=\widetilde e_0\mu$ is JS(0) and all normal nodes of $\mu$ have residue 0.
\end{enumerate}
This holds for example if $p=5$ and $\la=(4,3,2,1)=\widetilde e_0(5,3,2,1)$. Note that $(5,3,2,1)\not=(5^2,1)=\be_{11}$. This shows that \cite[Lemma 3.14(i)]{p} is wrong. Since it is unclear where the error is in the proof of \cite[Lemma 3.14]{p} we next give a different proof of \cite[Lemma 3.14(ii)]{p}.
\end{rem}

\begin{lemma}\label{L13}
Let  $\la\in\RPar_p(n)$. Then $\la\in JS(i)$ and $\widetilde e_i\la\in JS(j)$ for some $i,j\not=0$ if and only if $\la=\be_n$ with $n\not\equiv 0,1,2\Md p$.
\end{lemma}

\begin{proof}
For $\la=\be_n$ it can be easily checked that $\la\in JS(i)$ and $\widetilde e_i\la\in JS(j)$ for some $i,j\not=0$ if and only if $n\not\equiv 0,1,2\Md p$.

So assume that $\la\in JS(i)$ and $\widetilde e_i\la\in JS(j)$ for some $i,j\not=0$. Notice that the only normal node of $\la$ (of $\widetilde{e}_i\la$) is the last node on the bottom row, since $\la$ ($\widetilde e_i\la$) is JS. It then easily follows from $i,j\not=0$ that $h(\la)=h(\widetilde e_i\la)$, that $3\leq \la_{h(\la)}<p$ and $\widetilde e_i\la=(\la_1,\ldots,\la_{h(\la)-1},\la_{h(\la)}-1)$. If $p|\la_k$ for each $1\leq k<h(\la)$ then $\la=(p^{h(\la)-1},\la_{h(\la)})$ and so $\la=\be_n$ and $n\not\equiv 0,1,2\Md p$. So assume that this is not the case and let $k< h(\la)$ maximal such that $p\nmid \la_k$. Notice that $\la=(\la_1,\ldots,\la_k,p^{h(\la)-k-1},\la_{h(\la)})$. Since $\la$ is JS it can be checked that $\res(k,\la_k)=\res(h(\la),\la_{h(\la)}+1)$. On the other hand, since $\widetilde e_i\la=(\la_1,\ldots,\la_k,p^{h(\la)-k-1},\la_{h(\la)}-1)$ is also JS, we have that $\res(k,\la_k)=\res(h(\la),\la_{h(\la)})$. In particular $\res(h(\la),\la_{h(\la)})=\res(h(\la),\la_{h(\la)}+1)$ and so $p|\la_{h(\la)}$, contradicting $\la\in\RPar_p(n)$.
\end{proof}

\begin{lemma}\label{L081218}
Let $n\geq 5$ and $\la\in\RP_p(n)\setminus\{\be_n\}$. Then
\[\dim\End_{\widetilde\s_{n-2,2}}(D(\la)\da_{\widetilde\s_{n-2,2}})>\dim\End_{\widetilde\s_{n-1}}(D(\la)\da_{\widetilde\s_{n-1}})+\dim\End_{\widetilde\s_{n}}(D(\la))\]
unless one of the following holds:
\begin{enumerate}[-]
\item $\la$ is JS(1), $p=3$ and $\eps_0(\widetilde e_1\la)=3$,

\item $\la$ is JS(1), $p>3$, $\eps_0(\widetilde e_1\la)=1$ and $\eps_2(\widetilde e_1\la)=1$,

\item $\eps_0(\la)=2$, $\eps_i(\la)=0$ for $i>0$ and $\widetilde e_0\la\in\JS(0)$,

\item $\la$ is JS(0).
\end{enumerate}
\end{lemma}

\begin{proof}
Throughout the proof let $\eps_i:=\eps_i(\la)$ and for $\al\in\Par(n)$ let $d_\al:=\dim\End_{\widetilde\s_\al}(D(\la)\da_{\widetilde\s_\al})$. We will use Lemma \ref{Lemma39s} without further referring to it. 

Note that
\[D(\la)\da_{\widetilde\s_{n-2,2}}\cong \bigoplus_i E_i\oplus\bigoplus_{i<j}E_{i,j},\]
with $E_i\da_{\widetilde\s_{n-2}}\cong (\Res_i)^2D(\la)$ and $E_{i,j}\da_{\widetilde\s_{n-2}}\cong \Res_i\Res_jD(\la)\oplus\Res_j\Res_iD(\la)$.

Further, since $M^{(n-2)}\boxtimes M^{(1^2)}\cong(\1\boxtimes\1)\oplus(\1\boxtimes\sgn)$, we have that $E_i\da_{\s_{n-2}}\ua^{\widetilde\s_{n-2,2}}\cong E_i\oplus E_i'$ with $E_i'\cong E_i\otimes (\1\boxtimes \sgn)$ and $E_{i,j}\da_{\widetilde\s_{n-2}}\ua^{\widetilde\s_{n-2,2}}\cong E_{i,j}\oplus E_{i,j}'$ with $E_{i,j}'\cong E_{i,j}\otimes (\1\boxtimes \sgn)$. In particular $\dim\End_{\widetilde\s_{n-2,2}}(E_i)=\dim\End_{\widetilde\s_{n-2,2}}(E_i')$ and $\dim\End_{\widetilde\s_{n-2,2}}(E_{i,j})=\dim\End_{\widetilde\s_{n-2,2}}(E_{i,j}')$.

Consider first $E_i$. If $\eps_i>0$ then
\[(e_iD(\widetilde e_i\la))^{\oplus 2+2\de_{i>0}}\subseteq (e_i^{(2)}D(\la))^{\oplus 2+2\de_{i>0}}\subseteq E_i\da_{\widetilde\s_{n-2}}.\]
In particular $A=(e_iD(\widetilde e_i\la)\circledast D((2)))^{\oplus (2+2\de_{i>0})(1+a(\la))}\subseteq E_i\oplus E_i'$. So $(e_iD(\widetilde e_i\la)\circledast D((2)))^{\oplus (1+\de_{i>0})(1+a(\widetilde e_i\la))}$ is contained in $E_i$ or $E_i'$ from Lemma \ref{L101218_2} and similarly to \cite[Lemma 3.7]{m2}. Due to self-duality of the modules it then follows that
$$\dim\End_{\widetilde\s_{n-2,2}}(E_i)\geq 2\delta_{\eps_i>0}(1+\de_{i>0})^2(\eps_i-1)d_{(n)}.$$

Consider next $E_{i,j}$ with $0<i<j$. Assume that $\eps_i,\eps_j>0$. Then $(e_iD(\widetilde e_j\la)\oplus e_jD(\widetilde e_i\la))^{\oplus 2}\subseteq E_{i,j}\da_{\widetilde\s_{n-2}}$. In particular
$$(e_iD(\widetilde e_j\la)\circledast D((2))\oplus e_jD(\widetilde e_i\la)\circledast D((2)))^{\oplus 2+2a(\la)}\subseteq E_{i,j}\oplus E_{i,j}'.$$
Let $\{k,l\}=\{i,j\}$ with $\eps_k(\widetilde e_l\la)\geq\eps_l(\widetilde e_k\la)$. Then one of
\[(e_iD(\widetilde e_j\la)\circledast D(\hspace{-1pt}(2)\hspace{-1pt})\oplus e_jD(\widetilde e_i\la)\circledast D(\hspace{-1pt}(2)\hspace{-1pt})\hspace{-1pt})^{\oplus 1+a(\la)}\hspace{8pt}\text{or}\hspace{8pt}(e_kD(\widetilde e_l\la)\circledast D(\hspace{-1pt}(2)\hspace{-1pt})\hspace{-1pt})^{\oplus 2+a(\la)}\]
is contained in $E_{i,j}$ or $E_{i,j}'$. In either case it follows from $\eps_a(\widetilde e_b\la)\geq \eps_a$ (see Lemma \ref{L051218_4}) and from Lemma \ref{L081218_2} that
\begin{align*}
\dim\End_{\widetilde\s_{n-2,2}}(E_{i,j})&> 2\de_{\eps_i>0}\de_{\eps_j>0} (\eps_i(\widetilde e_j\la)+\eps_j(\widetilde e_i\la))d_{(n)}\\
&\geq 2\de_{\eps_i>0}\de_{\eps_j>0} (\eps_i+\eps_j)d_{(n)}.
\end{align*}

Last consider $E_{0,i}$ with $i>0$. Again assume that $\eps_0,\eps_i>0$. Then $(e_0D(\widetilde e_i\la)\oplus e_iD(\widetilde e_0\la))^{\oplus 1+a(\la)}\subseteq E_{0,i}\da_{\widetilde\s_{n-2}}$. In particular
$$(e_0D(\widetilde e_i\la)\circledast D((2))\oplus e_iD(\widetilde e_0\la)\circledast D((2)))^{\oplus 2}\subseteq E_{i,j}\oplus E_{i,j}'.$$
Similar to the previous case we obtain
\begin{align*}
\dim\End_{\widetilde\s_{n-2,2}}(E_{0,i})&> \de_{\eps_0>0}\de_{\eps_i>0} (\eps_0(\widetilde e_i\la)+\eps_i(\widetilde e_0(\la))d_{(n)}\\
&\geq \de_{\eps_0>0}\de_{\eps_i>0} (\eps_0+\eps_i)d_{(n)}.
\end{align*}

In particular, if $x=|\{j>0:\eps_j>0\}|$,
\[d_{(n-2,2)}\geq\left(\delta_{\eps_0>0}((2+x)\eps_0-2)+\sum_{i>0}\delta_{\eps_i>0}((6+2x+\de_{\eps_0>0})\eps_i-8)\right)d_{(n)}.\]

In view of Lemma \ref{L15} we may assume that
\[d_{(n-2,2)}\leq d_{(n-1,1)}+d_{(n)}=(1+\eps_0+2\sum_{i>0}\eps_i)d_{(n)}.\]
It easily follows that $x+\de_{\eps_0>0}\leq 2$ and that we are in one of the following cases:
\begin{enumerate}[-]
\item $\eps_0\leq 3$ and $\eps_k=0$ for $k>0$,

\item $\eps_0\leq 2$, $\eps_i=1$ and $\eps_k=0$ for $k\not=0,i$ for some $i>0$.

\item $\eps_i,\eps_j=1$ and $\eps_k=0$ for $k\not=i,j$ for some $i,j>0$.
\end{enumerate}
Excluding cases which are not considered in the lemma and considering the stronger bounds involving $\eps_i(\widetilde e_j\la)$, strict inequalities and that $E_{i,j}\not=0$ if $\eps_i>0$ and $\eps_j(\widetilde e_i\la)>0$, we may assume that we are in one of the following cases:
\begin{enumerate}
\item[(a)] $\eps_0=3$, $\eps_k=0$ and $\eps_k(\widetilde e_0\la)>0$ for $k>0$,

\item[(b)] $\eps_0=2$, $\eps_k=0$ for $k>0$ and there exists $i>0$ with $\eps_i(\widetilde e_0\la)>0$,

\item[(c)] $\la$ is $\JS(i)$ with $i>1$,

\item[(d)] $p=3$, $\la$ is $\JS(1)$ and $\eps_0(\widetilde e_1\la)\not=3$,

\item[(e)] $p>3$, $\la$ is $\JS(1)$ and $(\eps_0(\widetilde e_1\la),\eps_2(\widetilde e_1\la))\not=(1,1)$,

\item[(f)] $\eps_0,\eps_i=1$, $\eps_k=0$ for $k\not=0,i$ and $\eps_i(\widetilde e_0\la)+\eps_0(\widetilde e_i\la)\leq 3$  for some $i>0$.

\item[(g)] $\eps_i,\eps_j=1$, $\eps_k=0$ for $k\not=i,j$ and $\widetilde e_i\la$ and $\widetilde e_j\la$ are JS for some $i,j>0$.
\end{enumerate}

{\bf Case (a).} In this case $D(\la)\da_{\widetilde\s_{n-2}}\cong e_0^{(2)}D(\la)^{\oplus 2}$ and
\[[e_0^{(2)}D(\la):D(\widetilde e_0^2\la)]=3>2=[e_0D(\tilde e_0\la):D(\widetilde e_0^2\la)].\]
It can then be checked that $(e_0D(\widetilde e_0\la)\circledast D(2))^{\oplus 1+a(\la)}$ is strictly contained in $E_0$ or $E_0'$. Thus
\begin{align*}
\dim\End_{\widetilde\s_{n-2,2}}(D(\la)\da_{\widetilde\s_{n-2,2}})&\!>\!(1+a(\la))^2(\eps_0-1)\dim\End_{\widetilde\s_{n-2,2}}(D(\widetilde e_0^2\la,(2)))\\
&\!=\!4\dim\End_{\widetilde\s_{n}}(D(\la)).
\end{align*}
So also in this case the lemma holds.

{\bf Case (b).} In this case $\dim\End_{\widetilde\s_{n-2,2}}(E_0)\geq 2\dim\End_{\widetilde\s_n}(D(\la))$, so it is enough to prove that $\dim\End_{\widetilde\s_{n-2,2}}(E_{0,i})>\dim\End_{\widetilde\s_n}(D(\la))$ by Lemma \ref{L10}. This follows from $E_{0,i}$ not being zero or simple as supermodule (since $\eps_0(\la)=2$ and $\eps_i(\widetilde e_0\la)>0$) and since its composition factors are of the same type as $D(\la)$.

{\bf Case (c).}  Using argument similar to the above we have (letting $E_{i,j}=E_{j,i}$ and  $E_{i,j}'=E_{j,i}'$ for $i>j$) that $(e_jD(\widetilde e_i\la)\circledast D((2)))^{\oplus 1+a(\la)}$ is contained in $E_{i,j}$ or $E_{i,j}'$ for each $j\not=i$ with $j> 0$. From Lemma \ref{L13} and \cite[Lemma 3.8]{p} we have that $\sum_{j\not=i}\eps_j(\widetilde e_i\la)\geq 2$. From \cite[Lemma 20.2.3]{KBook} we have that $\eps_0(\widetilde e_i\la)=0$. The lemma then follows.

{\bf Case (d).} Notice that $e_0D(\widetilde e_1\la)\circledast D((2))$ is contained in $E_{0,1}$ or $E_{0,1}'$.  Since $\la\not=\be_n$ it can be easily checked that $\la$ ends by $(4,3^b,2)$ with $b\geq 0$. It can then be easily checked that $\eps_0(\widetilde e_1\la)\geq 3$. So in this case $\eps_0(\widetilde e_1\la)\geq 4$, from which the lemma follows.

{\bf Case (e).} From \cite[Lemma 20.2.3]{KBook} and since $\la\in\JS(1)$ we have that $\eps_k(\widetilde e_1\la)=0$ for $k\not=0,2$. If $\la_{h(\la)}=p-1$ then the bottom removable node of $\widetilde e_1\la$ is $2$-normal (since $p>3$). If $\la_{h(\la)}=2$ let $k<h(\la)$ maximal with $p\nmid \la_k$. Note that $k$ exists since $\la\not=\be_n$ From $\la\in\JS(1)$ it follows that $\res(k,\la_k)=2$ and by maximality of $k$ we have that $(k,\la_k)$ is normal for $\widetilde e_1\la$. In particular $\eps_2(\widetilde e_1\la)\geq 1$.

We have that  $(e_0D(\widetilde e_1\la)\circledast D((2)))^{\oplus 2}$ is contained in $E_{0,1}$ or $E_{0,1}'$ and $(e_jD(\widetilde e_i\la)\circledast D((2)))^{\oplus 1+a(\la)}$ is contained in $E_{i,j}$ or $E_{i,j}'$ for each $j\not=i$ with $j> 0$. Since $\widetilde e_1\la$ is not JS by Lemma \ref{L13} and \cite[Lemma 3.8]{p}, we have $\eps_2(\widetilde e_1\la)\geq 2$ or $\eps_0(\widetilde e_1\la),\eps_2(\widetilde e_1\la)\geq 1$ from which the lemma follows.

{\bf Case (f).} From Lemma \ref{L051218_5} we have that $\widetilde e_0(\la)$ and $\widetilde e_i(\la)$ are not both JS. Since $\eps_i(\widetilde e_0\la)+\eps_0(\widetilde e_i\la)\leq 3$, we have by Lemmas \ref{Lemma39s} and \ref{L081218_2} that $\widetilde e_i\widetilde e_0\la=\widetilde e_0\widetilde e_i\la$.

If $\eps_i(\widetilde e_0\la)+\eps_0(\widetilde e_i\la)=3$ the lemma follows from $(e_0D(\widetilde e_i\la)\circledast D((2))\oplus e_iD(\widetilde e_0\la)\circledast D((2)))$ being contained in $E_{0,i}$ or $E_{0,i}'$ or $(e_kD(\widetilde e_l\la)\circledast D((2)))^{\oplus 2}$ being contained in one of $E_{0,i}$ or $E_{0,i}'$ (with $\{k,l\}=\{0,i\}$ such that $\eps_k(\widetilde e_l\la)=2$).

If $\eps_i(\widetilde e_0\la)+\eps_0(\widetilde e_i\la)=2$ then $E_{0,i}$ is not the only non-zero block component of $D(\la)\da_{\widetilde\s_{n-2,2}}$, since $\widetilde e_0(\la)$ and $\widetilde e_i(\la)$ are not both JS. From $\widetilde e_i\widetilde e_0\la=\widetilde e_0\widetilde e_i\la$ we have that $(D(\widetilde e_i\widetilde e_0\la)\circledast D((2)))^{\oplus 2}\subseteq E_{0,i}$ or $E_{0,i}'$, from which the lemma then follows.

{\bf Case (g).} In this case from Lemma \ref{L081218_2} we have that $\widetilde e_i\widetilde e_j\la=\widetilde e_j\widetilde e_i\la$ and then $D(\widetilde e_i\widetilde e_j\la)^{\oplus 4}$ is contained in $D(\la)\da_{\widetilde\s_{n-2}}$. So $(D(\widetilde e_i\widetilde e_j\la)\circledast D((2)))^{\oplus 2+2a(\la)}$ is contained in $E_{i,j}$ or $E_{i,j}'$, from which the lemma follows by Lemma \ref{L10}.
\end{proof}

\begin{lemma} \label{L101218_3}
Let $\la\in\RP_p(n)\setminus\{\be_n\}$ with $\eps_i(\la)=0$ for all $i\neq 0$ and $\widetilde e_0 \la\in \JS(0)$, then $D(\la)\da_{\s_{n-1}}$ has a composition factor $D(\mu)$, where $\mu\not=\widetilde e_0\la$ is obtained from $\la$ by removing the bottom removable node.
\end{lemma}

\begin{proof}
Let $A=(h,\la_h)$ be the bottom removable node of $\la$. Then $A$ is normal for $\la$. Since all normal nodes of $\la$ have residue 0 and $\widetilde e_0\la\in\JS(0)$, we have that $A$ is not good, so $\mu\not=\widetilde e_0\la$. By Lemma \ref{Lemma39sr} it is then enough to prove that $\mu\in\RP_p(n-1)$. Note that $A$ has residue $0$, so $\la_h=1$. 
If $\mu\not\in\RP_p(n-1)$ then $\la_{h-1}=p$. So the node $B:=(h-1,p)$ is also normal for $\la$. Since $\widetilde e_0\la\in\JS(0)$ we have that $\eps_0(\la)=2$. In particular $B$ is the $0$-good node of $\la$. Let $k<h-1$ be maximal with $\la_k>p$ (such $k$  exists since $\la\not=\be_n$). By \cite[Lemma 3.7]{p} it follows that $\la_k=p+1$. In particular the node $(k,p+1)$ is removable of residue $0$ for $\la$, and then it is also $0$-normal, contradicting $B$ being the $0$-good node of $\la$.
\end{proof} 

\begin{lemma}\label{L101218_4}
Let $\la\in\RP_p(n)$ and $n\geq 4$. If $D(\la)$ is of type Q then, in the Grothendieck group, $[D(\la,+)\da_{\widetilde\s_{n-2}}]=[D(\la,-)\da_{\widetilde\s_{n-2}}]$. If $D(\la)$ is of type M then $[E(\la,+)\da_{\widetilde \A_{n-2}}]=[E(\la,-)\da_{\widetilde \A_{n-2}}]$.
\end{lemma}

\begin{proof}
Let $\mu\in\RP_p(m)$ for some $m\geq 0$. It can be checked by definition of residues that
\[m-h_{p'}(\mu)\equiv|\{\text{nodes of $\mu$ of residue }\not=0\}|\Md 2.\]
So if $i=0$ then all composition factors of $\Res_iD(\mu)$ are of the same type as $D(\mu)$, while if $i>0$ then all composition factors of $\Res_iD(\mu)$ are of type different from the type of $D(\mu)$.

Assume first that $D(\la)$ is of type Q. Then
\[D(\la,\pm)\da_{\widetilde\s_{n-2}}\cong\bigoplus_iD_{i,i}^\pm\oplus \bigoplus_{i<j}D_{i,j}^\pm\]
with $D_{i,i}^\pm\cong\Res_i^2D(\la,\pm)$ and $D_{i,j}^\pm\cong\Res_i\Res_jD(\la,\pm)\oplus\Res_j\Res_iD(\la,\pm)$. Further $D(\la,\pm)\da_{\widetilde\s_{n-2,2}}\cong\bigoplus_iE_{i,i}^\pm\oplus \bigoplus_{i<j}E_{i,j}^\pm$ with $E_{i,j}^\pm\da_{\widetilde\s_{n-2,2}}\cong D_{i,j}^\pm$ for $i\leq j$. For any $i\leq j$ we have $D_{i,j}^+\otimes\sgn\cong D_{i,j}^-$ and $E_{i,j}^+\otimes\sgn\cong E_{i,j}^-$. If $j>0$ we then easily have that $[D_{0,j}^+]=[D_{0,j}^-]$, since composition factors of $D_{0,j}^\pm$ are of the form $D(\mu,0)\cong D(\mu,0)\otimes \sgn$ for some $\mu\in\RP_p(n-2)$. If $0<i\leq j$ then $[E_{i,j}^+]=[E_{i,j}^-]$, since composition factors of $E_{i,j}^\pm$ are of the form $D(\mu,(2))\cong (D(\mu,(2)))\otimes \sgn$ for some $\mu\in\RP_p(n-2)$. Also in this case it then follows that $[D_{i,j}^+]=[D_{i,j}^-]$.

If $D(\la)$ is of type M use a similar argument involving conjugation with $\widetilde{(1,2)}$ instead of tensoring with $\sgn$.
\end{proof}

\begin{lemma}\label{L101218_5}
Let $n\geq 4$, $\la\in\RP_p(n)\setminus\{\be_n\}$. Let $G=\widetilde\s_n$ or $G=\widetilde \A_n$ and $D$ be a simple $F G$-module indexed by $\la$. Assume that one of the following holds:
\begin{enumerate}
\item $\la$ is JS(1), $p=3$ and $\eps_0(\widetilde e_1\la)=3$,

\item $\la$ is JS(1), $p>3$ and $\eps_0(\widetilde e_1\la)=1$ and $\eps_2(\widetilde e_1\la)=1$,

\item $\eps_0(\la)=2$, $\eps_i(\la)=0$ for $i>0$ and $\widetilde e_0\la\in\JS(0)$.
\end{enumerate}
Then
\[\dim\End_{\widetilde\s_{n-2,2}\cap G}(D\da_{\widetilde\s_{n-2,2}\cap G})>\dim\End_{\widetilde\s_{n-1}\cap G}(D\da_{\widetilde\s_{n-1}\cap G}).\]
\end{lemma}

\begin{proof}
We will prove the lemma corresponding to cases (i), (ii) and (iii) separately. We will use Lemma \ref{Lemma39s} without further reference.

{\bf Case (i).}
Notice that $D(\la)\da_{\widetilde\s_{n-1}}\cong D(\widetilde e_1\la)^{\oplus 1+a(\la)}$ and $D(\la)\da_{\widetilde\s_{n-2,2}}\cong e_0 D(\widetilde e_1\la)\circledast D((2))$. So the lemma holds if $G=\widetilde\s_n$ and $D\cong D(\la,0)$ or $G=\widetilde \A_n$ and $D\cong E(\la,0)$ by Lemma \ref{L101218_2}. Assume now that $G=\widetilde\s_n$ and $D\cong D(\la,\pm)$. Then $D(\la,\pm)\da_{\widetilde\s_{n-1}}\cong D(\widetilde e_1\la,0)$ and $D(\la,\pm)\da_{\widetilde\s_{n-2,2}}$ is indecomposable with simple head and socle, it has exactly $3$ composition factors of the form $(D(\widetilde e_0\widetilde e_1\la,(2)),+)$ or $(D(\widetilde e_0\widetilde e_1\la,(2)),-)$. Let $b,c\in\{\pm\}$ such that $D(\la,\pm)\da_{\widetilde\s_{n-2,2}}$ has a filtration of the form
\[
(D(\widetilde e_0\widetilde e_1\la,(2)),\pm)|\ldots|(D(\widetilde e_0\widetilde e_1\la,(2)),\pm b)|\ldots|(D(\widetilde e_0\widetilde e_1\la,(2)),\pm c).\]
Note that by self-duality of $D(\la)$ we have that
\[(D(\la,\pm)\da_{\widetilde\s_{n-2,2}})^*\in\{D(\la,\pm)\da_{\widetilde\s_{n-2,2}},D(\la,\mp)\da_{\widetilde\s_{n-2,2}}\}.\]
So there exists $d\in\{\pm\}$ such that $(D(\la,\pm)\da_{\widetilde\s_{n-2,2}})^*$ has a filtration
\[
(D(\widetilde e_0\widetilde e_1\la,(2)),\pm cd)|\ldots|(D(\widetilde e_0\widetilde e_1\la,(2)),\pm bd)|\ldots|(D(\widetilde e_0\widetilde e_1\la,(2)),\pm d).\]
It then follows that $c=+$ and so the lemma holds. The case $G=\widetilde \A_n$ and $D\cong E(\la,\pm)$ holds with similar arguments.

{\bf Case (ii).} Notice that in this case $\eps_k(\widetilde e_1\la)=0$ for $k\not=0,2$ since $\la\in\JS(1)$ and using \cite[Lemma 20.2.3]{KBook}. In particular $D(\la)\da_{\widetilde\s_{n-1}}\cong D(\widetilde e_1\la)^{\oplus 1+a(\la)}$ and $D(\la)\da_{\widetilde\s_{n-2,2}}\cong (D(\widetilde e_0\widetilde e_1\la,(2)))\oplus(D(\widetilde e_2\widetilde e_1\la,D(2)))^{\oplus 1+a(\la)}$. The lemma then easily follows.

{\bf Case (iii).} In this case by Lemma \ref{L101218_3} we have $D(\la)\da_{\widetilde\s_{n-1}}\cong e_0 D(\la)$ and $D(\la)\da_{\widetilde\s_{n-2,2}}\cong (D(\widetilde e_0^2\la,(2)))^{\oplus 1+a(\la)}\oplus A$ with $A\not=0$ corresponding to blocks different than the block of $D(\widetilde e_0^2\la,(2))$. So the lemma holds if $G=\widetilde\s_n$ and $D\cong D(\la,0)$ or $G=\widetilde \A_n$ and $D\cong E(\la,0)$. Assume now that $G=\widetilde\s_n$ and $D\cong D(\la,\pm)$. Then $D(\la,\pm)\da_{\widetilde\s_{n-2,2}}\cong (D(\widetilde e_0^2\la,(2)),0)\oplus A'$ with $A'\not=0$. So it is enough to prove that $\dim\End_{\widetilde\s_{n-1}}(D(\la,\pm)\da_{\widetilde\s_{n-1}})=1$. Note that $D(\la,\pm)\da_{\widetilde\s_{n-1}}$ has simple head and socle and exactly two composition factors of the form $D(\widetilde e_0\la,+)$ or $D(\widetilde e_0\la,-)$. Let $b\in\{\pm\}$ with
\[D(\la,\pm)\da_{\widetilde\s_{n-1}}\sim D(\widetilde e_0\la,\pm)|\ldots|D(\widetilde e_0\la,\pm b).\]
It is enough to prove that $b=-$. This follows from
\begin{align*}
\Res_0(D(\la,\pm)\da_{\widetilde\s_{n-1}})&\cong \Res_0^2(D(\la,\pm)\cong \Res_0 (D(\widetilde e_0\la,\pm)\oplus D(\widetilde e_0\la,\pm b))\\
&\cong D(\widetilde e_0^2\la,\pm)\oplus D(\widetilde e_0^2\la,\pm b)
\end{align*}
and from Lemma \ref{L101218_4}. The case $G=\widetilde \A_n$ and $D\cong E(\la,\pm)$ holds similarly.
\end{proof}

\begin{lemma}\label{L190219}
Let $p\geq 3$, $n\geq 4$ and $\la\in\RP_p(n)$. Assume that
\[\dim\End_{\widetilde\s_{n-2,2}}(D(\la)\da_{\widetilde\s_{n-2,2}})>\dim\End_{\widetilde\s_{n-1}}(D(\la)\da_{\widetilde\s_{n-1}})+\dim\End_{\widetilde\s_{n}}(D(\la)).\]
Then
\begin{enumerate}[-]
\item If $D(\la)$ is of type M then there exists
\[\psi\in\Hom_{\widetilde\s_n}(M_2,\End_F(D(\la,0)))\]
which does not vanish on $S_2$. Further there exist
\[\phi_1,\phi_2\in\Hom_{\widetilde \A_n}(M_2,\Hom_F(E(\la,\pm),E(\la)))\]
which are linearly independent over $S_2$.

\item If $D(\la)$ is of type Q then there exist
\[\psi_1,\psi_2\in\Hom_{\widetilde\s_n}(M_2,\Hom_F(D(\la,\pm),D(\la)))\]
which are linearly independent over $S_2$. Further there exists
\[\phi\in\Hom_{\widetilde \A_n}(M_2,\End_F(E(\la,0)))\]
which does not vanish on $S_2$.
\end{enumerate}
\end{lemma}

\begin{proof}
From Lemma \ref{Mk} we have that $M_2\sim S_2|M_1$.

Assume first that $D(\la)$ is of type M, so that
\[\dim\End_{\widetilde\s_{n-2,2}}(D(\la)\da_{\widetilde\s_{n-2,2}})>\dim\End_{\widetilde\s_{n-1}}(D(\la)\da_{\widetilde\s_{n-1}})+1.\]
Since $D(\la)\da_{\widetilde \A_n}\cong E(\la)$ and $D(\la,0)\cong D(\la)\cong E(\la,\pm)\ua^{\widetilde\s_n}$, for any partition $\mu\not=(1^n)$ we have that
\[\dim\End_{\widetilde\s_\mu}(D(\la)\da_{\widetilde\s_\mu})=\dim\End_{\widetilde \A_\mu}(E(\la,\pm)\da_{\widetilde \A_\mu},E(\la)\da_{\widetilde \A_\mu}).\]
The lemma then easily follows in this case.

Assume next that $D(\la)$ is of type Q, so that
\[\dim\End_{\widetilde\s_{n-2,2}}(D(\la)\da_{\widetilde\s_{n-2,2}})>\dim\End_{\widetilde\s_{n-1}}(D(\la)\da_{\widetilde\s_{n-1}})+2.\]
Then for some $\eps\in\{\pm\}$ we have that
\begin{align*}
&\dim\Hom_{\widetilde\s_{n-2,2}}(D(\la,\eps)\da_{\widetilde\s_{n-2,2}},D(\la)\da_{\widetilde\s_{n-2,2}})\\
&\geq\dim\Hom_{\widetilde\s_{n-1}}(D(\la,\eps)\da_{\widetilde\s_{n-1}},D(\la)\da_{\widetilde\s_{n-1}})+2.
\end{align*}
So there exist $\psi_1,\psi_2\in\Hom_{\widetilde\s_n}(M_2,\Hom_F(D(\la,\eps),D(\la)))$ which are linearly independent over $S_2$. The lemma then follows from
\[D(\la,+)\otimes D(\la)\cong D(\la,+)\otimes \sgn\otimes D(\la)\cong D(\la,\eps)\otimes D(\la)\]
and from $D(\la,\pm)\da_{\widetilde \A_n}\cong E(\la,0)$.
\end{proof}

\begin{lemma}\label{L190219_2}
Let $p\geq 3$, $n\geq 4$ and $\la\in\RP_p(n)$. Assume that $\la\not=\be_n$ and $\la$ is not JS(0). Then:
\begin{enumerate}[-]
\item If $D(\la)$ is of type M then there exists
\[\psi\in\Hom_{\widetilde\s_n}(M_2,\End_F(D(\la,0)))\]
which does not vanish on $S_2$. Further there exists
\[\phi\in\Hom_{\widetilde \A_n}(M_2,\End_F(E(\la,\pm)))\]
which does not vanish on $S_2$ or there exist
\[\phi_1,\phi_2\in\Hom_{\widetilde \A_n}(M_2,\Hom_F(E(\la,\pm),E(\la,\mp)))\]
which are linearly independent over $S_2$.

\item If $D(\la)$ is of type Q then there exists
\[\psi\in\Hom_{\widetilde\s_n}(M_2,\End_F(D(\la,\pm)))\]
which does not vanish on $S_2$ or there exist
\[\psi_1,\psi_2\in\Hom_{\widetilde\s_n}(M_2,\Hom_F(D(\la,\pm),D(\la,\mp)))\]
which are linearly independent over $S_2$. Further there exists
\[\phi\in\Hom_{\widetilde \A_n}(M_2,\End_F(E(\la,0)))\]
which does not vanish on $S_2$.
\end{enumerate}
\end{lemma}

\begin{proof}
From Lemma \ref{L190219} we may assume that
\[\dim\End_{\widetilde\s_{n-2,2}}(D(\la)\da_{\widetilde\s_{n-2,2}})\leq\dim\End_{\widetilde\s_{n-1}}(D(\la)\da_{\widetilde\s_{n-1}})+\dim\End_{\widetilde\s_{n}}(D(\la)).\]
Let $G\in\{\widetilde\s_n,\widetilde \A_n\}$ and $D$ be an $FG$-representation indexed by $\la$. Then by Lemmas \ref{L081218} and \ref{L101218_5} we have that
\[\dim\End_{\widetilde\s_{n-2,2}\cap G}(D\da_{\widetilde\s_{n-2,2}\cap G})>\dim\End_{\widetilde\s_{n-1}\cap G}(D\da_{\widetilde\s_{n-1}\cap G}).\]
Since $M_2\sim S_2|M_1$ by Lemma \ref{Mk}, the lemma easily follows.
\end{proof}

\subsection{Basic spin modules}\label{sbs}

\begin{lemma}\label{L22}
Let $p\geq 3$. Let $c=1$ if $p\nmid n$ or $c=2$ if $p\mid n$.
\begin{enumerate}[-]
\item If $D(\be_n)$ is of type M then $D(\be_n,0)\otimes D(\be_n)\cong\oplus_{k=0}^{n-c}\overline{D}_k$ and $E(\be_n,\pm)\otimes E(\be_n)\cong\overline{E}_{(n-c)/2,\pm}\oplus_{k=0}^{(n-2-c)/2}\overline{E}_k$.

\item If $D(\be_n)$ is of type Q then $D(\be_n,\pm)\otimes D(\be_n)\cong\oplus_{k=0}^{n-c}\overline{D}_k$ and $E(\be_n,0)\otimes E(\be_n)\cong\oplus_{k=0}^{(n-1-c)/2}\overline{E}_k$.
\end{enumerate}
\end{lemma}

\begin{proof}
Note that by \cite[Theorem 9.3]{s} 
\[[S((n),\eps)\otimes S((n))]=\sum_{k=0}^{n-1}[S^{(n-k,1^k)}],\]
with $\eps=0$ or $\pm$ depending on the type of $S((n))$. Let $\eps'=0$ or $\pm$ depending on the type of $D(\be_n)$. By Lemma \ref{LBS} we have that
\[[S((n),\eps)\otimes S((n))]=(1+\de_{p\mid n})[D(\be_n,\eps')\otimes D(\be_n)].\]
Further from Lemma \ref{LH}
\[\sum_{k=0}^{n-1}[S^{(n-k,1^k)}]=(1+\de_{p\mid n})\sum_{k=0}^{n-c}[\overline{D}_k].\]
So
\[[D(\be_n,\eps')\otimes D(\be_n)]=\sum_{k=0}^{n-c}[\overline{D}_k].\]
Since $D(\be_n,\eps')\otimes \sgn\cong D(\be_n,-\eps')$, we have that
\[D(\be_n,\eps')\otimes D(\be_n)\cong D(\be_n,-\eps')\otimes\sgn\otimes D(\be_n)\cong D(\be_n,-\eps')\otimes D(\be_n).\]
From $D(\be_n)$ being self-dual it then follows that so is $D(\be_n,\eps')\otimes D(\be_n)$. Since this module is self-dual, multiplicity free and its composition factors are self-dual, the lemma holds for $\widetilde\s_n$. For $\widetilde \A_n$ the lemma then follows by Lemma \ref{L20}.
\end{proof}

\begin{lemma}\label{L33}
Let $p\geq 3$ and $n\geq 10$. Then $\overline{D}_2\subseteq\End_F(D(\be_n,\delta))$ and $\overline{E}_2\subseteq\End_F(E(\be_n,\delta'))$.
\end{lemma}

\begin{proof}
In this case it can be easily checked from Lemma \ref{L20} that $\overline{D}_2\cong D_{1^2}$ and that $(n-2,1^2)>(n-2,1^2)^\Mull$. We will use Lemma \ref{Lemma39s} without further reference.

Note that any composition factor (as supermodule) of $D(\be_n)\da_{\widetilde\s_{n-k}}$ is of the form $D(\be_{n-k})$ (this holds for example by Lemma \ref{LBS} and branching in characteristic 0). So any composition factor of $D(\be_n)\da_{\widetilde\s_\al}$ is of the form $D(\be_{\al_1},\be_{\al_2},\ldots)$.

Consider first $D(\be_n,\delta)$. If $\delta=0$ then $D_{1^2}\subseteq\End_F(D(\be_n,\delta))$ by Lemma \ref{L22}. So we may assume that $\delta=\pm$. If $n\not\equiv 0,1,2\Md p$ then
\begin{align*}
D(\be_n)\da_{\widetilde\s_{n-1}}&\cong D(\be_{n-1})^{\oplus 2},\\
D(\be_n)\da_{\widetilde\s_{n-2}}&\cong D(\be_{n-2})^{\oplus 2},\\
D(\be_n)\da_{\widetilde\s_{n-2,2}}&\cong D(\be_{n-2},(2))^{\oplus 2},
\end{align*}
with $D(\be_{n-1})$ and $D(\be_{n-2},(2))$ of type M and $D(\be_{n-2})$ of type Q. So $D(\be_n,\pm)\da_{\widetilde\s_{n-1}}$ and $D(\be_n,\pm)\da_{\widetilde\s_{n-2,2}}$ are simple, while $D(\be_n,\pm)\da_{\widetilde\s_{n-2}}$ is a direct sum of two simple modules. So $D_{1^2}\subseteq\End_F(D(\be_n,\delta))$ by Lemmas \ref{l2} and \ref{L17}.

If $n\equiv 2\Md p$ then
\begin{align*}
D(\be_n)\da_{\widetilde\s_{n-1}}&\cong D(\be_{n-1})^{\oplus 2},\\
D(\be_n)\da_{\widetilde\s_{n-2}}&\cong(D(\be_{n-2})|D(\be_{n-2}))^{\oplus 2},\\
D(\be_n)\da_{\widetilde\s_{n-2,2}}&\cong D(\be_{n-2},(2))|D(\be_{n-2},(2)),\\
D(\be_n)\da_{\widetilde\s_{n-3,2}}&\cong D(\be_{n-3},(2))^{\oplus 2},
\end{align*}
with $D(\be_{n-1})$, $D(\be_{n-2})$ and $D(\be_{n-3},(3))$ of type M and $D(\be_{n-2},(2))$ and $D(\be_{n-3},(2))$ of type Q. In particular $D(\be_n,+)\da_{\widetilde\s_{n-1}}\cong D(\be_n,-)\da_{\widetilde\s_{n-1}}$ are simple, $D(\be_n,\pm)\da_{\widetilde\s_{n-2}}$ is uniserial with two isomorphic composition factors and $D(\be_n,\pm)\da_{\widetilde\s_{n-2,2}}$ is uniserial with two non-isomorphic composition factors (since $D(\be_n,+)\da_{\widetilde\s_{n-1}}\cong D(\be_n,-)\da_{\widetilde\s_{n-1}}$ the two composition factors of $D(\be_n,\pm)\da_{\widetilde\s_{n-3,2}}$ are not isomorphic). It then follows again by Lemmas \ref{l2} and \ref{L17} that $D_{1^2}\subseteq\End_F(D(\be_n,\delta))$.

If $n\equiv 1\Md p$ then
\begin{align*}
D(\be_n)\da_{\widetilde\s_{n-1}}&\cong D(\be_{n-1})|D(\be_{n-1}),\\
D(\be_n)\da_{\widetilde\s_{n-2}}&\cong(D(\be_{n-2}))^{\oplus 2},\\
D(\be_n)\da_{\widetilde\s_{n-2,2}}&\cong D(\be_{n-2},(2))^{\oplus 2},
\end{align*}
with $D(\be_{n-1})$ and $D(\be_{n-2})$ of type Q and $D(\be_{n-2},(2))$ of type M. In particular $D(\be_n,+)\da_{\widetilde\s_{n-2,2}}\cong D(\be_n,-)\da_{\widetilde\s_{n-2,2}}$ are simple, from which it follows that $D(\be_n,\pm)\da_{\widetilde\s_{n-2}}\cong D(\be_{n-2},+)\oplus D(\be_{n-2},-)$ and then that $D(\be_n)\da_{\widetilde\s_{n-1}}\cong D(\be_{n-1},\pm)|D(\be_{n-1},\mp)$, so again $D_{1^2}\subseteq\End_F(D(\be_n,\delta))$.

If $n\equiv 0\Md p$ and $p\not=3$ then
\begin{align*}
D(\be_n)\da_{\widetilde\s_{n-3}}&\cong D(\be_{n-3})^{\oplus 2},\\
D(\be_n)\da_{\widetilde\s_{n-3,2}}&\cong D(\be_{n-3},(2))^{\oplus 2},\\
D(\be_n)\da_{\widetilde\s_{n-3,3}}&\cong D(\be_{n-3},(3)),
\end{align*}
with $D(\be_{n-3})$ and $D(\be_{n-3},(3))$ of type Q, while $D(\be_{n-3},(2))$ is of type M. So $D(\be_n,\pm)\da_{\widetilde\s_{n-3,2}}$ and $D(\be_n,\pm)\da_{\widetilde\s_{n-3,3}}$ are simple, while $D(\be_n,\pm)\da_{\widetilde\s_{n-3}}$ is a direct sum of two simple modules. Then $D_{1^2}\subseteq\End_F(D(\be_n,\delta))$ by Lemmas \ref{l2} and \ref{L16}, since $\End_F(D(\be_n,\delta))$ is semisimple by Lemma \ref{L22}.

If $n\equiv 0\Md p$ and $p=3$ then
\begin{align*}
D(\be_n)\da_{\widetilde\s_{n-1}}&\cong D(\be_{n-1}),\\
D(\be_n)\da_{\widetilde\s_{n-2}}&\cong D(\be_{n-2})^{\oplus 2},\\
D(\be_n)\da_{\widetilde\s_{n-3}}&\cong (D(\be_{n-3})|D(\be_{n-3}))^{\oplus 2},\\
D(\be_n)\da_{\widetilde\s_{n-3,2}}&\cong D(\be_{n-3},(2))|D(\be_{n-3},(2)),\\
D(\be_n)\da_{\widetilde\s_{n-3,3}}&\cong D(\be_{n-3},(2,1))|D(\be_{n-3},(2,1)),\\
D(\be_n)\da_{\widetilde\s_{n-4,2}}&\cong D(\be_{n-4},(2))^{\oplus 2}.
\end{align*}
Further $D(\be_{n-2})$ and $D(\be_{n-3})$ are of type M while $D(\be_{n-1})$, $D(\be_{n-3},(2))$, $D(\be_{n-3},(2,1))$ and $D(\be_{n-4},(2))$ are of type Q. In particular $D(\be_n,+)\da_{\widetilde\s_{n-2}}\cong D(\be_n,+)\da_{\widetilde\s_{n-2}}$, from which follows that
\[D(\be_n,+)\da_{\widetilde\s_{n-4,2}}\cong D(\be_{n-4},(2),+)\oplus D(\be_{n-4},(2),-).\]
So 
\begin{align*}
D(\be_n,\pm)\da_{\widetilde\s_{n-3}}&\cong D(\be_{n-3},0)|D(\be_{n-3},0),\\
D(\be_n,\pm)\da_{\widetilde\s_{n-3,2}}&\cong D(\be_{n-3},(2),\pm)|D(\be_{n-3},(2),\mp),\\
D(\be_n,\pm)\da_{\widetilde\s_{n-3,3}}&\cong D(\be_{n-3},(2,1),\pm)|D(\be_{n-3},(2,1),\mp).
\end{align*}
Since $\End_F(D(\be_n,\delta))$ is semisimple by Lemma \ref{L22}, it follows from Lemma \ref{L16} that $D_{1^2}\subseteq\End_F(D(\be_n,\delta))$.

For $\widetilde \A_n$ the proof is similar (it uses the restriction to the corresponding subgroups of $\widetilde \A_n$).
\end{proof}

\section{Tensor products}\label{s6}

In this section we will consider tensor products with special classes of modules. In order to check if tensor products are irreducible we will at times use the following lemmas.

\begin{lemma}\label{L21a}
Let $D$ be an irreducible $F\widetilde\s_n$-module and $\mu\in\RPar_p(n)$. If $D\otimes D(\la,\de)$ is irreducible then
\begin{align*}
\dim\Hom_{\widetilde\s_n}(\End_F(D),\Hom_F(D(\mu),D(\mu,\delta))&\leq 1+a(\mu).
\end{align*}
Similarly if $E$ is an irreducible $F\widetilde \A_n$-module, $\mu\in\RPar_p(n)$ and $E\otimes D(\la,\de')$ is irreducible then
\begin{align*}
\dim\Hom_{\widetilde \A_n}(\End_F(E),\Hom_F(E(\mu),E(\mu,\delta'))&\leq 2-a(\mu).
\end{align*}
\end{lemma}

\begin{proof}
Similar to \cite[Lemma 3.4]{bk2}.
\end{proof}

\begin{lemma}\label{L50}
Let $\la\in\Par_p(n)$ and $\mu\in\RPar_p(n)$. If $D(\mu)$ is of type Q and
\begin{align*}
\dim\Hom_{\widetilde\s_n}(\End_F(D^\la),\Hom_F(D(\mu),D(\mu,\pm))=2
\end{align*}
then
\begin{enumerate}[-]
\item
if $D^\la\otimes D(\mu)$ has a composition factor of type M then $D^\la\otimes D(\mu,\pm)$ is irreducible,

\item
if $D^\la\otimes D(\mu)$ has a composition factor of type Q then $D^\la\otimes D(\mu,\pm)$ is not irreducible.
\end{enumerate}

Similarly if $\la\in\Par_p(n)\setminus\Parinv_p(n)$, $D(\mu)$ is of type M and
\begin{align*}
\dim\Hom_{\widetilde \A_n}(\End_F(E^\la),\Hom_F(E(\mu),E(\mu,\pm))=2
\end{align*}
then
\begin{enumerate}[-]
\item
if $D^\la\otimes D(\mu)$ has a composition factor of type M then $E^\la\otimes E(\mu,\pm)$ is not irreducible,

\item
if $D^\la\otimes D(\mu)$ has a composition factor of type Q then $E^\la\otimes E(\mu,\pm)$ is irreducible.
\end{enumerate}
\end{lemma}

\begin{proof}
We will prove the lemma only for $\widetilde\s_n$, the proof for $\widetilde \A_n$ being similar (using conjugation by elements in $\widetilde\s_n\setminus\widetilde \A_n$ instead of tensoring with $\sgn$).

As $D(\mu)=D(\mu,+)\oplus D(\mu,-)$ and $D(\mu,+)\cong D(\mu,-)\otimes\sgn$, 
\begin{align*}
\dim\End_{\widetilde\s_n}(D^\la\otimes D(\mu))&=\dim\Hom_{\widetilde\s_n}(\End_F(D^\la),\End_F(D(\mu))\\
&=2\dim\Hom_{\widetilde\s_n}(\End_F(D^\la),\Hom_F(D(\mu),D(\mu,\pm))\\
&=4.
\end{align*}
Let $D(\nu)\subseteq D^\la\otimes D(\mu)$. Assume first that $D(\nu)$ is of type M. Then $D(\nu)=D(\nu,0)\cong D(\nu,0)\otimes\sgn$. From $D(\mu,+)\cong D(\mu,-)\otimes\sgn$ it follows that $D(\nu)^{\oplus 2}\subseteq D^\la\otimes D(\mu)$. Since $D^\la\otimes D(\mu)$ is self-dual and so it has isomorphic head and socle, it follows that $D^\la\otimes D(\mu)\cong D(\nu)^{\oplus 2}$. In particular as module $D^\la\otimes D(\mu)$ has exactly two composition factors and so $D^\la\otimes D(\mu,\pm)$ is irreducible.

Assume now that $D(\nu)$ is of type Q. Then $D^\la\otimes D(\mu)\not\cong D(\nu)$. In particular as module $D^\la\otimes D(\mu)$ has more than two composition factors. Since $D^\la\otimes D(\mu,+)\cong (D^\la\otimes D(\mu,-))\otimes\sgn$, it then follows that $D^\la\otimes D(\mu,\pm)$ is not irreducible in this case.
\end{proof}

\subsection{Tensor products with natural modules}\label{snat}

\begin{lemma}\label{L8}
Let $n\geq 4$, $G=\widetilde \s_n$ or $\widetilde \A_n$, $\la\in\RPar_p(n)$ and $V$ be a simple spin $G$-module indexed by $\la$. If $V\otimes D^{(n-1,1)}\da_G$ is simple then, as supermodule,
\[[D(\la)\otimes M_1:D(\la)]=\left\{\begin{array}{ll}
1,&n\not\equiv 0\Md p,\\
2,&n\equiv 0\Md p.
\end{array}\right.\]
\end{lemma}

\begin{proof}
Since $n\geq 4$ we have that $D^{(n-1,1)}\da_G$ has dimension greater than 1. Let $V'$ be any simple spin $G$-module indexed by $\la$. Then $V'\otimes D^{(n-1,1)}\da_G$ is simple (by either tensoring with $\sgn$ if $G=\widetilde\s_n$ or conjugating with $\si\in\widetilde\s_n\setminus\widetilde \A_n$ if $G=\widetilde\A_n$) and so $V$ is not a composition factor of $V'\otimes D^{(n-1,1)}\da_G$. So $[D(\la)\otimes M_1:D(\la)]=[M_1:D_0]$ and then the lemma holds by Lemma \ref{M1}.
\end{proof}

\begin{lemma}\label{L11}
Let $G=\widetilde \s_n$ or $\widetilde \A_n$ and $\la\in\RPar_p(n)$.
\begin{enumerate}[-]
\item If $G=\widetilde\s_n$ and $D(\la)$ is of type M then $D(\la,0)\otimes D^{(n-1,1)}$ is irreducible if and only if as supermodule $D(\la)\otimes D^{(n-1,1)}$ is irreducible of type M.

\item If $G=\widetilde\s_n$ and $D(\la)$ is of type Q then $D(\la,\pm)\otimes D^{(n-1,1)}$ is irreducible if and only if as supermodule $D(\la)\otimes D^{(n-1,1)}$ is irreducible of type Q or it has exactly two composition factors both of type M.

\item If $G=\widetilde \A_n$ then $E(\la,0)\otimes E^{(n-1,1)}$ or $E(\la,\pm)\otimes E^{(n-1,1)}$ is irreducible if and only if as supermodule $D(\la)\otimes D^{(n-1,1)}$ is irreducible.
\end{enumerate}
\end{lemma}

\begin{proof}
This holds by comparing the number of composition factors of $D(\la)\da_G$ and of $(D(\la)\otimes D^{(n-1,1)})\da_G$.
\end{proof}

\begin{theor}\label{T2}
Let $n\geq 4$, $G=\widetilde \s_n$ or $\widetilde \A_n$, $\la\in\RPar_p(n)$ and $V$ be a simple spin $G$-module indexed by $\la$. If $V\otimes D^{(n-1,1)}\da_G$ is simple then $n\not\equiv 0\Md p$ and $\la\in JS(0)$.

In this case, if $\nu=(\la\setminus A)\cup B$ where $A$ is the bottom removable node of $\la$ and $B$ is the top addable node of $\la$,
\begin{enumerate}[-]
\item
if $D(\la)$ is of type M then $D(\la,0)\otimes D^{(n-1,1)}$ is not irreducible, while $E(\la,\pm)\otimes E^{(n-1,1)}\cong E(\nu,0)$,

\item
if $D(\la)$ is of type Q then $D(\la,\pm)\otimes D^{(n-1,1)}\cong D(\nu,0)$, while $E(\la,0)\otimes E^{(n-1,1)}$ is not irreducible.
\end{enumerate}
\end{theor}

\begin{proof}
Let $c:=1$ if $D(\la)$ is of type M or $c:=2$ if $D(\la)$ is of type Q. Assume that $V\otimes D^{(n-1,1)}\da_G$ is simple. We will use Lemmas \ref{Lemma39s} and \ref{Lemma40s} without further notice.

{\bf Case 1.} $n\equiv 0\Md p$. From Lemma \ref{L10} we have that $\eps_0(\la)+\phi_0(\la)$ is odd. So by Lemmas \ref{L9} and \ref{L8} we have that $\la\in JS(i)$ and $\phi_i(\la)=0$ for some $i\geq 1$. Note that
\begin{align*}
D(\la)\otimes M_1&\cong (f_ie_iD(\la))^{\oplus 2}\oplus\sum_{j\geq 1:j\not=i}(f_je_iD(\la))^{\oplus 2}\oplus (f_0e_iD(\la))^{\oplus c}\\
&\cong D(\la)^{\oplus 2}\oplus\sum_{j\geq 1:j\not=i}(f_jD(\widetilde{e}_i\la))^{\oplus 2}\oplus (f_0D(\widetilde{e}_i\la))^{\oplus c}.
\end{align*}
It then follows from Lemma \ref{M1} and considering block decomposition that
\[D(\la)\otimes D^{(n-1,1)}\cong \sum_{j\geq 1:j\not=i}(f_jD(\widetilde{e}_i\la))^{\oplus 2}\oplus (f_0D(\widetilde{e}_i\la))^{\oplus c}.\]

By Lemma \ref{L11} it follows that if $D(\la)$ is of type Q then it needs to have exactly two composition factors of type M, while if $D(\la)$ is of type M then $D(\la)\otimes D_1$ is irreducible as supermodule. In either case $\phi_i(\widetilde{e}_i\la)=\phi_0(\widetilde{e}_i\la)=1$ and $\phi_j(\widetilde{e}_i\la)=0$ for $j\not=0,i$.

In particular $D(\la)\otimes M_1\cong D(\la)^{\oplus 2}\oplus D(\widetilde{f}_0\widetilde{e}_i\la)^{\oplus c}$. Notice also that from Lemma \ref{L10} either $\phi_0(\la)=3$ and $\phi_k(\la)=0$ else or there exists $j\not=0,i$ such that $\phi_0(\la)=\phi_j(\la)=1$ and $\phi_k(\la)=0$ else.

{\bf Case 1.1.} $\phi_0(\la)=3$ and $\phi_j(\la)=0$ else.

From Lemma \ref{L051218_3}
\[D(\widetilde{f}_0\widetilde{e}_i\la)^{\oplus c}\cong\Ind_0\Res_i D(\la)\cong \Res_i\Ind_0 D(\la)\cong \Res_i f_0 D(\la)\]
and
\[0=\Ind_0\Res_j D(\la)\cong \Res_j\Ind_0 D(\la)\cong \Res_j f_0 D(\la)\]
for $j\not=0,i$. Since $c\leq 2<[f_0 D(\la):D(\widetilde f_0\la)]=\phi_0(\la)=3$, it follows that $\widetilde{f}_0$ has only normal nodes of residue 0 and then $\widetilde{f}_0\la\in JS(0)$, since $\eps_0\la=0$. Since $\phi_0(\widetilde{f}_0\la)=2$ we have from Lemma \ref{L12} that $n+1\equiv 0\Md p$, leading to a contradiction.

{\bf Case 1.2.} There exists $j\not=0,i$ such that $\phi_0(\la)=\phi_j(\la)=1$ and $\phi_k(\la)=0$ else. In this case by Lemma \ref{L051218_3}
\[\Res_iD(\widetilde f_j\la)^{\oplus c}\cong \Res_i\Ind_j D(\la)\cong \Ind_j\Res_i D(\la)\cong \Ind_j D(\widetilde e_i\la)^{\oplus c}=0\]
and
\[\Res_kD(\widetilde f_j\la)^{\oplus c}\cong \Res_k\Ind_j D(\la)\cong \Ind_j\Res_k D(\la)=0\]
for $k\not=i,j$. So all normal nodes of $\widetilde f_j\la$ have residue $j$. Since $\epsilon_j(\la)=0$ we then have that $\widetilde{f}_j\la\in JS(j)$, which by Lemma \ref{L13} contradicts $n\equiv 0\Md p$.

{\bf Case 2.} $n\not\equiv 0\Md p$. In this case $\la\in JS(0)$ and $\phi_0(\la)=0$ by Lemmas \ref{L9} and \ref{L8}. From Lemma \ref{L12} this is equivalent to $\la\in JS(0)$ since $n\not\equiv 0\Md p$. Notice that
\[D(\la)\otimes M_1\cong f_0e_0D(\la)\oplus\sum_{j\geq 1}(f_je_0D(\la))^{\oplus c}\cong D(\la)\oplus\sum_{j\geq 1}(f_jD(\widetilde{e}_0\la))^{\oplus c}.\]
From Lemma \ref{M1} it follows that
\[D(\la)\otimes D^{(n-1,1)}\cong\sum_{j\geq 1}(f_jD(\widetilde{e}_0\la))^c.\]
From \cite[Lemma 3.8]{p} $\widetilde{e}_0\la\in JS(1)$. Further $\phi_0(\widetilde{e}_0\la)=1$. So from Lemma \ref{L10} there exists $j\geq 1$ with $\phi_0(\widetilde{e}_0\la),\phi_j(\widetilde{e}_0\la)=1$ and $\phi_k(\widetilde{e}_0\la)=0$ for $k\not=0,j$. If $D(\la)$ is of type M then $D(\la)\otimes D^{(n-1,1)}\cong D(\widetilde{f}_j\widetilde{e}_0\la)$ and $D(\widetilde{f}_j\widetilde{e}_0\la)$ is of type Q. If $D(\la)$ is of type Q then $D(\la)\otimes D^{(n-1,1)}\cong D(\widetilde{f}_j\widetilde{e}_0\la)^{\oplus 2}$ and $D(\widetilde{f}_j\widetilde{e}_0\la)$ is of type M.

Note that $\widetilde{e}_0\la=\la\setminus A$, since $\la$ is JS(0) and the bottom addable node is always normal. Then $A$ is the bottom addable node of $\widetilde{e}_0\la$ and it is the conormal node of $\widetilde e_0\la$ of residue 0. Since $n\geq 4$ and $\la$ is JS(0) we have that $h(\la)\geq 2$. If $B$ is the top addable node of $\la$ then it is also the top addable node of $\widetilde e_0\la$. Since the top addable node is always conormal, it follows that $\widetilde{f}_j\widetilde{e}_0\la=(\la\setminus A)\cup B$. The theorem then follows from Lemma \ref{L11}.
\end{proof}

\subsection{Tensor products of basic spin and hooks}

\begin{theor}\label{THB}
Let $p\geq 3$. Let $G=\widetilde\s_n$ or $\widetilde \A_n$. Assume that $V$ is indexed by an element of $\h_p(n)$ and that $W$ is basic spin. If $V$ and $W$ are not 1-dimensional and $V\otimes W$ is irreducible, then one of the following holds:
\begin{enumerate}[-]
\item
$p\not=5$, $G=\widetilde \A_5$, $V\cong E^{(3,1^2)}_\pm$ and $W\cong E(\be_5,\pm)$, in which case two of the corresponding tensor products are irreducible and isomorphic to $E((4,1),0)$, while the other two tensor products are not irreducible.

\item
$p=3$, $G=\widetilde \A_6$, $V\cong E^{(4,1^2)}_\pm$ and $W\cong E((3,2,1),\pm)$, in which case two of the corresponding tensor products are irreducible and isomorphic to $E((4,2),\pm)$, while the other two tensor products are not irreducible.
\end{enumerate}
In the exceptional cases, if $\chi_V$ and $\chi_W$ are the characters of $V$ and $W$, we have that $V\otimes W$ is irreducible if and only if $(\chi_V\chi_W)\widetilde{(1,2,3,4,5)}=1$.
\end{theor}

\begin{proof}
For $n\leq 12$ the theorem can be proved by looking at decomposition matrices. So we may assume that $n>12$.

We may assume that $k<n/2$. From \cite[Theorem 9.3]{s},
\[[S_{1^k}\otimes S((n))]=[S((n))]+\sum_{1\leq j\leq k}d[S((n-j,j))],\]
where $d=1$ if $n$ is odd and $d=2$ if $n$ is even.

In particular, using Lemmas \ref{LBS} and \ref{LH} and induction on $k$ if $n\equiv 0\Md p$,
\begin{enumerate}[-]
\item
if $n\not\equiv 0\Md p$ then 
\[[\overline{D}_{k}\otimes D(\be_n)]=[D(\be_n)]+\sum_{1\leq j\leq k}d[S((n-j,j))],\]

\item
if $n\equiv 0\Md p$ and $k$ is even then
\[[\overline{D}_{k}\otimes D(\be_n)]=[D(\be_n)]+\sum_{1\leq j\leq k/2}[S((n-2j,2j))],\]

\item
if $n\equiv 0\Md p$ and $k$ is odd then
\[[\overline{D}_{k}\otimes D(\be_n)]=\sum_{0\leq j\leq (k-1)/2}[S((n-2j-1,2j+1))].\]
\end{enumerate}

When $n\equiv 0\Md p$ then $D(\be_n)$ is a composition factor of $S((n-1,1))$ by \cite[Table IV]{Wales}. So $D(\be_n)$ is always a composition factor of $\overline{D}_{k}\otimes D(\be_n)$ (as supermodule). Since $\overline{D}_{k}\otimes D(\be_n,\delta)$ is irreducible if and only if $\overline{D}_{k}\otimes D(\be_n,-\delta)$ is irreducible and since $\overline{D}_{k}$ is not 1-dimensional, it follows that $\overline{D}_{k}\otimes D(\be_n,\delta)$ is not irreducible. Similarly if $k\not=(n-c)/2$ then $\overline{E}_{k}\otimes E(\be_n,\delta')$ is not irreducible.

So assume now that $k=(n-c)/2$. Note that in this case either $n$ is odd with $n\not\equiv 0\Md p$ or $n$ is even with $n\equiv 0\Md p$, so $D(\be_n)$ is of type M and then $\delta'=\pm$. By Lemmas \ref{LBS} and \ref{LH} we have
\[\dim((\overline{E}_{k})_\pm\otimes E(\be_n,\pm))\hspace{-0.5pt}=\hspace{-0.5pt}\frac{1}{2}\binom{n-c}{(n-c)/2}2^{(n-c-2)/2}\hspace{-0.5pt}=\hspace{-0.5pt}2^{(n-c-4)/2}\binom{n-c}{(n-c)/2}\hspace{-0.5pt}.\]
Let $d_j$ be the dimension of any simple spin module of $\widetilde \A_n$ indexed by $(n-j,j)$ in characteristic 0. For $1\leq j\leq k$ we have
\[d_j=\frac{1}{2}\dim S((n-j,j))=2^{(n-c-2)/2}\frac{n-2j}{n-j}\binom{n-1}{j}.\]
Note that if $(\overline{E}_k)_\pm\otimes E(\be_n,\pm)$ is irreducible then it is not isomorphic to $E(\be_n,\pm)$ (since $(\overline{E}_k)_\pm$ is not 1-dimensional). In order to prove that $(\overline{E}_k)_\pm\otimes E(\be_n,\pm)$ is not irreducible it is then enough to prove that $\dim((\overline{E}_{k})_\pm\otimes E(\be_n,\pm))>d_j$ for any $1\leq j\leq k$. If $n$ is even note that
\[2\binom{n-2}{(n-2)/2}=
\frac{n}{n-1}\binom{n-1}{(n-2)/2}>\binom{n-1}{(n-2)/2}.\]
So it is enough to prove that
\[\binom{n-1}{\lfloor(n-1)/2\rfloor}=\binom{n-1}{(n-c)/2}>2^c\frac{n-2j}{n-j}\binom{n-1}{j}\]
for any $1\leq j\leq k=\lfloor(n-1)/2\rfloor$.

If $j>3/7n$ then $4(n-2j)/(n-j)<1$ and so the above inequality clearly holds. So we may assume that $j\leq 3/7n$. In this case it is enough to prove that
\[\frac{\binom{n-1}{\lfloor(n-1)/2\rfloor}}{\binom{n-1}{j}}=\prod_{i=j+1}^{\lfloor (n-1)/2\rfloor}\frac{n-i}{i}>4.\]
It is enough to prove this for $j=\lfloor 3/7n\rfloor$. If $n\geq 152$ then
\[\prod_{i=\lfloor 3/7n\rfloor+1}^{\lfloor (n-1)/2\rfloor}\frac{n-i}{i}\geq \prod_{a=1}^{\lfloor (n-1)/2\rfloor-\lfloor 3/7n\rfloor}\frac{4/7n-a}{3/7n+a}\geq \prod_{a=1}^9\frac{4/7n-a}{3/7n+a}>4.\]
 Using the above formulas, it can be checked that for $n\leq 151$ and $n\leq 3/7n$ we have $\dim((\overline{E}_{n,k})_\pm\otimes E(\be_n,\pm))>d_j$, unless possibly if $n\leq 20$ is even with $n\equiv 0\Md p$. In these cases notice that it is enough to prove that $\dim((\overline{E}_{n,k})_\pm\otimes E(\be_n,\pm))>d_j$ for $j$ odd if $n\equiv 0\Md 4$ or for $j$ even if $n\equiv 2\Md 4$, which again can be checked using the above formulas since we are assuming $n>12$.
\end{proof}

\subsection{Tensor products of basic spin and two rows partitions}

\begin{theor}\label{T011119}
Let $p\geq 3$ and $G\in\{\widetilde\s_n,\widetilde \A_n\}$. Let $V$ be a simple non-spin module indexed by $\la\in\Par_p(n)$ with $\min\{h(\la),h(\la^\Mull)\}=2$ and $W$ be basic spin. If $V\otimes W$ is irreducible then $\la$ is JS and $n\not\equiv 0,\pm 2\Md p$. Further in this case:
\begin{enumerate}[-]
\item if $G=\widetilde\s_n$ and $n$ is even then $V\otimes W\cong D(\mu,0)$ is irreducible with $\mu=\be_{\la_1}+\be_{\la_2}$ if $\la_1\not=\la_2$ or $\mu=\be_{n/2+1}\cup\be_{n/2-1}$ if $\la_1=\la_2$,

\item if $G=\widetilde\s_n$ and $n$ is odd then $V\otimes W$ is not irreducible,

\item if $G=\widetilde \A_n$ and $n$ is even then $V\otimes W$ is not irreducible,

\item if $G=\widetilde \A_n$ and $n$ is odd then $V\otimes W\cong E(\mu,0)$ is irreducible with $\mu=\be_{\la_1}+\be_{\la_2}$ if $\la_1\not=\la_2+p-2$ or $\mu=\be_{\la_1}\cup\be_{\la_2}$ if $\la=\la_2+p-2$.
\end{enumerate}
\end{theor}

\begin{proof}
For $n\leq 9$ the theorem can be proved looking at decomposition matrices. So assume that $n\geq 10$. Note that $V\cong D^\la\da_G$ by \cite[Lemma 1.8]{ks2}. Further we may assume that $h(\la)=2$. In view of Theorem \ref{T2} we may also assume that $\la_2\geq 2$ (since $(n-1,1)$ is JS if and only if $n\equiv 0\Md p$). Note that in this case $\la\not\in\h_p(n)$ (the case $p=3$ and $\la=(n)^\Mull$ is excluded by assumption). It is easy to see that $\la$ is JS if and only if $\la_1=\la_2$ or $\la_1-\la_2\equiv -2\Md p$.

Let $W'=D(\be_n)$ or $E(\be_n)$ (depending on $G$). Further from Lemmas \ref{L22} and \ref{L33} we have that $\overline{D}_0\oplus \overline{D}_2\subseteq\End_F(W)$ and $\overline{D}_0\oplus \overline{D}_1\oplus \overline{D}_2\oplus\overline{D}_3\subseteq\Hom_F(W',W)$.

If $\la$ is not JS then we have that $\overline{D}_0\oplus \overline{D}_2$ or $\overline{D}_0\oplus\overline{D}_1\oplus \overline{D}_3$ is contained in $\End_F(V)$ from Lemmas \ref{L39}, \ref{L41} and \ref{L42}. It follows that
\[\dim\Hom_G(\End_F(V),\End_F(W))\geq 2\]
or
\[\dim\Hom_G(\End_F(V),\Hom_F(W',W))\geq 3.\]
So $V\otimes W$ is not irreducible (in the second case by Lemma \ref{L21a}).

So assume now that $\la$ is JS. In view of Lemmas \ref{LH} and \ref{L3} we have that $\overline{D}_0\oplus \overline{D}_2$ or $\overline{D}_0\oplus\overline{D}_3$ is contained in $\End_F(V)$. So
\[\dim\Hom_G(\End_F(V),\Hom_F(W',W))\geq 2.\]
So by Lemmas \ref{L21a} and \ref{L50} if $G=\widetilde\s_n$ then $D^\la\otimes D(\be_n,\de)$ is irreducible if and only if $D(\be_n)$ is of type Q, $D^\la\otimes D(\be_n)$ has a composition factor of type M and
\[\dim\Hom_{\widetilde\s_n}(\End_F(D^\la),\Hom_F(D(\be_n),D(\be_n,\de)))=2.\]
Similarly if $G=\widetilde \A_n$ then $E^\la\otimes E(\be_n,\de')$ is irreducible if and only if $D(\be_n)$ is of type M, $D^\la\otimes D(\be_n)$ has a composition factor of type Q and
\[\dim\Hom_{\widetilde \A_n}(\End_F(E^\la),\Hom_F(E(\be_n),E(\be_n,\de')))=2.\]
On the other hand if $D^\la\otimes D(\be_n)$ has a composition factor of the same type as $D(\be_n)$ then $V\otimes W$ is not irreducible.

Note that $[D^\la]=[S^\la]+\sum_{j<\la_2}d_j[S^{(n-j,j)}]$ for some $d_j\in\Z$ with $d_j\not=0$ only if $D^\la$ and $D^{(n-j,j)}$ are in the same block by \cite[Corollary 12.2]{JamesBook}. Further $[D(\be_n)]=d[S((n))]$ for some $d>0$ by Lemma \ref{LBS}.

If $\la_1=\la_2$ then $D^\la$ and $D^{(\la_1+1,\la_1-1)}$ are in different blocks and so by \cite[Theorem 9.3]{s} 
\[[D^\la\otimes D(\be_n)]=c[S((\la_1+1,\la_1-1))]+\sum_{j<\la_1-1}c_j[S((n-j,j))]\]
for some $c,c_j\in\Q$ with $c>0$. In this case let $\nu:=(n/2+1,n/2-1)=(\la_1+1,\la_2-1)$.

If $\la_1>\la_2$ then
\[[D^\la\otimes D(\be_n)]=c[S(\la)]+\sum_{j<\la_2}c_j[S((n-j,j))]\]
for some $c,c_j\in\Q$ with $c>0$. In this case let $\nu:=\la$. Note that $\la_1\geq \la_2+p-2$. Further if $p=3$ then by assumption $\la_1-\la_2\geq 4$.

From \cite[Theorems 1.2, 1.3]{m4} there exists a composition factor $D(\mu)$ of $S(\nu)$ which is not a composition factor of $S((\pi_1,\pi_2))$ for $(\pi_1,\pi_2)\in\RP_0(n)$ with $\pi_1>\nu_1$. Then $D(\mu)$ is a composition factor of $D^\la\otimes D(\be_n)$.

{\bf Case 1:} $n\equiv 0\Md p$. In this case any composition factor of $S((n-j,j))$ with $j< n/2$ is in the same block as $D(\be_n)$, so they have the same type and then $V\otimes W$ is not irreducible in this case.

{\bf Case 2:} $n\equiv \pm 2\Md p$. In this case it can be checked that if $\la_1=\la_2$ then one part of $(n/2+1,n/2-1)$ is divisible by $p$, while if $\la_1>\la_2$ then one part of $\la$ is divisible by $p$ (since in this case $\la_1-\la_2\equiv p-2\Md p$). So $S(\nu)$ is in the same block as $S((n))$ and then again $V\otimes W$ is not irreducible.

{\bf Case 3:} $n\not\equiv 0,\pm 2\Md p$. In this case $p\geq 5$ and so Lemmas \ref{LH}, \ref{L20}, \ref{L44} and \ref{L22} 
\begin{align*}
\dim\Hom_{\widetilde\s_n}(\End_F(D^\la),\Hom_F(D(\be_n),D(\be_n,\de)))&=2,\\
\dim\Hom_{\widetilde \A_n}(\End_F(E^\la),\Hom_F(E(\be_n),E(\be_n,\de')))&=2.
\end{align*}
Further if $\la_1=\la_2$ then $p\nmid n/2\pm 1$, while if $\la_1>\la_2$ then $p\nmid \la_1,\la_2$. Since $n\not\equiv 0\Md p$ it can then be easily checked that $D(\mu)$ and $D(\be_n)$ are of different type. So $D^\la\otimes D(\be_n,\de)$ is irreducible if and only if $n$ is even and in this case $D^\la\otimes D(\be_n,\de)\cong D(\mu,0)$. Similarly $E^\la\otimes E(\be_n,\de')$ is irreducible if and only if $n$ is even and in this case $E^\la\otimes E(\be_n,\de')\cong E(\mu,0)$. The theorem then follows from \cite[Theorems 1.2, 1.3]{m4} to identify $\mu$.
\end{proof}

\subsection{Tensor products of basic spin and three rows partitions}

\begin{theor}\label{tbs3r}
Let $p=3$ and $G\in\{\widetilde\s_n,\widetilde \A_n\}$. Let $\la\in\Par_p(n)\setminus\h_3(p)$ with $\min\{h(\la),h(\la^\Mull)\}=3$, $V$ be a simple non-spin module indexed by $\la$ and $W$ be basic spin. Then $V\otimes W$ is not irreducible.
\end{theor}

\begin{proof}
We may assume that $h(\la)=3$. Since $\la\not\in\h_3(n)$ we then have that $\la=(\la_1,\la_2,\la_3)$ with $\la_1\geq\la_2+2$, $\la_2\geq\la_3+2$ and $\la_3\geq 1$. In particular $n\geq 9$. Further it is easy to check that $\la\not=\la^\Mull$, so $V\cong D^\la\da_G$.

If $W'=D(\be_n)$ if $G=\widetilde\s_n$ or $W'=D(\be_n)$ if $G=\widetilde \A_n$ then
\[\overline{D}_0\oplus \overline{D}_1\oplus\overline{D}_2\oplus\overline{D}_3\subseteq\Hom_F(W',W)\]
by Lemma \ref{L22}. If $\la$ is not JS then
\[\overline{D}_0\oplus \overline{D}_1\oplus\overline{D}_k\subseteq\End_F(V)\]
with $2\leq k\leq 3$ from Lemmas \ref{L39} and \ref{L56}. So in this case $V\otimes W$ is not irreducible by Lemma \ref{L21a}.

So we may assume that $\la$ is JS. So $\la_1-\la_2,\la_2-\la_3\equiv 1\Md 3$ and then we have $\la_1\geq\la_2+4$ and $\la_2\geq\la_3+4$. From Lemmas \ref{LH} and \ref{L3} we have that $\overline{D}_2$ or $\overline{D}_3$ is contained in $\End_F(V)$. Since we always have $\overline{D}_0\subseteq\End_F(V)$ from Lemmas \ref{L21a} and \ref{L50} to prove that $V\otimes W$ is not irreducible it is enough to prove that $D^\la\otimes D(\be_n)$ has a composition factor of the same type as $D(\be_n)$. Note that by Lemma \ref{LBS} and \cite[Theorem 9.3]{s} we have that
\[[D^\la\otimes D(\be_n)]=c[S(\la)]+\sum_{\mu\in\RP_0(n):\mu\rhd\la}c_\mu[S(\mu)]\]
with $c>0$. From Lemmas \ref{L54} and \ref{L55} we then have that if $\nu=\la^R=\be_{\la_1}+\be_{\la_2}+\be_{\la_3}$ then $D(\nu)$ is a composition factor of $D^\la\otimes D(\be_n)$ (since $\la_1\geq\la_2+4$ and $\la_2\geq\la_3+4$). From $\la_1-\la_2,\la_2-\la_3\equiv 1\Md 3$ we have that $n\equiv 0\Md 3$ and one of $\la_1$, $\la_2$ and $\la_3$ is divisible by 3. In particular $S(\la)$ and $S((n))$ are in the same block and so $D(\nu)$ and $D(\be_n)$ are of the same type. So $V\otimes W$ is not irreducible.
\end{proof}

\subsection{Tensor products of basic and second basic spin}\label{sbssbs}

The following result will only be needed for $p=3$, but the proof in the general case is the same as the proof for the case $p=3$ only, so we present it here in the general version. By definition, for $G=\widetilde\s_n$ or $\widetilde \A_n$, second basic spin modules of $G$ are composition factors of the reduction modulo $p$ of $S((n-1,1))\da_G$ which are not basic spin modules.

\begin{theor}\label{L30}
Let $p\geq 3$, $n\geq 6$ and $G=\widetilde\s_n$ or $\widetilde \A_n$. Assume that $V$ is second basic spin and that $W$ is basic spin. Then $V\otimes W$ is not irreducible.
\end{theor}

\begin{proof}
From Lemma \ref{LBS} and \cite[Table IV]{Wales} we have that any composition factor of $V\otimes W$ is a composition factor of the reduction modulo $p$ of $S((n-1,1))\otimes S((n))$. 
So from \cite[Theorem 9.3]{s}, any composition factor of $V\otimes W$ is a composition factor of a Specht module of the form $S^{(n-k,1^k)}$ with $0\leq k\leq n-1$ or $S^{(n-k,2,1^{k-2})}$ with $2\leq k\leq n-2$. Notice also that by \cite[Tables III and IV]{Wales}
\[\dim V\otimes W\geq 2^{n-4}(n-4).\]
It can be computed that
\[\dim S^{(n-k,1^k)}=\binom{n-1}{k},\hspace{24pt}\dim S^{(n-k,2,1^{k-2})}=\binom{n}{k}\frac{(n-k-1)(k-1)}{n-1}.\]
Since
\[\frac{(n-k-1)(k-1)}{n-1}\leq \frac{(n-2)^2}{4(n-1)}\leq\frac{n-2}{4},\]
it is enough to prove that
\[\binom{n}{\lfloor n/2\rfloor}\frac{n-2}{4}<2^{n-4}(n-4),\]
that is that
\[\frac{\binom{n}{\lfloor n/2\rfloor}(n-2)}{2^{n-2}(n-4)}<1.\]
Notice that $(n-2)/(n-4)$ is decreasing as is $\binom{n}{\lfloor n/2\rfloor}/2^{n-2}$, since
\begin{align*}
\binom{n}{\lfloor n/2\rfloor}/2^{n-2}&=\binom{n-1}{\lfloor n/2\rfloor}/2^{n-2}+\binom{n-1}{\lfloor n/2\rfloor-1}/2^{n-2}\\
&\leq \binom{n-1}{\lfloor (n-1)/2\rfloor}/2^{(n-1)-2}.
\end{align*}
Since $\binom{15}{7}\cdot 13/(2^{13}\cdot 11)<1$, the lemma holds for $n\geq 15$.

For $6\leq n\leq 14$ the lemma can be checked by looking at decomposition matrices for $\s_n$ and $\A_n$ to find the dimension of composition factors of the reduction modulo $p$ of the modules $S^{(n-k,1^k)}$ and $S^{(n-k,2,1^{k-2})}$ as well as exact formulas for $\dim V\otimes W$ coming \cite[Tables III and IV]{Wales}.
\end{proof}

\section{Proofs of Theorems \ref{TS} and \ref{TA}}\label{s7}

By \cite{bk,bk2,m2,m3,z1} we may assume that $W$ is a spin representation. Further we may assume that neither $V$ nor $W$ is 1-dimensional. For $n\leq 12$ the theorems can be proved using GAP \cite{gap} 
or looking at decomposition matrices (and using Lemma \ref{L54} to identify modular spin representations). So assume that $n\geq 13$. It can then be checked (using Lemma \ref{L20} and \cite[Lemma 2.2]{bkz} to help check some cases) that if $\al$ is one of $(n-3,3)$, $(n-3,1^3)$, $(n-5,1^5)$ or $(n-5,3,1^2)$ (the last one only for $p=3$) and $\al\in\Par_p(n)$, then $\al>\al^\Mull$. Let $G\in\{\widetilde\s_n,\widetilde \A_n\}$ depending on which theorem we are considering. Since $n\geq 13$ we have from \cite[Lemma 1.8]{ks2} that $(n-k,k)\not=(n-k,k)^\Mull$ for any $0\leq k\leq n/2$. Further the modules $\overline{E}_k$ for $0\leq k\leq 5$ are defined and they are simple and  pairwise non-isomorphic.

{\bf Case 1:} $p\geq 5$ and neither $V$ nor $W$ is a basic spin representation or a natural representation (a non-spin representation indexed by $(n-1,1)$ or $(n-1,1)^\Mull$).

Parts of this case could be proved using results from \cite{bk5,kt}. However the cases where $V$ is a non-spin representation which is indexed by a 2-rows JS partition (or its Mullineux dual) or if $G=\widetilde \A_n$ and $V$ is non spin and indexed by a Mullineux-fixed partition are not covered by results from \cite{bk5,kt}.

By Lemmas \ref{LM3S} and \ref{L4} there exist $\phi_3\in\Hom_G(M_3,\End_F(W))$ and $\phi_{1^3}\in\Hom_G(M_{1^3},\End_F(W))$ which do not vanish on $S_3$ and $S_{1^3}$ respectively. Further from Lemmas \ref{LM3}, \ref{LM3S} and \ref{L3} we have that there exists $\psi_3\in\Hom_G(M_3,\End_F(V))$ or $\psi_{1^3}\in\Hom_G(M_{1^3},\End_F(V))$ which does not vanish on $S_3$ or $S_{1^3}$. Since $M_0=S_0$ is the trivial module, so that there also always exist non-zero $\phi_0\in\Hom_G(M_0,\End_F(W))$ and $\psi_0\in\Hom_G(M_0,\End_F(V))$, we then have from Lemma \ref{l15} that
\[\dim\End_G(V\otimes W)=\dim\Hom_G(\End_F(V),\End_F(W))\geq 2\]
and so $V\otimes W$ is not irreducible.

{\bf Case 2:} $p=3$, neither $V$ nor $W$ is a basic spin representation or a natural representation and either $n\equiv 2\Md 3$ and $V$ is not a non-spin representation indexed by $(n-2,2)$ or $(n-2,2)^\Mull$ or $n\not\equiv 2\Md 3$.

This case holds similarly to the previous case, using Lemmas \ref{LM3}, \ref{LM3S}, \ref{L52}, \ref{L51} and \ref{L7a} (so using $M_{3,1^2}$ instead of $M_{1^3}$).

{\bf Case 3:} $p=3$, $n\equiv 2\Md 3$ and $V$ is a non-spin representation indexed by $(n-2,2)$ or $(n-2,2)^\Mull$ and $W$ is not basic spin.

We have that $(n-2,2)\not=(n-2,2)^\Mull$. So (up to tensoring with $\sgn$) $V\cong D^{(n-2,2)}$ or $E^{(n-2,2)}$. From \cite[Corollary 3.9]{bk5} and Lemma \ref{Mk} there exists $\psi_2\in\Hom_G(M_2,\End_F(V))$ which does not vanish on $S_2$. From \cite[Lemma 3.7]{p} we have that, for $p=3$, any JS(0) partition in $\RP_3(m)$ is of the form $\be_{\mu_1}+\ldots+\be_{\mu_k}$ with $\mu_j\equiv 0\Md 3$ for $j<k$ and $\mu_k=1$ or $\mu_k\equiv 0\Md 3$. Since $n\equiv 2\Md 3$ there is then no JS(0) partition in $\RP_3(n)$. Let $\nu$ be the partition indexing $W$. We will now consider $G=\widetilde\s_n$, the case $G=\widetilde \A_n$ being similar. By Lemma \ref{L190219_2} there exists $\phi_2\in\Hom_{\widetilde\s_n}(M_2,\End_F(W))$ which does not vanish on $S_2$ or $W\cong D(\nu,\pm)$ and there exist $\phi_2',\phi_2''\in\Hom_{\widetilde\s_n}(M_2,\Hom_F(D(\la,\pm),D(\la,\mp)))$ which are linearly independent over $S_2$. In the first case we can conclude as in Case 1. In the second case we have by Lemma \ref{l15} that
\begin{align*}
&\dim\Hom_{\widetilde\s_n}(V\otimes D(\la,\pm),V\otimes D(\la,\mp))\\
&=\dim\Hom_{\widetilde\s_n}(\End_F(V),\Hom_F(D(\la,\pm),D(\la,\mp)))\\
&\geq 2.
\end{align*}
Since $D(\la,+)$ and $D(\la,-)$ have the same dimension, this contradicts $V\otimes W$ being irreducible.

{\bf Case 4:} $V$ is a natural module.

Up to tensoring with $\sgn$ we have that $V\cong D^{(n-1,1)}$ or $E^{(n-1,1)}$. The theorems then follow from Theorem \ref{T2}.

{\bf Case 5:} $V$ and $W$ are basic spin.

Let $A:=D(\be_n)$ if $G=\widetilde\s_n$ or $A:=E(\be_n)$ if $G=\widetilde \A_n$. Then by Lemma \ref{L22}
\[\dim\Hom_G(A\otimes A,V\otimes W)=\dim\Hom_G(\Hom_F(V,A),\Hom_F(A,W)\geq 5.\]
Since $\dim A\leq 2\dim V$ and $\dim V=\dim W$, it follows that $V\otimes W$ is not irreducible.

{\bf Case 6:} $W$ is basic spin and $V$ is either a non-spin representation indexed by $\la\not\in\h_p(n)$ with $h(\la),h(\la^\Mull)\geq 3+\de_{p=3}$ or a spin representation indexed by $\mu\not=\be_n$ with $\mu_1\geq 5$.

From Lemmas \ref{LH} and \ref{L20} we have that $[S_{1^3}]=[\overline{D}_3]+\de_{p\mid n}[\overline{D}_2]$ and $[S_{1^5}]=[\overline{D}_5]+\de_{p\mid n}[\overline{D}_4]$.

From Lemmas \ref{L3} and \ref{L4} there exists $0\not=\phi_3\in\Hom_G(S_{1^3},\End_F(V))$ and from Lemmas \ref{L36} and \ref{L14} there exists $0\not=\phi_5\in\Hom_G(S_{1^5},\End_F(V))$. In particular there exist $a=0<b<c\leq 5$ with $\overline{D}_k\da_G\subseteq\End_F(V)$ for $k\in\{a,b,c\}$. If again $A:=D(\be_n)$ if $G=\widetilde\s_n$ or $A:=E(\be_n)$ if $G=\widetilde \A_n$ then by Lemma \ref{L22}
\[\dim\Hom_G(V\otimes A,V\otimes W)=\dim\Hom_G(\End_F(V),\Hom_F(A,W)\geq 3\]
and so $V\otimes W$ is not irreducible by Lemma \ref{L21a}.

{\bf Case 7:} $W$ is basic spin and $V$ is a non-spin representation indexed by $\la\in\h_p(n)$.

In this case the theorems hold by Theorem \ref{THB}.

{\bf Case 8:} $W$ is basic spin and $V$ is a non-spin representation indexed by $\la$ with $h(\la),h(\la^\Mull)=2$.

This case is covered by Theorem \ref{T011119}.

{\bf Case 9:} $p=3$, $W$ is basic spin and $V$ is a non-spin representation indexed by $\la\not\in\h_3(n)$ with $h(\la),h(\la^\Mull)=3$.

In this case $V\otimes W$ is not irreducible by Theorem \ref{tbs3r}.

{\bf Case 10:} $W$ is basic spin and $V$ is a spin representation indexed by $\mu\not=\be_n$ with $\mu_1\leq 4$.

Note that in this case $p=3$ since $n\geq 13$. Since $\mu\not=\be_n$ we have that $\mu=(4,\be_{n-4})=\be_{n-1}+\be_1$. In view of Lemma \ref{L55} and \cite[Table IV]{Wales} we have that $V$ is second basic spin. So the theorems hold by Theorem \ref{L30}.

\section*{Acknowledgements}

The author thanks Alexander Kleshchev for some comments on the paper. The author also thanks the referees for comments.

\end{document}